\documentclass[11pt]{amsart}

\usepackage[margin=1.2in]{geometry}
\usepackage{amsmath,amssymb,mathtools}
\usepackage{enumitem}
\usepackage{xcolor}
\usepackage[colorlinks=true,linkcolor=blue,citecolor=blue,urlcolor=blue]{hyperref}

\numberwithin{equation}{section}

\theoremstyle{plain}
\newtheorem{theorem}{Theorem}[section]
\newtheorem*{maintheorem}{Main result}
\newtheorem{proposition}[theorem]{Proposition}
\newtheorem{lemma}[theorem]{Lemma}

\theoremstyle{definition}
\newtheorem{assumption}[theorem]{Assumption}

\theoremstyle{remark}
\newtheorem{remark}[theorem]{Remark}

\newcommand{\T}{\mathbb{T}}
\newcommand{\R}{\mathbb{R}}
\newcommand{\Z}{\mathbb{Z}}
\newcommand{\E}{\mathbb{E}}
\newcommand{\Prob}{\mathbb{P}}
\newcommand{\Pp}{\mathcal{P}}
\newcommand{\F}{\mathcal{F}}
\newcommand{\dm}{\boldsymbol{d}}                
\newcommand{\Lip}{\operatorname{Lip}}
\newcommand{\diver}{\operatorname{div}}

\newcommand{\bx}{\boldsymbol{x}}
\newcommand{\by}{\boldsymbol{y}}
\newcommand{\bX}{\boldsymbol{X}}
\newcommand{\bY}{\boldsymbol{Y}}
\newcommand{\balpha}{\boldsymbol{\alpha}}
\newcommand{\1}{\mathbf{1}}

\title[Optimal rate in mean field control via shadow flows]
      {The optimal rate of convergence in mean field control via recoupled shadow flows}

\author{Sebastian Munoz}
\address{Department of Mathematics, University of California, Los Angeles,
CA 90095, USA}
\email{sebastian@math.ucla.edu}


\begin{document}

\begin{abstract}
We determine the rate of convergence of the value functions of the
$N$-particle stochastic optimal control problem to the value function
of the corresponding mean field control problem, for mean field costs
that are merely Lipschitz continuous in the $1$-Wasserstein
distance---a class that covers problems whose mean field optimizers
are neither unique nor stable.  For $d\geq2$, the optimal rate is the
empirical-measure rate ($N^{-1/d}$ for $d\geq3$,
$N^{-1/2}\sqrt{\log N}$ for $d=2$): this proves the rate conjectured
by Daudin, Delarue and Jackson, and removes the semiconcavity
hypothesis made there.  In dimension one, we discover that the
empirical-measure benchmark is not optimal: cooperating particles beat
it, and the optimal polynomial exponent is $4/7$, strictly between the
accuracy of independent samples and that of quantization by freely
placed points.  The proofs are control-theoretic: from each
realization of an $N$-particle control we build a pathwise
Fokker--Planck flow---a \emph{shadow flow}---which, repeatedly
recoupled to the particles by optimal transport, shadows the empirical
measure at the optimal rate.  The one-dimensional rate requires additional constructions:
we correct the shadow flow with a filter built on the future of the
discarded noise, draw the cooperating particles from a Gibbs law, and
prove the optimality of the exponent by a Schr\"odinger ground-state
estimate.  The empirical-measure rate also holds under additive common
noise, uniformly in its intensity.
\end{abstract}

\maketitle

\setcounter{tocdepth}{1}
\tableofcontents

\section{Introduction}

\subsection{The convergence problem}
Mean field control theory studies the optimal control of McKean--Vlasov
dynamics, understood as the limit, as $N\to\infty$, of the cooperative optimal
control of $N$ exchangeable interacting particles.  The basic objects are two
value functions.  On the one hand, for $N\in\mathbb N$, an initial time
$t_0\in[0,T]$ and an initial configuration $\bx=(x^1,\dots,x^N)\in(\T^d)^N$,
the \emph{$N$-particle value function} is
\begin{equation*}
 V^N(t_0,\bx)
 =\inf_{\balpha}\,
 \E\bigg[\int_{t_0}^T\Big(\frac1N\sum_{i=1}^NL(X^i_t,\alpha^i_t)
 +F(\mu^N_t)\Big)\,dt+G(\mu^N_T)\bigg],
 \qquad
 \mu^N_t:=\frac1N\sum_{i=1}^N\delta_{X^i_t},
\end{equation*}
where the states evolve on the torus $\T^d$ according to
$dX^i_t=\alpha^i_t\,dt+\sqrt2\,dW^i_t$, driven by independent Brownian motions
$(W^i)_{i\geq1}$, and the infimum runs over all square-integrable controls
that are progressively measurable with respect to the common
filtration.  In
particular, the players may correlate their actions in an arbitrary,
path-dependent way.  On the other hand, the \emph{mean field value
function} is defined through the Fokker--Planck (Eulerian) formulation: for
$m_0\in\Pp(\T^d)$,
\begin{equation}\label{eq:U-informal}
 U(t_0,m_0)
 =\inf_{(m,\alpha)}
 \bigg\{\int_{t_0}^T\Big(\int_{\T^d}L\big(y,\alpha(t,y)\big)\,m_t(dy)
 +F(m_t)\Big)dt+G(m_T)\bigg\},
\end{equation}
the infimum running over \emph{deterministic} curves
$m:[t_0,T]\to\Pp(\T^d)$ paired with Borel drift fields $\alpha$, subject to
the Fokker--Planck equation
$\partial_tm_t=\Delta m_t-\diver(m_t\,\alpha(t,\cdot))$ with $m_{t_0}=m_0$.
See Section~\ref{subsec:problems} for the precise class of pairs and for
the equivalence with the McKean--Vlasov (Lagrangian) formulation.  It is a
classical fact that
$V^N(t_0,\bx)-U(t_0,m^N_{\bx})\to 0$ as $N\to\infty$, where
$m^N_{\bx}:=\frac1N\sum_i\delta_{x^i}$.  Qualitative convergence results of
this type were proved in \cite{Lacker2017}, via weak-convergence methods
going back to \cite{BudhirajaDupuisFischer2012}, and, in great
generality including common noise, in \cite{DjetePossamaiTan2022}.  See also
\cite{FornasierLisiniOrrieriSavare2019,CavagnariLisiniOrrieriSavare2022,
MayorgaSwiech2023} and
\cite{CarmonaDelarue2018I,CarmonaDelarue2018II} for background, and
\cite{HuangMalhameCaines2006,LasryLions2007,Fischer2017,
CardaliaguetDelarueLasryLions2019}
for the companion (non-cooperative) theory of mean field games.

In its quantitative form, the \emph{convergence problem} in mean field
control asks for the sharp rate at which
\begin{equation}\label{eq:gap}
 \sup_{(t,\bx)\in[0,T]\times(\T^d)^N}
 \big|V^N(t,\bx)-U(t,m^N_{\bx})\big|
\end{equation}
tends to zero.  The first global algebraic rate in the continuous-state
diffusive setting, valid without convexity of the costs $F$ and $G$ in the
measure argument and without smoothness of the limiting value function,
was obtained by Cardaliaguet, Daudin, Jackson and Souganidis
\cite{CardaliaguetDaudinJacksonSouganidis2023}.  Earlier, Germain, Pham
and Warin \cite{GermainPhamWarin2022} had obtained the rate $1/N$ for
limiting problems with smooth solutions.  See also
\cite{BayraktarCecchinChakraborty2023} for a regime-switching model.  It was then substantially
sharpened by Daudin, Delarue and Jackson \cite{DaudinDelarueJackson2024}, who
also identified the critical role played by the regularity of the data.  When
$F$ and $G$ are Lipschitz continuous with respect to the $1$-Wasserstein
distance $d_1$ (the ``$d_1$-regular case''), the natural benchmark for the
quantity \eqref{eq:gap} is the \emph{empirical-measure
rate}
\begin{equation}\label{eq:RNd}
 r_{N,d}:=
 \begin{cases}
 N^{-1/2}, & d=1,\\
 N^{-1/2}\sqrt{\log(1+N)}, & d=2,\\
 N^{-1/d}, & d\geq3,
 \end{cases}
\end{equation}
the worst-case expected $d_1$-error of the empirical measure of
$N$ independent samples from a probability measure on $\T^d$
\cite{FournierGuillin2015,DereichScheutzowSchottstedt2013,AjtaiKomlosTusnady1984,BobkovLedoux2021}.
For $d\geq2$, the quantity \eqref{eq:gap} indeed cannot tend to zero faster
than $N^{-1/d}$: a counterexample realizing this obstruction is given in
\cite[Sec.~2.4]{DaudinDelarueJackson2024}.  At $d=1$
this obstruction falls short of the benchmark, and there is a surprising mechanism at play (Section~\ref{subsec:intro-one-dim}).  In the positive direction,
the main result of \cite[Thm.~2.6]{DaudinDelarueJackson2024} in the
$d_1$-regular case states that, when $F$ and $G$ are Lipschitz and
semiconcave with respect to $d_1$, for every $\eta>0$,
\begin{equation}\label{eq:DDJ-rate}
 V^N(t,\bx)-Cr_{N,d}
 \;\leq\;
 U(t,m^N_{\bx})
 \;\leq\;
 V^N(t,\bx)+CN^{-2/(3d+6)+\eta}.
\end{equation}
The lower bound (the ``easy inequality'', in the terminology of
\cite{DaudinDelarueJackson2024}) already reaches the
benchmark: it is established there with
the factor $\log(1+N)$ at $d=2$, and its proof yields
\eqref{eq:DDJ-rate} as stated once the empirical-measure estimate of
\cite{FournierGuillin2015} is replaced by that of
\cite[Thm.~3]{BobkovLedoux2021}.  The upper
bound (the ``hard inequality'') does not, and closing this gap, that is,
proving
\begin{equation}\label{eq:hard-intro}
 U(t,m^N_{\bx})\;\leq\;V^N(t,\bx)+C\,r_{N,d}
\end{equation}
in the $d_1$-regular case, was raised as an open problem in
\cite[Sec.~1.5]{DaudinDelarueJackson2024}.  Subsequently, Cecchin,
Daudin, Jackson and Martini
\cite[Thm.~2.4]{CecchinDaudinJacksonMartini2025}, refining the
techniques of \cite{DaudinDelarueJackson2024}, obtained, for a
constant nondegenerate idiosyncratic noise and allowing for a common
noise, the bound
\[
 U(t,m^N_{\bx})\;\leq\;V^N(t,\bx)+C\,
 \begin{cases}
  N^{-1/(d+7)}, & d\ \text{even},\\
  N^{-1/(d+6)}, & d\ \text{odd},
 \end{cases}
\]
which improves the rate for $d\geq7$.

Using a different, purely analytic viscosity-solution
approach, Bayraktar, Ekren and Zhang \cite{BayraktarEkrenZhang2025}
obtained a non-sharp bound on \eqref{eq:gap} for smooth data; see also
\cite{BayraktarEkrenZhang2025Comparison,BayraktarEkrenHeZhang2026} for
the underlying comparison theory.

\subsection{Main result: the conjectured rate}
In this paper we prove the hard inequality \eqref{eq:hard-intro}, assuming
only $d_1$-Lipschitz continuity of the costs (semiconcavity is not needed).
In this class the mean field optimizers may be neither
unique nor stable, and the value function $U$ need not be
differentiable.
Together with the easy inequality (established in
\cite[Prop.~6.2]{DaudinDelarueJackson2024}, also without
semiconcavity), this yields the following two-sided bound.  See
Section~\ref{sec:setting} for the precise setting,
Theorem~\ref{thm:main} for the statement including constant
dependencies, and Theorem~\ref{thm:common-noise} for
the common-noise statement.

\begin{maintheorem}
Assume that the Hamiltonian $H\in C^2(\T^d\times\R^d)$ satisfies, for some
constant $C_0\geq1$,
\[
 C_0^{-1}I_d\leq D^2_{pp}H(x,p)\leq C_0I_d,
 \qquad
 |D_xH(x,p)|\leq C_0\big(1+|p|\big),
 \qquad (x,p)\in\T^d\times\R^d,
\]
and that $F,G:\Pp(\T^d)\to\R$ are
Lipschitz continuous with respect to $d_1$.  Then there exists a constant $C$,
independent of $N$, such that for all $N\geq1$ and all
$(t,\bx)\in[0,T]\times(\T^d)^N$,
\[
 \big|U(t,m^N_{\bx})-V^N(t,\bx)\big|\leq Cr_{N,d}.
\]
For $d\geq2$ the rate $r_{N,d}$ cannot be improved.  Moreover, the
same result holds in the presence of an additive common noise
$\sigma_0\,dW^0_t$, with $C$ independent of both $N$ and $\sigma_0\in\R$.
\end{maintheorem}

At $d=2$, the optimality includes the exact logarithmic factor
(Remark~\ref{rem:comparison-DDJ}).  Dimension one is exceptional.  It is the subject of
Section~\ref{subsec:intro-one-dim} below.

In the common-noise setting, Theorem~\ref{thm:common-noise} improves
the rate of \cite[Thm.~2.4]{CecchinDaudinJacksonMartini2025} to
$r_{N,d}$ and removes their semiconcavity hypothesis.

Theorem~\ref{thm:main} is complementary to the work of Cardaliaguet, Jackson,
Mimikos-Sta\-ma\-to\-pou\-los and Souganidis
\cite{CardaliaguetJacksonMimikosSouganidis2023}, who proved, for smooth data,
the much faster rate $1/N$ locally in the open and dense region of strong
regularity identified by Cardaliaguet and Souganidis
\cite{CardaliaguetSouganidis2022}.  The two results describe
the two regimes---local strong regularity versus global low regularity---in
which sharp rates are now available.  For the corresponding sharp results on
finite state spaces see \cite{Cecchin2021,Kolokoltsov2012}; for convex
costs see \cite[Prop.~2.8]{DaudinDelarueJackson2024} and
\cite{JacksonLacker2025}; for related
quantitative results see \cite{GermainPhamWarin2022}.

\subsection{Main result: the exceptional dimension one}\label{subsec:intro-one-dim}
In dimension one, strikingly, the conjectured benchmark can be beaten.

\begin{maintheorem}[dimension one]
Let $d=1$ and let the assumptions above hold.  There exists a constant
$C$, independent of $N$, such that for all $N\geq1$ and all
$(t,\bx)\in[0,T]\times\T^N$,
\[
 \big|U(t,m^N_{\bx})-V^N(t,\bx)\big|\leq CN^{-4/7}\log^{1/7}(1+N),
\]
and the polynomial exponent $4/7$ cannot be improved.
\end{maintheorem}

See Theorem~\ref{thm:one-dim} for the full statement: the hard
inequality holds without the logarithmic factor, and the optimality is
established by an explicit example.

The benchmark $r_{N,1}=N^{-1/2}$ is the accuracy of independent
samples, and cooperating particles can do better.  The opposite
extreme is full coordination: $N$ points of mass $1/N$ that are placed
rather than sampled (at quantiles, say) approximate any measure on
$\T$ in $d_1$ at the equal-weight quantization rate $N^{-1}$
\cite{GrafLuschgy2000,BencheikhJourdain2022}.
The optimal exponent $4/7$ falls strictly between the two benchmarks:
cooperation pays off, but particles that are to stay better spread than
random samples must fight their own diffusion, incurring a running cost.  This comparison also explains why the exception is low-dimensional: for $d\geq3$, independent samples already achieve the
quantization rate $N^{-1/d}$ of absolutely continuous targets, and the
example of \cite{DaudinDelarueJackson2024} transfers this geometric
obstruction to the control problem, while at $d=2$ the two benchmarks
differ only by the factor
$\sqrt{\log(1+N)}$ (Remark~\ref{rem:comparison-DDJ}).  Only at $d=1$
is there polynomial room between the two, and the true rate lies
strictly inside it.

At this precision the easy inequality also stops being easy: the known
proofs (e.g., \cite[Prop.~6.2]{DaudinDelarueJackson2024}) realize a
mean field control through $N$ independent copies,
whose empirical measure fluctuates at the central-limit scale, and no
such construction reaches the exponent $4/7$.  Both directions
therefore require new constructions, discussed below and carried out in
Section~\ref{sec:one-dim}, the technical climax of the paper.

\subsection{Strategy of the proofs}\label{subsec:strategy}
The obstruction in the hard inequality \eqref{eq:hard-intro}, discussed
already in \cite[Sec.~1.5]{DaudinDelarueJackson2024}, is the following.  Given an
$N$-particle control $\balpha$, its empirical measure flow
$(\mu^N_t)_{t_0\leq t\leq T}$ is a \emph{random} flow of measures, while the
mean field problem optimizes over \emph{deterministic} flows.  Since the
particle controls may be arbitrarily correlated, no law of large numbers is
available to decouple them.  Diffusion is the source of the obstruction.  Indeed, when
idiosyncratic noise is absent, Cecchin, Daudin, Jackson and Martini have
shown that an $N$-particle control lifts directly to a mean field control,
yielding $U(t,m^N_{\bx})\leq V^N(t,\bx)$ without error
\cite[Prop.~6.1]{CecchinDaudinJacksonMartini2025}.  Independent
Brownian noises destroy this lift, because the realized empirical flow no
longer solves an admissible deterministic diffusive Fokker--Planck equation.

The PDE approach of \cite{DaudinDelarueJackson2024}, and its refinement in \cite{CecchinDaudinJacksonMartini2025}, circumvent this obstruction through comparison arguments on the Wasserstein space, but yield non-sharp exponents. In the argument of \cite{DaudinDelarueJackson2024}, the loss arises from the regularization needed to carry out the comparison with \(V^N\): applying the finite-particle Laplacian to the regularized candidate subsolution produces a correction of order \(N^{-1}\), controlled by its second derivative in the measure variable. With merely Lipschitz costs, \(U\) does not have this regularity a priori, and the approximation error introduced in creating it leads to the weaker rate.

Our proof is control-theoretic: it does not regularize the value function
and never touches the infinite-dimensional Hamilton--Jacobi equation.
Since $U$ is an infimum, the hard inequality needs only a competitor: it
suffices to produce, from each particle control, an admissible pair for the
mean field problem whose expected cost is within the optimal error of the
particle cost.  The shadow flow is that competitor.  The proof rests on
three ideas.

\smallskip
\emph{(i) A pathwise shadow flow.}
Fix an adapted particle control $\balpha=(\alpha^1,\dots,\alpha^N)$ with
$|\alpha^i|\leq M$, and fix a realization $\omega$ of the noise.  We
construct from it a flow $\nu=(\nu_t)_{t_0\leq t\leq T}$ of
probability measures---the \emph{shadow flow}---starting from the initial
empirical measure, which
solves pathwise a Fokker--Planck equation
\[
 \partial_t\nu_t=\Delta\nu_t-\diver(\beta_t\nu_t),
 \qquad |\beta_t|\leq M .
\]
The flow, and with it the drift $\beta$, is assembled from mass-$(1/N)$
components, one per particle: the component attached to particle $i$ is
transported by the realized drift $\alpha^i_t(\omega)$ of that particle,
while the realized Brownian increments are replaced by heat smoothing.
At the times $s_k=t_0+kh$ of a grid of mesh $h$, the
components are recoupled to the particles: the existing mass of
$\nu_{s_k}$ is redecomposed into $N$ pieces of mass $1/N$ so as to
minimize the total cost of transporting each piece to its particle.
The minimal cost is the distance $d_1(\nu_{s_k},\mu^N_{s_k})$ itself.  This
relabeling is free: no mass actually moves, so the cost functional
never sees it and the pathwise Fokker--Planck equation survives each
grid time; yet it reattaches each particle control to nearby shadow
mass.  

This freeness is what makes the shadow flow effective: the
coupling to the particles can be re-optimized as often as the argument
demands, at no price in action or in admissibility. Between grid times the attachment
degrades (the components spread and drift while the particles
diffuse), but only by $C\sqrt h$ (Lemma~\ref{lem:provisional}).
Crucially, these local errors do not accumulate, because they enter the
final estimate only through time integrals against integrable kernels rather than
interval by interval.

The flow depends on the control, on the realized noise, and on the grid,
as well as on $N$.  The mesh $h$ is sent to zero only at the end of the
argument.  The shadow flow is thus a quantitative diffusive analogue of the
naive lift available in the noiseless case
\cite{CecchinDaudinJacksonMartini2025}, to which it collapses when the
idiosyncratic noise is switched off: every component is then a Dirac
mass following its particle.  Heat smoothing and recoupling are what
diffusion adds.  The grid construction may be viewed formally as an approximation of a continuously recoupled shadow flow, whose attachment to the particles is refreshed at every instant; see Remark~\ref{rem:continuous-shadow}.

\smallskip
\emph{(ii) Shadowing the empirical measure.}
The key quantitative fact that justifies our constructed competitor is the estimate
(Proposition~\ref{prop:tracking})
\[
 \sup_{t_0\leq t\leq T}\E\, d_1(\nu_t,\mu^N_t)\leq C_M\big(r_{N,d}+\sqrt h\big),
\]
which says that the shadow flow follows the empirical measure at the optimal rate,
uniformly in time. The difference of the two flows is driven by two error
terms.  The first is a flux mismatch: both flows transport their mass
along the same realized controls, applied on one side to the particles and
on the other to the nearby components, so the tightness of the coupling
controls the heat-smoothed flux mismatch
(Lemma~\ref{lem:flux}).  The second is the discarded Brownian noise, which
enters the difference as a martingale.  Heat smoothing tames both: it
converts the flux mismatch into an integrable singularity
$(t-s)^{-1/2}$---the critical parabolic gain---and it caps the
mass of the martingale
(Lemma~\ref{lem:martingale}).  By the It\^o
isometry and translation invariance, the size of the
martingale does not depend on the particle positions at all. For this
reason, beyond the independence of the driving idiosyncratic Brownian motions, the estimate requires neither independence nor exchangeability of the particle states, nor
any mixing of the controls.

The accounting is exactly critical: as the
smoothing time shrinks, the flux grows at precisely the rate at which the
heat flow damps it, so the flux mismatch is returned neither amplified nor
reduced---it comes back proportional to the distance being
estimated---and the true smallness must come entirely from the martingale
term and from the choice of the smoothing scale.  This exact balance is where the sharp
rate is won.  The second-order regularity that the comparison argument of
\cite{DaudinDelarueJackson2024} must manufacture by regularizing the value
function, at a price in powers of $N$, is provided here by the heat flow
at no cost in the rate (Remark~\ref{rem:critical}):
parabolic smoothing is intrinsic to the dynamics, but it acts on
flows, affecting only the competitor, while the value
function inherits none of their smoothing.  A Gr\"onwall argument closes the loop, and optimizing the
smoothing scale produces exactly the empirical-measure rate $r_{N,d}$.

\smallskip
\emph{(iii) The cost of the competitor.}
Finally, the running cost is transferred from the particles to the shadow
flow: the drift $\beta_t$ is, at each point, a convex combination of the
realized particle controls, so Jensen's inequality bounds the action of
$\nu$ by the average particle action, up to a coupling error, and the
$d_1$-Lipschitz continuity of $F$ and $G$ handles the mean field costs.
For each fixed $\omega$, the pair $(\nu^\omega,\beta^\omega)$ is an
admissible deterministic competitor in the Fokker--Planck formulation of
$U$, so $U(t_0,m^N_{\bx})\leq J(\nu^\omega,\beta^\omega)$ pointwise, and
the hard inequality follows by taking expectations.

\smallskip
The above proves the hard inequality \eqref{eq:hard-intro} for particle
controls bounded by $M$,
with a constant depending on $M$ but not on $N$
(Proposition~\ref{prop:comparison}).  The passage to the value function
$V^N$ uses a smoothing of $F$ and $G$ that preserves the $d_1$-Lipschitz
constant \cite[Lem.~4.1 and App.~B]{DaudinDelarueJackson2024}, together
with the uniform bound $|D_{x^i}V^{N,n}|\leq C/N$ for the smoothed problems
\cite[Lem.~6.1]{DaudinDelarueJackson2024}, which makes their optimal
feedback controls bounded uniformly in $n$ and $N$.

\smallskip
\emph{The one-dimensional case.}
At $d=1$ each direction requires a substantially deeper
construction, and both reflect
the trade described in Section~\ref{subsec:intro-one-dim}: accuracy beyond the
central-limit scale is bought with control energy.

For the hard inequality \eqref{eq:hard-intro}, the construction is a
corrected shadow flow: the transport of each component acquires an
additional drift, built from the future of the realized noise.  The
correction is tailored to control the dominant term of the shadowing error, the
discarded noise, which enters at the central-limit scale $N^{-1/2}$.
Heat smoothing already controls its high frequencies, so the added
drift steers against the low Fourier modes before they can
accumulate.  The
shadowing argument of Section~\ref{sec:tracking} is then rerun for the
corrected flow: damping the first $K$ frequencies costs drift energy
of order $K^{3}/N$, the surviving noise contributes $1/\sqrt{NK}$, and
at the optimal choice $K=N^{1/7}$ the two meet at $N^{-4/7}$.  The
corrected flow is anticipating rather than adapted, yet each of its
realizations is a deterministic Fokker--Planck pair, admissible in
\eqref{eq:U-informal}.  The resulting bound on $U$ holds realization by
realization, and the hard inequality follows by taking expectations.

The (no longer) easy inequality requires a competitor for the other problem: $V^N$ is itself an infimum, and the task is to produce a
particle control whose empirical measure follows a near-optimal mean
field flow, at a
cost within $N^{-4/7}\log^{1/7}(1+N)$ of the mean field value $U$.  As discussed earlier, independent
copies fluctuate at the central-limit scale $N^{-1/2}$, so our construction necessarily makes the particles cooperate.  Their positions are drawn from a Gibbs law that
exponentially penalizes the $d_1$-distance of the empirical measure to
the uniform measure, at an inverse temperature $\varsigma$, and they are carried
along the target flow itself (regularized over a short initial
heat layer) by its quantile map, each particle
tracking a quantile of the moving flow.  At
inverse temperature $\varsigma$ the empirical measure typically lies within
$(N\varsigma)^{-1/3}$ of the uniform measure, and the quantile map carries this
accuracy along the flow.  The price is the drift that holds the
configuration together: it is a pointwise multiple of the gradient of
the Gibbs potential, so its first-order cost vanishes by an exact
integration by parts against the Gibbs law, which the dynamics hold
stationary.  What remains is its
quadratic energy, of order $(\varsigma/N)^2$ times the logarithm
generated by the initial heat layer.  As before, the two scalings meet at
the optimal temperature, $\varsigma=N^{5/7}\log^{-3/7}(1+N)$, at the
rate $N^{-4/7}\log^{1/7}(1+N)$.

The optimality of the exponent follows from a Schr\"odinger ground-state
estimate: for explicit data, a
Cole--Hopf transformation, as in the optimality example of
\cite{DaudinDelarueJackson2024}, writes the particle value as a partition
function, and a cell-by-cell spectral bound shows that even the
best-spread configurations pay at least $N^{-4/7}$.  This is carried
out in Section~\ref{sec:one-dim}.

\smallskip
\emph{Extension to common noise.}
With additive common noise $\sigma_0\,dW^0_t$ in the dynamics, the random
translation $x\mapsto x-\sigma_0(W^0_t-W^0_{t_0})$, a standard reduction for
spatially homogeneous additive common noise
\cite{LackerWebster2015,CardaliaguetDelarueLasryLions2019}, reduces the particle
system to the zero-common-noise dynamics, and the shadowing estimate applies
verbatim, since its martingale term involves only the idiosyncratic noises.
What requires care is admissibility: a mean field competitor may draw its
randomness from $W^0$ alone, while the translated shadow flow depends on
all the noises.  Conditioning on $W^0$ would not help, because for
correlated controls the empirical measure need not concentrate
conditionally on the common noise.  Instead we freeze the idiosyncratic
paths and read the shadow flow as a family of $W^0$-adapted competitors,
one for each frozen realization (Lemma~\ref{lem:common-sections}).  The
easy inequality is also proved directly in the common-noise setting, by
replicating a mean field control across the particles and exploiting
conditional independence, so that semiconcavity is not needed on either
side.

\smallskip
\emph{The whole space, nonseparable Hamiltonians, and degenerate noise.}
The shadow flow is constructed pathwise, and makes no essential use of
the compactness of the torus or of the separated form
$L(x,\alpha)+F(m)$ of the cost.  We therefore expect
Theorem~\ref{thm:main} to extend to the whole space $\R^d$ under moment
bounds on the initial empirical measure, and to nonseparable
Hamiltonians $H(x,p,m)$ Lipschitz in the measure argument. Regarding the still-open case $d=1$ with no idiosyncratic noise (where the missing ingredient is the ``easy'' inequality), we expect the sharp rate to be the full quantization rate $N^{-1}$, since in this case the particles need not spend energy steering against the noise.  The ideas of Section~\ref{subsec:gibbs} appear well-suited for such an extension. By contrast, one of the most interesting directions left to investigate is the problem with degenerate idiosyncratic noise, which seems to require substantially new ideas.

\subsection{Organization of the paper}
Section~\ref{sec:setting} describes the setting and states the main results.
Section~\ref{sec:shadow} constructs the recoupled shadow flow and proves the
coupling-cost estimate.  Section~\ref{sec:tracking} proves the sharp
shadowing estimate.  Section~\ref{sec:comparison} transfers the cost to the shadow
flow and
completes the proof of Theorem~\ref{thm:main}.
Section~\ref{sec:one-dim} establishes
Theorem~\ref{thm:one-dim}, and Section~\ref{sec:common-noise} gives
the extension to additive common noise
(Theorem~\ref{thm:common-noise}).
Appendix~\ref{app:construction} collects the
measurable-selection, Fokker--Planck, and noise-regularity lemmas
and the construction of regular near-optimal mean field flows,
together with
the properties of the Lagrangian and the heat-kernel,
frequency-cutoff, empirical-measure, and quantization estimates used
throughout.
Appendix~\ref{app:two-dim} contains the construction that shows optimality of
Theorem~\ref{thm:main} in dimension $2$.

\section{Setting and main results}\label{sec:setting}

\subsection{Notation}\label{subsec:notation}
We work on the flat torus $\T^d=\R^d/\Z^d$, whose geodesic distance is
denoted by $\dm$.  Note $\operatorname{diam}(\T^d)=\sqrt d/2$.  Points of
$(\T^d)^N$ are written $\bx=(x^1,\dots,x^N)$, and
$m^N_{\bx}:=\frac1N\sum_{i=1}^N\delta_{x^i}$.  The set $\Pp(\T^d)$ of Borel
probability measures on $\T^d$ is equipped with the $1$-Wasserstein distance, which by
 duality \cite[Chap.~5]{Villani2009} may be defined by
\begin{equation}\label{eq:d1-duality}
 d_1(m_1,m_2)
 =\sup\Big\{\int_{\T^d}\varphi\,d(m_1-m_2)\;:\;
 \varphi:\T^d\to\R,\ \Lip(\varphi)\leq1\Big\}.
\end{equation}
 More generally, for
a finite signed measure $\sigma$ on $\T^d$ with $\sigma(\T^d)=0$ we write
$\|\sigma\|_{\Lip^*}$ for the supremum in \eqref{eq:d1-duality} with $m_1-m_2$
replaced by $\sigma$, so that $d_1(m_1,m_2)=\|m_1-m_2\|_{\Lip^*}$.

Functions on $\T^d$ are identified with $\Z^d$-periodic functions on $\R^d$.
The Fourier coefficients of $f$ are
$\widehat f(k)=\int_{\T^d}f(x)e^{-2\pi ik\cdot x}dx$, $k\in\Z^d$.  We write
$(P_\tau)_{\tau\geq0}$ for the heat semigroup generated by $\Delta$ on
$\T^d$, i.e.\ $\widehat{P_\tau f}(k)=e^{-4\pi^2|k|^2\tau}\widehat f(k)$,
acting on measures by duality,
$\int_{\T^d}\varphi\,d(P_\tau\mu):=\int_{\T^d}P_\tau\varphi\,d\mu$.
Equivalently, $P_\tau\mu=p_\tau\star\mu$, where
\begin{equation}\label{eq:heat-kernel-def}
 p_\tau(x)=\sum_{k\in\Z^d}(4\pi\tau)^{-d/2}\,e^{-|x+k|^2/(4\tau)}
\end{equation}
is the periodization of the Gaussian kernel on $\R^d$.  On finite
$\R^d$-valued measures, $P_\tau$ acts componentwise, and
$\int_{\T^d}\Phi\cdot d\mathsf m:=\sum_{j=1}^d\int_{\T^d}\Phi^j\,d\mathsf m^j$
for such a measure $\mathsf m$ and a bounded Borel
$\Phi:\T^d\to\R^d$.  For $a\in\T^d$,
$\tau_a$ denotes the translation $x\mapsto x+a$, and $\tau_a{}_\#\mu$ the
push-forward of the measure $\mu$.

Throughout, $T>0$ is a fixed horizon, and
$(\Omega,\F,(\F_t)_{0\leq t\leq T},\Prob)$ is a filtered probability space
satisfying the usual conditions, with $\F_0$ atomless, carrying independent
$d$-dimensional Brownian motions $(W^i)_{i\geq1}$.  For a continuous
vector-valued martingale $\mathsf M$, its quadratic covariation is the
matrix-valued process
\[
 \langle\mathsf M\rangle_t
 :=\lim_{|\Pi|\to0}\sum_{k}\big(\mathsf M_{t_{k+1}}-\mathsf M_{t_k}\big)
 \big(\mathsf M_{t_{k+1}}-\mathsf M_{t_k}\big)^{\mathsf T},
\]
where $\Pi=\{0=t_0<t_1<\dots<t_n=t\}$ ranges over the partitions of
$[0,t]$, $|\Pi|$ denotes the mesh, and the limit is in probability.  Constants denoted $C$, $c$ depend only on
the dimension $d$, the horizon $T$, the Hamiltonian $H$, and the Lipschitz
constants $L_F$, $L_G$ of Assumption~\ref{ass:standing}, and may change
from line to line.  Constants denoted $C_M$, $c_M$ may in addition depend
on the bound $M$ of \eqref{eq:bounded-control}.  None of these constants
depends on $N$, on the controls, or on the initial condition.  For
$d\geq2$ the dependence on $H$ is only through the constant $C_0$ and
$\sup_x|D_pH(x,0)|$.  The one-dimensional arguments of
Section~\ref{sec:one-dim} involve $H$ also through local bounds on its
second derivatives.

\subsection{The two control problems}\label{subsec:problems}
The data of the problem are a Hamiltonian $H:\T^d\times\R^d\to\R$ and two
mean field costs $F,G:\Pp(\T^d)\to\R$.

\begin{assumption}\label{ass:standing}
There exist constants $C_0\geq1$ and $L_F,L_G\geq0$ such that:
\begin{enumerate}[label=\textup{(A\arabic*)},leftmargin=3.2em]
\item\label{ass:H-convex} $H\in C^2(\T^d\times\R^d)$ and
$C_0^{-1}I_d\leq D^2_{pp}H(x,p)\leq C_0I_d$ for all $(x,p)$;
\item\label{ass:H-x} $|D_xH(x,p)|\leq C_0(1+|p|)$ for all $(x,p)$;
\item\label{ass:FG} $F$ and $G$ are, respectively, $L_F$- and $L_G$-Lipschitz
continuous with respect to $d_1$.
\end{enumerate}
\end{assumption}

The Lagrangian is defined through the Legendre duality
\begin{equation}\label{eq:Legendre}
 L(x,a):=\sup_{p\in\R^d}\big\{-H(x,p)-a\cdot p\big\},
 \qquad (x,a)\in\T^d\times\R^d,
\end{equation}
so that, conversely, $H(x,p)=\sup_{a}\{-L(x,a)-a\cdot p\}$.

\begin{remark}\label{rem:legendre}
Assumption~\ref{ass:standing} matches the setting of 
\cite{DaudinDelarueJackson2024}, except for the fact that we do not assume semiconcavity of the costs. 
\end{remark}

Fix $(t_0,\bx)\in[0,T]\times(\T^d)^N$.  The set $\mathcal A^N_{t_0}$ of
\emph{admissible particle controls} consists of all
$\balpha=(\alpha^1,\dots,\alpha^N)$, defined on an arbitrary stochastic
basis carrying independent Brownian motions $W^1,\dots,W^N$ on
$[t_0,T]$, with each $\alpha^i$ an $\R^d$-valued process on $[t_0,T]$,
progressively measurable with respect to any common filtration for
which $W^1,\dots,W^N$ remain independent Brownian motions, and
satisfying $\E\int_{t_0}^T|\alpha^i_t|^2dt<\infty$.  The corresponding state processes
and empirical measures are
\begin{equation}\label{eq:particle-SDE}
 dX^i_t=\alpha^i_t\,dt+\sqrt2\,dW^i_t,\quad t\in[t_0,T],
 \qquad X^i_{t_0}=x^i,
 \qquad
 \mu^N_t:=\frac1N\sum_{i=1}^N\delta_{X^i_t},
\end{equation}
the processes being regarded as $\T^d$-valued.  The cost and the value
function are
\begin{equation}\label{eq:JN-definition}
 J_N(t_0,\bx;\balpha)
 :=\E\bigg[\int_{t_0}^T\Big(\frac1N\sum_{i=1}^NL(X^i_t,\alpha^i_t)
 +F(\mu^N_t)\Big)dt+G(\mu^N_T)\bigg],
\end{equation}
\begin{equation}\label{eq:VN-definition}
 V^N(t_0,\bx):=\inf_{\balpha\in\mathcal A^N_{t_0}}J_N(t_0,\bx;\balpha).
\end{equation}

The mean field value function is defined through the Fokker--Planck
(Eulerian) formulation: for $m_0\in\Pp(\T^d)$,
\begin{equation}\label{eq:U-definition}
 U(t_0,m_0)
 :=\inf_{(m,\alpha)}
 \bigg\{\int_{t_0}^T\Big(\int_{\T^d}L\big(y,\alpha(t,y)\big)\,m_t(dy)
 +F(m_t)\Big)dt+G(m_T)\bigg\},
\end{equation}
the infimum running over all pairs $(m,\alpha)$ such that
$m=(m_t)_{t_0\leq t\leq T}\in C([t_0,T];\Pp(\T^d))$,
$\alpha:[t_0,T]\times\T^d\to\R^d$ is Borel measurable with
$\int_{t_0}^T\int_{\T^d}|\alpha(t,y)|^2m_t(dy)\,dt<\infty$, and
\begin{equation*}
 \partial_tm_t=\Delta m_t-\diver\big(m_t\,\alpha(t,\cdot)\big)
 \quad\text{in }(t_0,T)\times\T^d,
 \qquad m_{t_0}=m_0,
\end{equation*}
in the sense of distributions.  As discussed in
\cite[Sec.~2.2]{DaudinDelarueJackson2024}, by a mimicking argument
\cite{LackerShkolnikovZhang2020,BrunickShreve2013} this definition agrees with
the McKean--Vlasov (Lagrangian) formulation of the mean field control problem.
We will only use \eqref{eq:U-definition}.

\subsection{Main results}
Recall $r_{N,d}$ from \eqref{eq:RNd}.

\begin{theorem}\label{thm:main}
Let Assumption~\ref{ass:standing} hold.  There exists a constant $C$, depending
only on $d$, $T$, $H$, $L_F$, $L_G$, such that for every
$N\geq1$ and every
$(t,\bx)\in[0,T]\times(\T^d)^N$,
\begin{equation}\label{eq:two-sided}
 \big|U(t,m^N_{\bx})-V^N(t,\bx)\big|\;\leq\;C\,r_{N,d}.
\end{equation}
Moreover, for $d\geq2$ the rate $r_{N,d}$ cannot be improved under
Assumption~\ref{ass:standing}: the bound \eqref{eq:two-sided} fails
for every sequence of smaller order.
\end{theorem}

The content of the theorem is the hard inequality,
\begin{equation}\label{eq:hard-inequality}
 U(t,m^N_{\bx})\;\leq\;V^N(t,\bx)+C\,r_{N,d},
 \qquad (t,\bx)\in[0,T]\times(\T^d)^N,
\end{equation}
which occupies Sections~\ref{sec:shadow}--\ref{sec:comparison}.  The easy
inequality, $V^N(t,\bx)\leq U(t,m^N_{\bx})+C\,r_{N,d}$, is
\cite[Prop.~6.2]{DaudinDelarueJackson2024}, established there under
Assumption~\ref{ass:standing} alone, with
\cite[Thm.~3]{BobkovLedoux2021} supplying the sharp $d=2$ empirical
rate.

In dimension one, the rate $r_{N,1}=N^{-1/2}$ of Theorem~\ref{thm:main} is
achieved but not optimal.  The next result determines the optimal
polynomial exponent, with a logarithmic gap.

\begin{theorem}[The one-dimensional rate]\label{thm:one-dim}
Let $d=1$ and let Assumption~\ref{ass:standing} hold.  There exists a
constant $C$, depending only on $T$, $H$, $L_F$, $L_G$, such
that for every $N\geq1$ and every $(t,\bx)\in[0,T]\times\T^N$,
\begin{equation}\label{eq:one-dim-rate}
 \big|U(t,m^N_{\bx})-V^N(t,\bx)\big|
 \leq CN^{-4/7}\log^{1/7}(1+N).
\end{equation}
Moreover,
\begin{equation}\label{eq:one-dim-hard-statement}
 U(t,m^N_{\bx})\leq V^N(t,\bx)+CN^{-4/7}.
\end{equation}
The polynomial exponent $4/7$ cannot be improved under
Assumption~\ref{ass:standing}.
\end{theorem}

The optimality is established by an explicit example: for $H(x,p)=\frac12p^2$,
$F(m)=d_1(m,\lambda)$, with $\lambda$ the uniform measure on $\T$, and
$G=0$, there exist $c>0$ and, for every sufficiently large integer $R$, a
configuration $\bx_R\in\T^N$ with $N=R^7$, such that
\begin{equation}\label{eq:one-dim-lower}
 V^N(0,\bx_R)-U(0,m^N_{\bx_R})\geq c\,N^{-4/7}.
\end{equation}
We conjecture that the logarithmic factor in \eqref{eq:one-dim-rate}
can be removed.

The third main result extends Theorem~\ref{thm:main} to the model in
which the particles and the mean field limit share an additive common
noise: the
two-sided bound \eqref{eq:two-sided} persists, with a constant
independent of the intensity $\sigma_0$ of the common noise, and the
rate remains optimal for $d\geq2$.  Its statement requires the
common-noise formulations of the two control problems, and is
therefore deferred to Section~\ref{sec:common-noise}
(Theorem~\ref{thm:common-noise}).

\begin{remark}\label{rem:comparison-DDJ}
(i) The bound \eqref{eq:two-sided} improves the exponent $2/(3d+6)$
of \cite[Thm.~2.6]{DaudinDelarueJackson2024} to the sharp rate when
$d\geq2$.  For $d\geq3$, the example establishing the optimality
clause, in the proof of Theorem~\ref{thm:main}
(Section~\ref{sec:comparison}), is an adaptation to the torus of
\cite[Prop.~2.9]{DaudinDelarueJackson2024}, with the Gaussian target
replaced by the heat kernel.  Since the case $\sigma_0=0$ is admissible in
Theorem~\ref{thm:common-noise}, its rate cannot be improved
either.  At $d=2$ the factor $\sqrt{\log(1+N)}$ is necessary, and the
obstruction is genuinely dynamical: coordinated configurations attain
$N^{-1/2}$ statically, but the idiosyncratic noise regenerates
matching-scale fluctuations on every short time window faster than
any affordable control can suppress them
(Proposition~\ref{prop:two-dim-lower}).  The same factor is necessary
for the empirical approximation problem itself, where the uniform
measure realizes the matching lower bound
\cite{AjtaiKomlosTusnady1984}.  See \cite[Sec.~7]{BobkovLedoux2021}
for a Fourier--analytic proof, whose method our
Lemma~\ref{lem:fresh-noise} adapts.  For $d=1$ the torus example
yields only a lower bound of order $N^{-1}$, the quantization rate,
and thus does not determine the sharp exponent; hence the bespoke
example \eqref{eq:one-dim-lower}.

(ii) There is no contradiction with the local rate $1/N$ of
\cite{CardaliaguetJacksonMimikosSouganidis2023}: that result assumes smooth
data and holds locally in the region of strong regularity identified in
\cite{CardaliaguetSouganidis2022}, where the mean
field optimizer is unique and stable and $U$ enjoys higher local Wasserstein
regularity.  The invariance of this region under optimal trajectories is also
essential.  Our estimate is instead global and assumes only $d_1$-Lipschitz
continuity, a class for which the example recalled in (i)
exhibits the empirical-measure obstruction.
\end{remark}

Theorem~\ref{thm:main} is a consequence of the following proposition, which
compares the mean field value with the cost of an arbitrary bounded particle
control: no optimality, independence, exchangeability, or Markov property is
required.

\begin{proposition}\label{prop:comparison}
Let Assumption~\ref{ass:standing} hold and let $M>0$.  Let
$(t_0,\bx)\in[0,T]\times(\T^d)^N$ and let $\balpha\in\mathcal A^N_{t_0}$
satisfy
\begin{equation}\label{eq:bounded-control}
 |\alpha^i_t|\leq M
 \quad\text{for }dt\otimes\Prob\text{-a.e.\ }(t,\omega),
 \quad i=1,\dots,N.
\end{equation}
Then
\begin{equation}\label{eq:comparison}
 U(t_0,m^N_{\bx})\;\leq\;J_N(t_0,\bx;\balpha)+C_M\,r_{N,d},
\end{equation}
where $C_M$ depends only on $d$, $T$, $H$, $L_F$, $L_G$, and $M$.
\end{proposition}

Sections~\ref{sec:shadow}--\ref{sec:comparison} are devoted to the proof of
Proposition~\ref{prop:comparison}.  Section~\ref{sec:comparison} concludes
with the passage to Theorem~\ref{thm:main}.  Throughout
Sections~\ref{sec:shadow}--\ref{sec:comparison} we fix $t_0$, $\bx$, $M$ and a
control $\balpha$ as in Proposition~\ref{prop:comparison}.  Modifying
each $\alpha^i$ on a $dt\otimes d\Prob$-null set, which changes neither
the state processes nor the cost, we assume that the bound
$|\alpha^i|\leq M$ holds at every $(t,\omega)$.  Every estimate
below is valid for all $N\geq1$.

\section{Construction of the shadow flow}\label{sec:shadow}

In this section we construct, from the fixed control $\balpha$, the
\emph{shadow flow}: a random flow $\nu=(\nu_t)_{t_0\leq t\leq T}$ with
values in $\Pp(\T^d)$, adapted to
$(\F_t)$, which starts from the empirical measure,
\begin{equation}\label{eq:shadow-initial}
 \nu_{t_0}=m^N_{\bx}=\mu^N_{t_0},
\end{equation}
and which, for each fixed $\omega$ outside a $\Prob$-null set, solves a
Fokker--Planck equation with drift bounded by $M$.
The construction is relative to a uniform partition
\begin{equation}\label{eq:grid}
 s_k:=t_0+kh,\qquad k=0,1,\dots,(T-t_0)/h,
\end{equation}
of $[t_0,T]$, whose mesh $h$, taken with $(T-t_0)/h$ an integer, is sent
to zero only at the very end, in the
proof of Proposition~\ref{prop:comparison}.

\subsection{The recursive definition}\label{subsec:construction}
The flow is defined recursively over the grid intervals.  Suppose
$\nu_{s_k}$ has been defined and is $\F_{s_k}$-measurable (as a
$\Pp(\T^d)$-valued random variable; for $k=0$ this is
\eqref{eq:shadow-initial}).  We first choose a decomposition of $\nu_{s_k}$
which couples it optimally with the labeled particle configuration:
a family $\pi_k=(\pi_{i,k})_{i=1}^N$ of (random) nonnegative measures on
$\T^d$, each of total mass $1/N$, such that
\begin{equation*}
 \sum_{i=1}^N\pi_{i,k}=\nu_{s_k}
\end{equation*}
and
\begin{equation}\label{eq:grid-coupling-optimal}
 \sum_{i=1}^N\int_{\T^d}\dm\big(y,X^i_{s_k}\big)\,\pi_{i,k}(dy)
 =d_1\big(\nu_{s_k},\mu^N_{s_k}\big).
\end{equation}

\begin{lemma}[Optimal decompositions; measurable selection]\label{lem:selection}
For every $\nu\in\Pp(\T^d)$ and $\by\in(\T^d)^N$, the minimum of
$\sum_i\int\dm(y,y^i)\,\pi_i(dy)$ over all families $(\pi_i)_{i=1}^N$ of
nonnegative measures with $\pi_i(\T^d)=1/N$ and $\sum_i\pi_i=\nu$ is
attained and equals $d_1(\nu,m^N_{\by})$.  A minimizing family can be
chosen as a Borel measurable function of
$(\nu,\by)\in\Pp(\T^d)\times(\T^d)^N$.
\end{lemma}

The proof, a standard compactness and measurable-selection argument, is given in
Appendix~\ref{app:deferred}.  Applying the Borel selector of
Lemma~\ref{lem:selection} to the $\F_{s_k}$-measurable pair
$(\nu_{s_k},\bX_{s_k})$ produces an $\F_{s_k}$-measurable choice of
$\pi_k$.

Next, for $t\in[s_k,s_{k+1})$ set
\begin{equation}\label{eq:component-definition}
 A^k_i(t):=\int_{s_k}^t\alpha^i_r\,dr,
 \qquad
 \nu_{i,t}:=\big(\tau_{A^k_i(t)}\big)_\#\big(P_{t-s_k}\,\pi_{i,k}\big),
 \qquad
 \nu_t:=\sum_{i=1}^N\nu_{i,t}.
\end{equation}
We call $\nu_{i,t}$ the \emph{component} attached to particle $i$: it is
transported by the realized drift of that particle and smoothed by the
heat flow, while the realized Brownian increments are discarded.  Each
$\nu_{i,t}$ has mass $1/N$, so
$\nu_t\in\Pp(\T^d)$.  The limit
$\nu_{s_{k+1}}:=\lim_{t\uparrow s_{k+1}}\nu_t$ exists in $\Pp(\T^d)$ (the
heat flow and the translations are continuous in $t$), which completes the
recursion: at $s_{k+1}$ the measure itself is retained, and only its
decomposition is refreshed via Lemma~\ref{lem:selection}.  We call this
refreshment the \emph{recoupling}.  In particular $t\mapsto\nu_t$ is
continuous on $[t_0,T]$, without jumps at the grid times.

For $t\in(s_k,s_{k+1})$ and $i=1,\dots,N$, the measure $\nu_{i,t}$ has the
strictly positive density
\[
 \rho_{i,t}(y)=\int_{\T^d}p_{t-s_k}\big(y-A^k_i(t)-z\big)\,\pi_{i,k}(dz),
\]
jointly continuous in $(t,y)\in(s_k,s_{k+1})\times\T^d$.  Define the
weights, the flux, and the drift
\begin{equation}\label{eq:beta-definition}
 \theta_i(t,y):=\frac{\rho_{i,t}(y)}{\sum_{i'=1}^N\rho_{i',t}(y)},
 \qquad
 q_t:=\sum_{i=1}^N\alpha^i_t\,\nu_{i,t},
 \qquad
 \beta_t(y):=\sum_{i=1}^N\theta_i(t,y)\,\alpha^i_t,
\end{equation}
for $t\in\bigcup_k(s_k,s_{k+1})$ and $y\in\T^d$, and $\beta_t:=0$ for
$t$ in the (Lebesgue-null) set of grid times.  The weights are
nonnegative with $\sum_i\theta_i=1$ and $\theta_i\,d\nu_t=d\nu_{i,t}$:
the drift is, at each point, a convex combination of the realized
controls, so $|\beta_t|\leq M$ everywhere.  Moreover
$q_t=\beta_t\,\nu_t$ off the grid times, and $\beta$ is jointly
measurable in $(t,y,\omega)$ and adapted.  Both $\nu$ and $\beta$
depend on $\omega$ through the realized controls and Brownian
increments.  We write this dependence as a superscript, as in
$(\nu^\omega,\beta^\omega)$, only when the pathwise nature of the
construction is being stressed.

\begin{lemma}[Pathwise Fokker--Planck equation for the shadow flow]\label{lem:FP}
There exists a $\Prob$-null set outside of which
\begin{equation}\label{eq:shadow-FP}
 \partial_t\nu_t=\Delta\nu_t-\diver\big(\beta_t\nu_t\big)
 \quad\text{in }(t_0,T)\times\T^d,
 \qquad \nu_{t_0}=m^N_{\bx},
\end{equation}
in the sense of distributions.  In particular, for each $\omega$ outside
this null set, $(\nu^\omega,\beta^\omega)$ is an admissible pair in the
definition \eqref{eq:U-definition} of $U(t_0,m^N_{\bx})$.
\end{lemma}

The proof (Appendix~\ref{app:deferred}) is elementary: on each grid
interval the integral $\int_{\T^d}\varphi\,d\nu_{i,t}$ can be computed
explicitly from \eqref{eq:component-definition}, and the heat semigroup and
the translation flow contribute exactly the two terms on the right of
\eqref{eq:shadow-FP}.  Continuity at the grid times glues the intervals.  The
admissibility statement then follows from
$\int_{t_0}^T\int|\beta_t|^2d\nu_tdt\leq M^2T<\infty$ and the continuity of
$t\mapsto\nu_t$.

\subsection{The provisional coupling and its cost}\label{subsec:provisional}
For $t\in[s_k,s_{k+1})$, pairing each component with its particle defines
a transport plan between $\nu_t$ and $\mu^N_t$, namely
$\sum_{i=1}^N\nu_{i,t}(dy)\,\delta_{X^i_t}(dx)$.  We call it the
\emph{provisional coupling}.  Its cost is
\begin{equation}\label{eq:c-definition}
 c_t:=\sum_{i=1}^N\int_{\T^d}\dm\big(y,X^i_t\big)\,\nu_{i,t}(dy),
\end{equation}
while the optimal cost is the distance $d_1(\nu_t,\mu^N_t)$ itself.  Thus
$d_1(\nu_t,\mu^N_t)\leq c_t$, with equality at the grid times by
\eqref{eq:grid-coupling-optimal}.  Both quantities are adapted, and
$t\mapsto d_1(\nu_t,\mu^N_t)$ is continuous outside a null set, because
both flows are.

The role of the recoupling is to keep the gap between the provisional and
the optimal cost of order $\sqrt h$, uniformly in time:

\begin{lemma}[Cost of the provisional coupling]\label{lem:provisional}
For every $t\in[t_0,T)$,
\begin{equation}\label{eq:provisional}
 \E\,c_t\;\leq\;\E\,d_1\big(\nu_t,\mu^N_t\big)+C_M\sqrt h .
\end{equation}
\end{lemma}

\begin{proof}
Fix $k$ and $t\in[s_k,s_{k+1})$, and abbreviate $\tau:=t-s_k\leq h$,
$A_i:=A^k_i(t)$, $\Delta W^i:=W^i_t-W^i_{s_k}$.  Since
$\sqrt2\,\Delta W^i$ has density $p_\tau$ on $\T^d$, the moment bound of
Lemma~\ref{lem:heat-kernel} gives
\begin{equation}\label{eq:increment-moment}
 \E\,\dm\big(\sqrt2\,\Delta W^i,0\big)\leq C\sqrt\tau .
\end{equation}
By
\eqref{eq:component-definition}, the identity
$X^i_t=X^i_{s_k}+A_i+\sqrt2\,\Delta W^i$, and the definition of the heat
semigroup,
\[
 \int_{\T^d}\dm\big(y,X^i_t\big)\,\nu_{i,t}(dy)
 =\int_{\T^d}\int_{\T^d}
 \dm\big(z+A_i+u,\;X^i_{s_k}+A_i+\sqrt2\,\Delta W^i\big)\,
 p_\tau(u)\,du\,\pi_{i,k}(dz).
\]
Using the translation invariance of $\dm$ and the triangle inequality,
\[
 \dm\big(z+A_i+u,\,X^i_{s_k}+A_i+\sqrt2\Delta W^i\big)
 \leq\dm\big(z,X^i_{s_k}\big)+\dm(u,0)+\dm\big(\sqrt2\Delta W^i,0\big).
\]
Summing over $i$ (the masses $\pi_{i,k}(\T^d)=1/N$ add up to one), taking
expectations, and using \eqref{eq:grid-coupling-optimal},
the moment bound of Lemma~\ref{lem:heat-kernel}, and
\eqref{eq:increment-moment},
\begin{equation}\label{eq:c-from-grid}
 \E\,c_t
 \leq\E\Big[\sum_{i=1}^N\int\dm\big(z,X^i_{s_k}\big)\pi_{i,k}(dz)\Big]
 +C\sqrt\tau
 =\E\,d_1\big(\nu_{s_k},\mu^N_{s_k}\big)+C\sqrt\tau .
\end{equation}

It remains to return from $s_k$ to time $t$.  By the triangle
inequality for $d_1$,
\[
 d_1\big(\nu_{s_k},\mu^N_{s_k}\big)
 \leq d_1\big(\nu_{s_k},\nu_t\big)+d_1\big(\nu_t,\mu^N_t\big)
 +d_1\big(\mu^N_t,\mu^N_{s_k}\big).
\]
For the first term, transporting each component of $\nu_{s_k}$ to the
corresponding component of $\nu_t$ through $z\mapsto z+A_i+u$, with $u$
averaged over $p_\tau$, gives, by the moment bound of
Lemma~\ref{lem:heat-kernel} and $|A_i|\leq M\tau$,
$\E\,d_1(\nu_{s_k},\nu_t)\leq Mh+C\sqrt h$.  For the third term, matching the
particles by their labels gives
$d_1(\mu^N_t,\mu^N_{s_k})\leq\frac1N\sum_i\dm(X^i_t,X^i_{s_k})$, and thus,
again by \eqref{eq:increment-moment} and $|A_i|\leq M\tau$,
$\E\,d_1(\mu^N_t,\mu^N_{s_k})\leq Mh+C\sqrt h$.  Inserting the two
estimates into the triangle inequality and combining with
\eqref{eq:c-from-grid} proves \eqref{eq:provisional}.
\end{proof}

\begin{remark}\label{rem:no-accumulation}
The estimate \eqref{eq:provisional} compares $c_t$ with
$d_1(\nu_t,\mu^N_t)$ at the same time $t$.  The recoupling resets the provisional cost to the optimal one
at each grid time, so the errors $C_M\sqrt h$ do not add up across the
$(T-t_0)/h$ intervals.  They will enter the final estimates only through
time integrals against integrable kernels, producing a total contribution
$O(\sqrt h)$, which disappears as $h\downarrow0$.
\end{remark}

\begin{remark}\label{rem:continuous-shadow}
As conceived here, the shadow flow is discretely recoupled: the
attachment of the shadow mass to the particles is refreshed only at
the grid times.  One could plausibly have worked instead with the formal limit $h\downarrow0$, in which
the recoupling is instantaneous: a flow $\nu$ starting from
$\mu^N_{t_0}$ and solving
\[
 \partial_t\nu_t=\Delta\nu_t
 -\diver\Big(\sum_{i=1}^N\alpha^i_t\,\pi_{i,t}\Big),
\]
where, for a.e.\ $t$, the family $(\pi_{i,t})_{i=1}^N$ couples $\nu_t$
optimally with $\mu^N_t$.  Because
$\pi_{i,t}\leq\nu_t$, the flux would take the form $\beta_t\nu_t$,
where $\beta_t=\sum_{i=1}^N\frac{d\pi_{i,t}}{d\nu_t}\,\alpha^i_t$ is,
at each point, a convex combination of the realized particle controls
weighted by the optimal attachment.  It is plausible that the
results of this paper could be proved using such a flow. Making this limiting picture rigorous would require a selection of optimal decompositions measurable in time.  Uniqueness of the resulting flow is unclear, and adaptedness would have to be preserved in the common-noise setting. We work with
the discrete construction because it already yields the sharp estimates, while being both concrete and flexible. Indeed, after translation the construction
transfers unchanged to the common-noise setting of
Section~\ref{sec:common-noise}. Its pathwise nature is also useful to the one-dimensional arguments of Section~\ref{sec:one-dim}: since the competitor is constructed separately for each realization, it accommodates anticipating corrections to the component drifts.
\end{remark}

\section{Shadowing of the empirical measure}\label{sec:tracking}

This section proves that, in expectation, the shadow flow follows the
empirical measure at the empirical-measure rate, uniformly in time.
We keep the objects and notation of Section~\ref{sec:shadow}: the grid
\eqref{eq:grid} of mesh $h$, the shadow flow $\nu$ with components
$\nu_{i,t}$ and drift $\beta$ \eqref{eq:beta-definition}, and the
provisional cost $c_t$ \eqref{eq:c-definition}; $P_\tau$ denotes the
heat semigroup, with kernel $p_\tau$ \eqref{eq:heat-kernel-def}.

\begin{proposition}[Sharp shadowing of the empirical measure]\label{prop:tracking}
The shadow flow $\nu$ of Section~\ref{sec:shadow} satisfies
\begin{equation}\label{eq:tracking}
 \sup_{t_0\leq t\leq T}\E\,d_1\big(\nu_t,\mu^N_t\big)
 \leq C_M\big(r_{N,d}+\sqrt h\big).
\end{equation}
\end{proposition}

The proof treats the difference of the two flows, empirical and
shadow, as a solution of a stochastic heat
equation with two forcings, and estimates it through Duhamel's
formula: the flux mismatch between particles and components is bounded
by the provisional cost $c_t$, through the flux estimate,
Lemma~\ref{lem:flux}, while the discarded idiosyncratic noise enters
as a martingale and is bounded at the scale $N^{-1/2}$, through the
noise estimate, Lemma~\ref{lem:martingale}.  A heat smoothing at a scale
$\varepsilon>0$, at the price $C\sqrt\varepsilon$, gives the necessary regularity. A weakly singular
Gr\"onwall inequality then closes the estimate, and $\varepsilon$ is chosen optimally
to balance the mollification against the noise.

\subsection{The shadowing error and its Duhamel representation}
We define the shadowing error and the flux mismatch by
\begin{equation}\label{eq:z-r-definition}
 z_t:=\mu^N_t-\nu_t,
 \qquad
 \mathsf{r}_t:=\sum_{i=1}^N\alpha^i_t
 \Big(\frac1N\delta_{X^i_t}-\nu_{i,t}\Big),
 \qquad t\in\bigcup_k(s_k,s_{k+1}),
\end{equation}
with $\mathsf{r}_t:=0$ at the (Lebesgue-null) set of grid times.
Note that each summand of $\mathsf{r}_t$ is $\alpha^i_t$ times a zero-mass measure, and
$|\mathsf{r}_t|\leq M(\mu^N_t+\nu_t)$ in the sense of total variation, so $\mathsf{r}_t$ is a
finite vector measure with $\mathsf{r}_t(\T^d)=0$.

By It\^o's formula for the particles and Lemma~\ref{lem:FP} for the
shadow flow, the shadowing error solves, formally, the
stochastic heat equation
\begin{equation}\label{eq:z-evolution}
 dz_t=\Delta z_t\,dt-\diver\mathsf{r}_t\,dt
 -\frac{\sqrt2}{N}\sum_{i=1}^N\diver\big(\delta_{X^i_t}\,dW^i_t\big),
 \qquad z_{t_0}=0,
\end{equation}
driven by the flux mismatch and by the discarded noise, the initial
condition being \eqref{eq:shadow-initial}.  The next lemma is a
rigorous version of \eqref{eq:z-evolution}
(cf.\ \cite[Lem.~3.1]{BechtoldCoppini2021}).

\begin{lemma}[Duhamel representation for the shadowing error]\label{lem:duhamel}
For every $t\in(t_0,T]$ and every $\varphi\in C^\infty(\T^d)$, almost
surely,
\begin{equation}\label{eq:duhamel}
 \int_{\T^d}\varphi\,dz_t
 =\int_{t_0}^t\int_{\T^d}P_{t-s}\nabla\varphi\cdot d\mathsf{r}_s\,ds
 +\frac{\sqrt2}{N}\sum_{i=1}^N\int_{t_0}^t
 \big(P_{t-s}\nabla\varphi\big)\big(X^i_s\big)\cdot dW^i_s .
\end{equation}
\end{lemma}

\begin{proof}
Fix $t\in(t_0,T]$ and $\varphi\in C^\infty(\T^d)$, and set
$\varphi_s:=P_{t-s}\varphi$ for $s\in[t_0,t]$, which solves the
backward heat equation $\partial_s\varphi_s=-\Delta\varphi_s$.

On the particle side, recalling \eqref{eq:particle-SDE}, It\^o's formula gives
\begin{multline*}
 d\varphi_s\big(X^i_s\big)
 =\big(\partial_s\varphi_s+\Delta\varphi_s
 +\alpha^i_s\cdot\nabla\varphi_s\big)\big(X^i_s\big)\,ds
 +\sqrt2\,\nabla\varphi_s\big(X^i_s\big)\cdot dW^i_s
 \\= \big(\alpha^i_s\cdot\nabla\varphi_s\big)\big(X^i_s\big)\,ds
 +\sqrt2\,\nabla\varphi_s\big(X^i_s\big)\cdot dW^i_s .
\end{multline*}
Integrating over $[t_0,t]$ and averaging over
$i$ yields
\begin{equation}\label{eq:particle-identity}
 \int_{\T^d}\varphi\,d\mu^N_t
 -\int_{\T^d}P_{t-t_0}\varphi\,d\mu^N_{t_0}
 =\int_{t_0}^t\frac1N\sum_{i=1}^N
 \alpha^i_s\cdot\nabla\varphi_s\big(X^i_s\big)\,ds
 +\frac{\sqrt2}{N}\sum_{i=1}^N\int_{t_0}^t
 \nabla\varphi_s(X^i_s)\cdot dW^i_s .
\end{equation}

On the shadow side, testing the
distributional equation \eqref{eq:shadow-FP} against $\varphi_s$, and using the continuity of $s\mapsto\nu_s$, we get
\begin{multline}\label{eq:shadow-identity}
 \int_{\T^d}\varphi\,d\nu_t-\int_{\T^d}P_{t-t_0}\varphi\,d\nu_{t_0}
 =\int_{t_0}^t\int_{\T^d}\big(\partial_s\varphi_s+\Delta\varphi_s
 +\beta_s\cdot\nabla\varphi_s\big)\,d\nu_s\,ds
 \\=\int_{t_0}^t\int_{\T^d}\beta_s\cdot\nabla\varphi_s\,d\nu_s\,ds .
\end{multline}

Subtracting \eqref{eq:shadow-identity} from \eqref{eq:particle-identity},
the initial terms cancel by \eqref{eq:shadow-initial}.
The corresponding flux difference is $\mathsf{r}_s$ by
\eqref{eq:beta-definition} and \eqref{eq:z-r-definition}. Hence
\[
 \int_{\T^d}\varphi\,dz_t
 =\int_{t_0}^t\int_{\T^d}\nabla\varphi_s\cdot d\mathsf{r}_s\,ds
 +\frac{\sqrt2}{N}\sum_{i=1}^N\int_{t_0}^t
 \nabla\varphi_s(X^i_s)\cdot dW^i_s ,
\]
which is exactly \eqref{eq:duhamel}, since
$\nabla\varphi_s=P_{t-s}\nabla\varphi$.
\end{proof}

Fix $\varepsilon>0$ and define
\begin{equation}\label{eq:martingale-definition}
 D^\varepsilon_t:=\int_{t_0}^tP_{t-s+\varepsilon}\,\mathsf{r}_s\,ds,
 \qquad
 M^\varepsilon_t(y):=\frac{\sqrt2}{N}\sum_{i=1}^N\int_{t_0}^t
 p_{t-s+\varepsilon}\big(y-X^i_s\big)\,dW^i_s .
\end{equation}
These are the drift and noise terms of the representation
\eqref{eq:duhamel}, each smoothed by $P_\varepsilon$.  Accordingly,
$D^\varepsilon_t$ is a smooth $\R^d$-valued function on $\T^d$, and
$M^\varepsilon_t$ has a jointly measurable, $y$-continuous version
(Lemma~\ref{lem:noise-version}), which we fix. 

The price of this heat smoothing is an error $C\sqrt\varepsilon$, and will be paid in the proof of
Proposition~\ref{prop:tracking}, through the following simple consequence of \eqref{eq:duhamel}.

\begin{lemma}[Mollified Duhamel representation]\label{lem:mollified-duhamel}
For every $\varepsilon>0$ and $t\in(t_0,T]$, almost surely,
simultaneously for all Lipschitz $\phi:\T^d\to\R$,
\begin{equation}\label{eq:mollified-duhamel}
 \int_{\T^d}P_\varepsilon\phi\,dz_t
 =\int_{\T^d}\big(D^\varepsilon_t+M^\varepsilon_t\big)\cdot\nabla\phi\,dy .
\end{equation}
\end{lemma}

\begin{proof}
Fix $\varepsilon>0$ and apply \eqref{eq:duhamel} with
$\varphi=P_\varepsilon\phi$, $\phi\in C^\infty(\T^d)$.  Every heat
kernel below is then evaluated at time at least $\varepsilon$, which is
what legitimizes the following interchanges.  In the drift
term of \eqref{eq:duhamel}, the self-adjointness of the heat semigroup,
Fubini's theorem, and the definition \eqref{eq:martingale-definition}
of $D^\varepsilon_t$ give
\[
 \int_{t_0}^t\int_{\T^d}P_{t-s+\varepsilon}\nabla\phi\cdot d\mathsf{r}_s\,ds
 =\int_{t_0}^t\int_{\T^d}\big(P_{t-s+\varepsilon}\,\mathsf{r}_s\big)
 \cdot\nabla\phi\,dy\,ds
 =\int_{\T^d}D^\varepsilon_t\cdot\nabla\phi\,dy .
\]
In the stochastic term,
since the heat kernel is even,
\[
 \big(P_{t-s+\varepsilon}\nabla\phi\big)\big(X^i_s\big)
 =\int_{\T^d}p_{t-s+\varepsilon}\big(y-X^i_s\big)\,\nabla\phi(y)\,dy ,
\]
and the stochastic Fubini theorem (the integrand is bounded and
progressively measurable), together with the definition
\eqref{eq:martingale-definition} of $M^\varepsilon_t$, gives
\[
 \frac{\sqrt2}{N}\sum_{i=1}^N\int_{t_0}^t
 \big(P_{t-s+\varepsilon}\nabla\phi\big)\big(X^i_s\big)\cdot dW^i_s
 =\int_{\T^d}M^\varepsilon_t(y)\cdot\nabla\phi(y)\,dy .
\]
This proves \eqref{eq:mollified-duhamel} for each fixed
$\phi\in C^\infty(\T^d)$, almost surely.  Since both sides are
continuous under uniform convergence of $\phi$ and $\nabla\phi$, a
countable dense family of test functions upgrades this to a single
null set for all smooth $\phi$.  On the same event, the identity
extends to every Lipschitz $\phi$ by a standard approximation argument. 
\end{proof}

\subsection{The drift and noise estimates}
Our goal is now to control the regularized drift and noise terms $D^{\varepsilon}$ and $M^{\varepsilon}$. For the drift term $D^{\varepsilon}_t$, we estimate below its integrand, namely the flux mismatch regularized by the heat flow, in terms of the provisional coupling cost.
\begin{lemma}[Flux estimate]\label{lem:flux}
For every $\tau>0$ and $t\in[t_0,T)$,
\begin{equation*}
 \big\|P_\tau\mathsf{r}_t\big\|_{L^1(\T^d)}\leq CM\,\tau^{-1/2}\,c_t .
\end{equation*}
\end{lemma}

\begin{proof}
Since $\nu_{i,t}(\T^d)=1/N$, the $i$-th summand of $P_\tau\mathsf{r}_t$ is
\[
 P_\tau\Big(\alpha^i_t\Big(\frac1N\delta_{X^i_t}-\nu_{i,t}\Big)\Big)(y)
 =\alpha^i_t\int_{\T^d}\Big(p_\tau\big(y-X^i_t\big)-p_\tau(y-x)\Big)\,
 \nu_{i,t}(dx).
\]
Joining $x$ to $X^i_t$ by a geodesic and integrating $\nabla p_\tau$
along it, the gradient bound of Lemma~\ref{lem:heat-kernel} gives
$\|p_\tau(\cdot-X^i_t)-p_\tau(\cdot-x)\|_{L^1}\leq
\dm\big(X^i_t,x\big)\,\|\nabla p_\tau\|_{L^1}\leq
C\tau^{-1/2}\dm\big(X^i_t,x\big)$, so
\[
 \big\|P_\tau\mathsf{r}_t\big\|_{L^1}
 \leq M\sum_{i=1}^N\int_{\T^d}
 C\tau^{-1/2}\,\dm\big(X^i_t,x\big)\,\nu_{i,t}(dx)
 =CM\,\tau^{-1/2}c_t,
\]
by the definition \eqref{eq:c-definition} of $c_t$.
\end{proof}

Our next estimate controls the stochastic forcing. Set, for $0<\varepsilon\leq1$,
\begin{equation}\label{eq:theta-definition}
 \Theta_d(\varepsilon):=
 \begin{cases}
 \,1, & d=1,\\
 \,\sqrt{\log\big(1+\varepsilon^{-1}\big)}, & d=2,\\
 \,\varepsilon^{-(d-2)/4}, & d\geq3 .
 \end{cases}
\end{equation}
 It is here that the dimension enters the proof, and the trichotomy above is the source of the three regimes of the empirical-measure
rate \eqref{eq:RNd}.  The trichotomy is the dynamic counterpart of the
heat-kernel computation behind the static matching rates
\cite[Prop.~2]{BobkovLedoux2021}, with the Fourier mass generated here
by a stochastic convolution rather than by sampling.  Beyond the independence of the driving Brownian motions, neither the
mutual dependence of the particle states induced by the controls nor the
adaptedness structure of the controls plays any role in the proof below: only progressive
measurability is used.

\begin{lemma}[Noise estimate]\label{lem:martingale}
For every $\varepsilon\in(0,1]$ and $t\in[t_0,T]$,
\begin{equation}\label{eq:martingale-bound}
 \E\,\big\|M^\varepsilon_t\big\|_{L^1(\T^d)}
 \leq\frac{C}{\sqrt N}\,\Theta_d(\varepsilon).
\end{equation}
\end{lemma}

\begin{proof}
Fix $y\in\T^d$.  The integrands in \eqref{eq:martingale-definition} are
bounded and progressively measurable, and the Brownian motions
$(W^i)_{i=1}^N$ are independent, so the It\^o isometry gives 
\[
 \E\big|M^\varepsilon_t(y)\big|^2
 =\frac{2d}{N^2}\sum_{i=1}^N\int_{t_0}^t
 \E\Big[p^2_{t-s+\varepsilon}\big(y-X^i_s\big)\Big]\,ds .
\]
We then have
\begin{equation}\label{eq:martingale-L1}
 \E\,\big\|M^\varepsilon_t\big\|_{L^1}
 =\int_{\T^d}\E\big|M^\varepsilon_t(y)\big|\,dy
 \leq\Big(\int_{\T^d}\E\big|M^\varepsilon_t(y)\big|^2dy\Big)^{1/2}
 =\Big(\frac{2d}{N}\int_0^{t-t_0}
 \big\|p_{u+\varepsilon}\big\|_{L^2}^2\,du\Big)^{1/2},
\end{equation}
where the last equality uses that
$\int_{\T^d}p^2_{t-s+\varepsilon}(y-X^i_s)\,dy
=\|p_{t-s+\varepsilon}\|_{L^2}^2$.  By
the $L^2$ bound of Lemma~\ref{lem:heat-kernel},
\begin{equation}\label{eq:kernel-time-integral}
 \int_0^{T}\big\|p_{u+\varepsilon}\big\|_{L^2}^2\,du
 \leq C\int_0^{T}\Big(1+(u+\varepsilon)^{-d/2}\Big)\,du
 \leq C\,\Theta_d(\varepsilon)^2.
\end{equation}
Combining \eqref{eq:martingale-L1} and
\eqref{eq:kernel-time-integral} proves \eqref{eq:martingale-bound}.
\end{proof}

\subsection{Proof of Proposition~\ref{prop:tracking}}

\begin{proof}
Set $f(t):=\E\,d_1(\nu_t,\mu^N_t)$. Note that $f$ is bounded by
$\operatorname{diam}(\T^d)$, and it is continuous because
$t\mapsto d_1(\nu_t,\mu^N_t)$ is a.s.\ continuous and uniformly bounded.
Fix $\varepsilon\in(0,1]$, to be chosen at the end.

A $1$-Lipschitz $\phi$ satisfies
$\|\phi-P_\varepsilon\phi\|_{L^\infty}\leq C\sqrt\varepsilon$, by the
moment bound of Lemma~\ref{lem:heat-kernel}, so
\eqref{eq:mollified-duhamel}, $|z_t|(\T^d)\leq2$, and
$\|\nabla\phi\|_{L^\infty}\leq1$ give
\begin{multline*}
 \int_{\T^d}\phi\,dz_t
 =\int_{\T^d}(\phi-P_\varepsilon\phi)\,dz_t
 +\int_{\T^d}\big(D^\varepsilon_t+M^\varepsilon_t\big)\cdot\nabla\phi\,dy
 \\ \leq C\sqrt\varepsilon
 +\big\|D^\varepsilon_t\big\|_{L^1}+\big\|M^\varepsilon_t\big\|_{L^1}.
\end{multline*}
Taking the supremum in the duality \eqref{eq:d1-duality} yields, almost
surely,
\begin{equation}\label{eq:mollified-duality}
 d_1\big(\mu^N_t,\nu_t\big)
 \leq C\sqrt\varepsilon
 +\big\|D^\varepsilon_t\big\|_{L^1}+\big\|M^\varepsilon_t\big\|_{L^1}.
\end{equation}

By Lemma~\ref{lem:flux} with $\tau=t-s+\varepsilon$,
\[
 \big\|D^\varepsilon_t\big\|_{L^1}
 \leq\int_{t_0}^t\big\|P_{t-s+\varepsilon}\,\mathsf{r}_s\big\|_{L^1}\,ds
 \leq C_M\int_{t_0}^t(t-s)^{-1/2}\,c_s\,ds .
\]
Taking expectations in \eqref{eq:mollified-duality} and applying
Lemma~\ref{lem:martingale},
\[
 f(t)\leq C\sqrt\varepsilon+\frac{C}{\sqrt N}\,\Theta_d(\varepsilon)
 +C_M\int_{t_0}^t(t-s)^{-1/2}\,\E c_s\,ds .
\]
By Lemma~\ref{lem:provisional}, $\E c_s\leq f(s)+C_M\sqrt h$, and
$\int_{t_0}^t(t-s)^{-1/2}ds\leq2\sqrt T$, so
\begin{equation}\label{eq:pre-gronwall}
 f(t)\leq C\sqrt\varepsilon+\frac{C}{\sqrt N}\,\Theta_d(\varepsilon)
 +C_M\sqrt h
 +C_M\int_{t_0}^t(t-s)^{-1/2}f(s)\,ds,
\end{equation}
and the weakly singular Gr\"onwall inequality
\cite[Lem.~7.1.1]{Henry1981} (which readily reduces to conventional Gronwall by substituting \eqref{eq:pre-gronwall} into itself) yields
\begin{equation}\label{eq:after-gronwall}
 \sup_{t_0\leq t\leq T}f(t)
 \leq C_M\Big(\sqrt\varepsilon
 +\frac{\Theta_d(\varepsilon)}{\sqrt N}+\sqrt h\Big).
\end{equation}
It remains to choose $\varepsilon$: with $\varepsilon:=N^{-2/d}$ for
$d\geq3$ and $\varepsilon:=N^{-1}$ for $d\in\{1,2\}$,
\[
 \sqrt\varepsilon+\frac{\Theta_d(\varepsilon)}{\sqrt N}\leq C\,r_{N,d}
\]
in every dimension, and \eqref{eq:after-gronwall} becomes
\eqref{eq:tracking}.
\end{proof}

\begin{remark}\label{rem:critical}
The factor $(t-s)^{-1/2}$ in the proof above is critical in the
following sense: the estimate must absorb two derivatives, and the heat
kernel prices each of them at $(t-s)^{-1/2}$, so it can repay one, while
two would cost the non-integrable $(t-s)^{-1}$.  The first derivative
is the divergence through which the flux mismatch drives
\eqref{eq:z-evolution}.  The Duhamel representation \eqref{eq:duhamel}
absorbs it for free, by placing it on the test function, whose gradient
is bounded in the duality \eqref{eq:d1-duality}.  The second is the one
the kernel repays: each summand of $\mathsf r_s$ is a zero-mass
difference between a particle and its component, so its pairing with
any kernel is controlled by the gradient of the kernel times their
coupling distance, in total $c_s$.  This is the flux estimate,
Lemma~\ref{lem:flux}, at the integrable price $(t-s)^{-1/2}$.  This use of parabolic smoothing contrasts with the value-function
comparison in \cite{DaudinDelarueJackson2024}: there the particle
Laplacian produces a correction involving a second derivative in the
measure variable, which must be manufactured by regularization at the
price of additional approximation errors.
\end{remark}

\section{The shadow flow as a mean field competitor}\label{sec:comparison}

In this section, we take advantage of the preceding shadowing estimate to turn the shadow flow into a near-optimal mean field competitor, proving Proposition~\ref{prop:comparison} and, consequently, Theorem~\ref{thm:main}.
We keep the objects and notation of Sections~\ref{sec:shadow}
and~\ref{sec:tracking}: the grid \eqref{eq:grid} of mesh $h$, the
control $\balpha$, bounded by $M$, the shadow flow $\nu$ with
components $\nu_{i,t}$, weights $\theta_i$, and drift $\beta$
\eqref{eq:beta-definition}, and the provisional cost $c_t$
\eqref{eq:c-definition}.  The basic properties of the Lagrangian $L$ used below are collected in
Lemma~\ref{lem:lagrangian}.

\begin{lemma}[Running cost estimate]\label{lem:action}
Outside a $\Prob$-null set, for a.e.\ $t\in(t_0,T)$,
\begin{equation}\label{eq:action-comparison}
 \int_{\T^d}L\big(y,\beta_t(y)\big)\,\nu_t(dy)
 \;\leq\;\frac1N\sum_{i=1}^NL\big(X^i_t,\alpha^i_t\big)+C_M\,c_t .
\end{equation}
\end{lemma}

\begin{proof}
Fix $\omega$ outside the null set of Lemma~\ref{lem:FP} and
$t\in(s_k,s_{k+1})$.  By the convexity of $L(y,\cdot)$
(Lemma~\ref{lem:lagrangian}(ii)) and Jensen's inequality applied to
the weights $\theta_i(t,y)$ of \eqref{eq:beta-definition},
\[
 \int_{\T^d}L\big(y,\beta_t(y)\big)\,\nu_t(dy)
 \leq\sum_{i=1}^N\int_{\T^d}L\big(y,\alpha^i_t\big)\,
 \theta_i(t,y)\,\nu_t(dy)
 =\sum_{i=1}^N\int_{\T^d}L\big(y,\alpha^i_t\big)\,\nu_{i,t}(dy),
\]
since $\theta_i\,d\nu_t=d\nu_{i,t}$.  By
Lemma~\ref{lem:lagrangian}(iii) and $|\alpha^i_t|\leq M$, the function
$L(\cdot,\alpha^i_t)$ is $C_M$-Lipschitz, so, by
$\nu_{i,t}(\T^d)=1/N$,
\[
 \int_{\T^d}L\big(y,\alpha^i_t\big)\,\nu_{i,t}(dy)
 \leq\frac1NL\big(X^i_t,\alpha^i_t\big)
 +C_M\int_{\T^d}\dm\big(y,X^i_t\big)\,\nu_{i,t}(dy),
\]
and summing over $i$ gives \eqref{eq:action-comparison} by
\eqref{eq:c-definition}.
\end{proof}

\begin{proof}[Proof of Proposition~\ref{prop:comparison}]
By Lemma~\ref{lem:FP}, for each $\omega$ outside a $\Prob$-null set the pair
$(\nu^\omega,\beta^\omega)$ is admissible in \eqref{eq:U-definition} with
initial condition $\nu_{t_0}=m^N_{\bx}$.  Therefore, pointwise in
$\omega$,
\begin{equation}\label{eq:pathwise-bound}
 U(t_0,m^N_{\bx})
 \;\leq\;
 \int_{t_0}^T\Big(\int_{\T^d}L\big(y,\beta^\omega_t(y)\big)\nu^\omega_t(dy)
 +F(\nu^\omega_t)\Big)dt+G(\nu^\omega_T).
\end{equation}
By the $d_1$-Lipschitz continuity of $F$ and $G$,
\begin{equation}\label{eq:FG-comparison}
 F(\nu_t)\leq F(\mu^N_t)+L_F\,d_1\big(\nu_t,\mu^N_t\big),
 \qquad
 G(\nu_T)\leq G(\mu^N_T)+L_G\,d_1\big(\nu_T,\mu^N_T\big).
\end{equation}
Taking
expectations in \eqref{eq:pathwise-bound} and using
\eqref{eq:FG-comparison}, Lemmas~\ref{lem:action},
\ref{lem:provisional} and Proposition~\ref{prop:tracking},
\begin{align*}
 U(t_0,m^N_{\bx})
 &\leq\E\bigg[\int_{t_0}^T\Big(\frac1N\sum_{i=1}^NL(X^i_t,\alpha^i_t)
 +F(\mu^N_t)\Big)dt+G(\mu^N_T)\bigg]
 \\
 &\qquad
 +C_M\int_{t_0}^T\big(\E c_t+\E\,d_1(\nu_t,\mu^N_t)\big)\,dt
 +L_G\,\E\,d_1\big(\nu_T,\mu^N_T\big)
 \\
 &\leq J_N(t_0,\bx;\balpha)
 +C_M\Big(\sqrt h
 +\sup_{t_0\leq t\leq T}\E\,d_1\big(\nu_t,\mu^N_t\big)\Big)
 \\
 &\leq J_N(t_0,\bx;\balpha)+C_M\big(r_{N,d}+\sqrt h\big).
\end{align*}
The left-hand side does not depend on the partition.  Letting $h\downarrow0$
in \eqref{eq:grid} proves \eqref{eq:comparison}.
\end{proof}

Proposition~\ref{prop:comparison} applies to bounded controls.  To reach the
value function $V^N$ we approximate $F$ and $G$ by costs whose associated
finite-dimensional problems admit optimal Markov feedbacks that are bounded
uniformly in $n$ and $N$.  The required smoothing procedure is that of
\cite{DaudinDelarueJackson2024}.

\begin{lemma}[Smoothing; {\cite[Lem.~4.1 and
App.~B]{DaudinDelarueJackson2024}}]\label{lem:smoothing}
Let $\Phi:\Pp(\T^d)\to\R$ be $\ell$-Lipschitz continuous with respect to
$d_1$.  There exists a sequence $(\Phi^n)_{n\geq1}$ of functions
$\Pp(\T^d)\to\R$ such that:
\begin{enumerate}[label=\textup{(\alph*)},leftmargin=2.6em]
\item $\Phi^n\to\Phi$ uniformly on $\Pp(\T^d)$;
\item each $\Phi^n$ is $\ell$-Lipschitz continuous with respect to $d_1$;
\item each $\Phi^n$ depends smoothly on finitely many Fourier
coefficients of $m$.  In particular, for every $N\geq1$, the map
$(\T^d)^N\ni\bx\mapsto\Phi^n(m^N_{\bx})$ is of class $C^\infty$.
\end{enumerate}
\end{lemma}

\begin{proof}
Although \cite[Lem.~4.1]{DaudinDelarueJackson2024} is stated under the
additional assumption that $\Phi$ is $d_1$-semiconcave, semiconcavity is used
in \cite[App.~B]{DaudinDelarueJackson2024} only to show that the approximants
remain semiconcave.  The same construction therefore gives (a) and (b), and
Step~2 there gives the first assertion in (c), with Fourier coefficients
$(\widehat m(k))_{|k|_\infty<n}$.  Since each
$\widehat{m^N_{\bx}}(k)=\frac1N\sum_ie^{-2\pi ik\cdot x^i}$ is a smooth
function of $\bx$, so is $\bx\mapsto\Phi^n(m^N_{\bx})$.
\end{proof}

Fix such approximations $F^n$, $G^n$ of $F$, $G$ (with constants $L_F$,
$L_G$), and let $U^n$, $V^{N,n}$, $J^n_N$ denote the value functions and costs
\eqref{eq:JN-definition}, \eqref{eq:U-definition} with $(F,G)$ replaced by
$(F^n,G^n)$.

\begin{lemma}[The smoothed finite-dimensional problems]\label{lem:finite}
Let Assumption~\ref{ass:standing} hold with $(F,G)$ replaced by data
$(\widetilde F,\widetilde G)$ that are, respectively, $L_F$- and
$L_G$-Lipschitz for $d_1$ and such that
$\bx\mapsto\widetilde F(m^N_{\bx}),\ \widetilde G(m^N_{\bx})$ are $C^\infty$.
Denote by $\widetilde V^N$ the associated value function
\eqref{eq:JN-definition}.  There exist constants $C_*$ and
$M_*$, depending only on $d$, $T$, $H$, $L_F$, and $L_G$, such that:
\begin{enumerate}[label=\textup{(\alph*)},leftmargin=2.6em]
\item $\widetilde V^N$ is the unique classical solution in
$C^{1,2}([0,T]\times(\T^d)^N)$ of the Hamilton--Jacobi--Bellman equation
\begin{equation}\label{eq:HJBN}
 -\partial_t\widetilde V^N-\sum_{i=1}^N\Delta_{x^i}\widetilde V^N
 +\frac1N\sum_{i=1}^NH\big(x^i,\,N D_{x^i}\widetilde V^N\big)
 =\widetilde F\big(m^N_{\bx}\big),
 \qquad
 \widetilde V^N(T,\cdot)=\widetilde G\big(m^N_{\cdot}\big);
\end{equation}
\item for every $i$ and $t$,
\begin{equation}\label{eq:gradient-bound}
 \big\|D_{x^i}\widetilde V^N(t,\cdot)\big\|_{L^\infty((\T^d)^N)}
 \leq\frac{C_*}{N};
\end{equation}
\item for every $(t_0,\bx)$, the Markov feedback control
\begin{equation}\label{eq:optimal-feedback}
 \alpha^{*,i}_t
 :=-D_pH\Big(X^i_t,\,ND_{x^i}\widetilde V^N\big(t,\bX_t\big)\Big),
 \qquad i=1,\dots,N,
\end{equation}
is well defined (the closed-loop system, \eqref{eq:particle-SDE} with
this feedback as drift, admits a strong solution), admissible,
optimal, in the sense that
$J_N(t_0,\bx;\balpha^*)=\widetilde V^N(t_0,\bx)$ with the costs
$(\widetilde F,\widetilde G)$, and satisfies
\begin{equation}\label{eq:feedback-bound}
 |\alpha^{*,i}_t|\leq M_*
 \qquad\text{for all }t,\ i .
\end{equation}
\end{enumerate}
\end{lemma}

Lemma~\ref{lem:finite} collects classical facts.
Equation~\eqref{eq:HJBN} is a semilinear, uniformly parabolic equation on
the compact manifold $(\T^d)^N$ with smooth data, and the existence and
uniqueness of a classical solution with bounded, continuous gradient are
part of classical parabolic theory
\cite[Chap.~V]{LadyzhenskayaSolonnikovUraltseva1968}, \cite{Krylov1980},
\cite[Chap.~IV]{FlemingSoner2006}.  The gradient estimate in part~(b) is
\cite[Lem.~6.1]{DaudinDelarueJackson2024}, whose proof is that of
\cite[Lem.~3.1]{CardaliaguetDaudinJacksonSouganidis2023}.  Part~(c) is the
standard verification argument: the feedback \eqref{eq:optimal-feedback} is
bounded by \eqref{eq:gradient-bound} and Lemma~\ref{lem:lagrangian}(iv),
the closed-loop system then admits a strong solution by
\cite[Thm.~1]{Veretennikov1981}, and It\^o's formula with the equality case
of the Legendre transform (Lemma~\ref{lem:lagrangian}(ii)) shows
optimality.

\begin{lemma}[Stability]\label{lem:stability}
Let $(F_1,G_1)$ and $(F_2,G_2)$ be bounded measurable costs.  Then the
associated value functions satisfy, for all $N$, $t$, $\bx$, $m$,
\[
 \big|V^{N}_1-V^{N}_2\big|(t,\bx)\;\vee\;\big|U_1-U_2\big|(t,m)
 \;\leq\;T\,\|F_1-F_2\|_{\infty}+\|G_1-G_2\|_{\infty},
\]
with $\|\cdot\|_\infty$ the supremum norm on $\Pp(\T^d)$.
\end{lemma}

\begin{proof}
For any fixed admissible control (particle or Fokker--Planck), the two costs
differ by at most $T\|F_1-F_2\|_{\infty}+\|G_1-G_2\|_{\infty}$, so the claim follows by taking infima.
\end{proof}

\begin{proof}[Proof of Theorem~\ref{thm:main}]
Fix $(t,\bx)$ and $n\geq1$.  By Lemma~\ref{lem:finite} applied to
$(\widetilde F,\widetilde G)=(F^n,G^n)$, the optimal control
$\balpha^*=\balpha^{*,N,n}$ of the smoothed $N$-particle problem satisfies
$|\alpha^{*,i}|\leq M_*$ with $M_*$ independent of $n$ and $N$.  Applying
Proposition~\ref{prop:comparison} to this control, with the data
$(H,F^n,G^n)$, which
satisfy Assumption~\ref{ass:standing} with the same constants, we
obtain
\begin{equation*}
 U^n(t,m^N_{\bx})
 \leq J^n_N(t,\bx;\balpha^{*})+C\,r_{N,d}
 =V^{N,n}(t,\bx)+C\,r_{N,d},
\end{equation*}
with $C=C_{M_*}$ independent of $n$ and $N$.  Setting
$\varepsilon_n:=T\|F^n-F\|_{\infty}+\|G^n-G\|_{\infty}$,
Lemma~\ref{lem:stability}
gives $U(t,m^N_{\bx})\leq U^n(t,m^N_{\bx})+\varepsilon_n$ and
$V^{N,n}(t,\bx)\leq V^N(t,\bx)+\varepsilon_n$, so that
\[
 U(t,m^N_{\bx})\leq V^N(t,\bx)+C\,r_{N,d}+2\varepsilon_n .
\]
Letting $n\to\infty$ (Lemma~\ref{lem:smoothing}(a)) proves
\eqref{eq:hard-inequality}.  The easy inequality,
$V^N(t,\bx)\leq U(t,m^N_{\bx})+C\,r_{N,d}$, is
\cite[Prop.~6.2]{DaudinDelarueJackson2024}, up to the rate at $d=2$:
the proof there replicates a near-optimal mean field control across
the $N$ independent noises, pays the expected $d_1$-distance between
the empirical measure of the resulting independent particles and
their common law, and bills it by the estimate of
\cite{FournierGuillin2015}, which at $d=2$ carries the suboptimal factor
$\log(1+N)$ in place of $\sqrt{\log(1+N)}$.  The rate enters only
through that last estimate, so substituting the optimal matching bound of
\cite[Thm.~3]{BobkovLedoux2021} yields $r_{N,d}$.  Combining with
\eqref{eq:hard-inequality} proves \eqref{eq:two-sided}.

\emph{Optimality of the rate for $d\geq3$.}  We take
$H(x,p)=\frac12|p|^2$, $F=0$, and $G(m)=d_1(m,P_T\delta_0)$, which
satisfy Assumption~\ref{ass:standing}, and the initial configuration
$\bx=(0,\dots,0)$, for which $m^N_{\bx}=\delta_0$.  The heat flow
$(P_t\delta_0)_{0\leq t\leq T}$ with zero drift is admissible in
\eqref{eq:U-definition} and has zero cost, so $U(0,\delta_0)=0$.  The
torus kernel is a sum of Gaussian images, so
$p_T\geq(4\pi T)^{-d/2}e^{-d/(16T)}=:\kappa>0$ pointwise, and every
realized empirical measure $\mu^N_T$ is supported on at most $N$
points, so Lemma~\ref{lem:quantization} gives, pathwise,
$G(\mu^N_T)\geq c\,N^{-1/d}$ with $c=c(d,\kappa)$.  Since $L\geq0$ and
$F=0$, every admissible particle control satisfies
$J_N(0,\bx;\balpha)\geq c\,N^{-1/d}$, and so
\[
 V^N(0,\bx)-U(0,m^N_{\bx})\geq c\,N^{-1/d}
 \qquad\text{for every }N\geq1 .
\]
For $d\geq3$ we have $N^{-1/d}=r_{N,d}$, so \eqref{eq:two-sided}
cannot hold, under Assumption~\ref{ass:standing}, with $r_{N,d}$
replaced by any sequence of smaller order.

\emph{Optimality of the rate for $d=2$.}  The example above misses the
logarithm, and we use instead the dynamical example of
Appendix~\ref{app:two-dim}.  We take the data
\eqref{eq:two-dim-data} (quadratic Hamiltonian, running cost the
$d_1$-distance to the uniform measure $\lambda$ on $\T^2$, no terminal
cost) and, for $N=R^2$, the grid configuration $\bx_R\in(\T^2)^N$
whose coordinates are the centers of the $R^2$ squares of side
$R^{-1}$ partitioning $\T^2$.  Transporting each square to its center
gives
$d_1(m^N_{\bx_R},\lambda)\leq CN^{-1/2}$.  The heat flow
$(P_tm^N_{\bx_R})_{0\leq t\leq T}$ with zero drift is admissible in
\eqref{eq:U-definition} with zero action, and the $d_1$-contractivity
of the heat semigroup gives
\[
 U\big(0,m^N_{\bx_R}\big)
 \leq\int_0^Td_1\big(P_tm^N_{\bx_R},\lambda\big)\,dt
 \leq CTN^{-1/2},
\]
while Proposition~\ref{prop:two-dim-lower} gives
$V^N(0,\bx_R)\geq c\,r_{N,2}$ for all large $R$.  Since
$N^{-1/2}=o(r_{N,2})$, the difference of the two values exceeds
$\tfrac c2\,r_{N,2}$ along the squares $N=R^2$, and this single
sequence of configurations already forbids \eqref{eq:two-sided}, under
Assumption~\ref{ass:standing}, with $r_{N,2}$ replaced by any sequence
of smaller order.
\end{proof}

\section{The exceptional one-dimensional rate}\label{sec:one-dim}

Throughout this section $d=1$: we write $\T$ for the circle, $\lambda$ for
the uniform measure on $\T$, and we let Assumption~\ref{ass:standing} be
in force.  The section proves Theorem~\ref{thm:one-dim}.

In Section~\ref{subsec:corrected}, we obtain the hard inequality by correcting the shadow flow
with a filter built on the future of the discarded noise.  For the
 formerly easy inequality we construct, in Section~\ref{subsec:gibbs}, cooperating
near-optimal particles from a Gibbs law at the cost
of a factor $\log^{1/7}(1+N)$, while
Section~\ref{subsec:one-dim-lower} proves that the polynomial exponent $4/7$
is optimal.

The following two elementary properties of the Lagrangian supplement
Lemma~\ref{lem:lagrangian}.

\begin{lemma}\label{lem:lagrangian-onedim}
There exists $C$ such that for all $x,y\in\T$ and $a,\gamma\in\R$,
\begin{align}
 L(x,a+\gamma)&\leq L(x,a)+D_aL(x,a)\,\gamma+C\,\gamma^2,
 \label{eq:L-quadratic}\\
 |L(x,a)-L(y,a)|&\leq C\,(1+|a|)\,\dm(x,y).
 \label{eq:L-space}
\end{align}
\end{lemma}

\begin{proof}
By Lemma~\ref{lem:lagrangian}(ii), $D_aL(x,a)=-p(x,a)$.  Differentiating
the optimality relation $a=-D_pH(x,p(x,a))$ with respect to $a$ gives
\[
 D^2_{aa}L=\big(D^2_{pp}H\big)^{-1}\in[C_0^{-1},C_0].
\]
This gives \eqref{eq:L-quadratic}.  For \eqref{eq:L-space},
$|D_xL(x,a)|=|D_xH(x,p(x,a))|\leq C_0\big(1+|p(x,a)|\big)\leq
C(1+|a|)$ by \ref{ass:H-x} and Lemma~\ref{lem:lagrangian}(i).
\end{proof}

\subsection{The corrected shadow flow and the hard inequality}\label{subsec:corrected}

Fix, as in Sections~\ref{sec:shadow}--\ref{sec:comparison}, an initial
configuration $\bx$ and an adapted control $\balpha$ with
$|\alpha^i|\leq M$.  We take $t_0=0$ throughout this subsection, the
general case following by translating time. 

In higher dimensions the noise estimate of
Lemma~\ref{lem:martingale} degrades as the smoothing scale shrinks,
and the two prices in Proposition~\ref{prop:tracking} balance at the
rate $r_{N,d}$, which Theorem~\ref{thm:main} shows to be unimprovable.
At $d=1$ the smoothing price can be made negligible, and what remains
is a hard floor that no choice of the smoothing scale can improve: the
discarded noise, which returns as the martingale driving the
shadowing error and contributes $N^{-1/2}$ outright
($\Theta_1\equiv1$ in \eqref{eq:theta-definition}).  

Improving the
rate therefore means attacking the noise itself.
We do this through a correction of the shadow flow that modifies only
the transport of the components: their drifts acquire an anticipating
term, the filter of Lemma~\ref{lem:filter} below, which damps the low
Fourier modes of the shadowing error.  The
corrected flow is therefore no longer adapted, but what matters is
that each of its realizations is still an admissible competitor
in \eqref{eq:U-definition}.

To see how the correction is chosen, suppose provisionally that the
shadow flow $\nu$ is constructed from components $\nu_{i,t}$
transported by the drifts $\alpha^i_t+\gamma^i_t$, where $\gamma^i_t$
is still to be chosen.  Write $z_t=\mu^N_t-\nu_t$, and let $R_t$ be
the mean-zero primitive of $z_t$.
Fix $1\leq K\leq N$, and let $S_K$ be the smooth frequency cutoff of
Lemma~\ref{lem:cutoff}.  Then
\begin{equation}\label{eq:wasserstein-frequency-split}
 d_1(\mu^N_t,\nu_t)=\|z_t\|_{\Lip^*}
 \leq\|R_t\|_{L^1}
 \leq\|S_KR_t\|_{L^1}+\|(I-S_K)R_t\|_{L^1}.
\end{equation}
This reduces the shadowing estimate to controlling the low- and
high-frequency parts of $R_t$.  The high-frequency part will be
controlled directly by heat smoothing.  The purpose of the correction
is to reduce the low-frequency noise.  To analyze the first term on the
right-hand side of \eqref{eq:wasserstein-frequency-split}, we expand
$S_KR_t$ in the normalized nonconstant real Fourier modes
$(e_\ell)_{\ell\geq1}$ on $\T$, ordered by
\[
 e_{2n-1}(x)=\sqrt2\,\cos(2\pi nx),
 \qquad
 e_{2n}(x)=\sqrt2\,\sin(2\pi nx),
 \qquad n\geq1,
\]
and let $n_\ell$ denote the frequency of $e_\ell$.  The coefficient of
$R_t$ along $e_\ell$ can be written, by integration by parts, as
\begin{equation}\label{eq:primitive-fourier-coefficient}
 \int_\T R_te_\ell\,dy=-\int_\T\psi_\ell\,dz_t,
\end{equation}
where $\psi_\ell$ is the mean-zero primitive of $e_\ell$.  Thus the
low Fourier modes of $R_t$ are obtained by testing the shadowing error
$z_t$ against the primitives $\psi_\ell$.

Set $\Psi_K=(e_1,\dots,e_{2K})^{\mathsf T}$.  The functions $\psi_\ell$
satisfy
\begin{equation}\label{eq:fourier-frame}
 \psi_\ell'=e_\ell,
 \qquad
 -\psi_\ell''=\vartheta_\ell\,\psi_\ell
 \ \text{ with }\ \vartheta_\ell=4\pi^2n_\ell^2,
 \qquad
 \sup_{x\in\T}|\Psi_K(x)|^2\leq CK.
\end{equation}
For $\ell\leq 2K$, $n_\ell\leq K$.  Write
$\Lambda_K=\operatorname{diag}(\vartheta_1,\dots,\vartheta_{2K})$.
Set
\begin{equation}\label{eq:low-mode-vector}
 Z_t:=\bigg(\int_\T\psi_\ell\,dz_t\bigg)_{\ell\leq 2K}.
\end{equation}
It\^o's formula for the particle system gives, for $\ell\leq 2K$,
\begin{equation}\label{eq:particle-low-mode-equation}
 \begin{aligned}
 d\int_\T\psi_\ell\,d\mu^N_t
 &=\bigg(\int_\T\psi_\ell''\,d\mu^N_t
 +\frac1N\sum_{i=1}^N\alpha^i_t e_\ell(X^i_t)\bigg)\,dt
 +d\mathsf M_{\ell,t},\\
 \mathsf M_t&:=\frac{\sqrt2}N\sum_{i=1}^N
 \int_0^t\Psi_K(X^i_s)\,dW^i_s.
 \end{aligned}
\end{equation}
The quadratic variation of $\mathsf M$ is
\begin{equation}\label{eq:filter-qv}
 d\langle\mathsf M\rangle_t=\frac2N\,G_t\,dt,
 \qquad
 G_t:=\frac1N\sum_{i=1}^N
 \Psi_K(X^i_t)\Psi_K(X^i_t)^{\mathsf T}.
\end{equation}
Thus $G_t$ is the normalized covariance matrix of the low-mode noise.
It is predictable, symmetric and nonnegative.

The terms $\gamma^i_t$ will be chosen to counteract the martingale
$\mathsf M$.  By the weak form of \eqref{eq:shadow-FP} and the
componentwise flux formula \eqref{eq:beta-definition}, the resulting
shadow flow satisfies
\begin{equation}\label{eq:shadow-low-mode-equation}
 d\int_\T\psi_\ell\,d\nu_t
 =\bigg(\int_\T\psi_\ell''\,d\nu_t
 +\sum_{i=1}^N(\alpha^i_t+\gamma^i_t)
       \int_\T\psi_\ell'\,d\nu_{i,t}\bigg)\,dt.
\end{equation}
In the difference between \eqref{eq:particle-low-mode-equation} and
\eqref{eq:shadow-low-mode-equation}, the correction terms decompose,
using $\psi_\ell'=e_\ell$, as
\begin{equation}\label{eq:correction-mode-decomposition}
\begin{gathered}
 -\bigg(\sum_{i=1}^N\gamma^i_t\int_\T e_\ell\,d\nu_{i,t}\bigg)_{\ell\leq 2K}
 =-\Gamma_t
 +\bigg(\int_\T e_\ell\,d\mathsf r^\gamma_t\bigg)_{\ell\leq 2K},\\
 \Gamma_t
 :=\bigg(\frac1N\sum_{i=1}^N
 \gamma^i_t e_\ell(X^i_t)\bigg)_{\ell\leq 2K},
 \qquad
 \mathsf r^\gamma_t:=\sum_{i=1}^N\gamma^i_t
 \Big(\frac1N\delta_{X^i_t}-\nu_{i,t}\Big).
\end{gathered}
\end{equation}
Since the correction is designed to control the first $2K$ modes, let
$Y_t$ encode their correction amplitudes.  The Fourier polynomial
$y\mapsto\Psi_K(y)^{\mathsf T}Y_t$ turns these amplitudes into a spatial
drift.  This reasoning forces the ansatz
\[
 \gamma^i_t=\kappa\,\Psi_K(X^i_t)^{\mathsf T}Y_t,
\]
where $\kappa>0$ is the strength of the correction.  By the
definitions of $\Gamma_t$ and $G_t$, we then have
$\Gamma_t=\kappa G_tY_t$.  Thus $G_t$ describes both the covariance of
the projected noise and the response of the low modes to the
correction.  Subtracting
\eqref{eq:shadow-low-mode-equation} from
\eqref{eq:particle-low-mode-equation} and using
\eqref{eq:low-mode-vector}, \eqref{eq:fourier-frame},
\eqref{eq:z-r-definition}, and \eqref{eq:correction-mode-decomposition},
we obtain the following equation for $Z_t$:
\begin{equation}\label{eq:low-mode-equation}
 dZ_t=-\Lambda_KZ_t\,dt-\kappa G_tY_t\,dt+d\mathsf M_t
 +\Big(\int_\T e_\ell\,d\big(\mathsf r_t+\mathsf r^\gamma_t\big)
 \Big)_{\ell\leq 2K}\,dt,
 \qquad Z_0=0.
\end{equation}
To focus on the correction and martingale noise, it is convenient to analyze the simpler equation obtained from discarding the
flux-mismatch terms: namely, let $\mathsf e$
solve
\begin{equation}\label{eq:filter-error-equation}
 d\mathsf e_t=-\Lambda_K\mathsf e_t\,dt
 -\kappa G_tY_t\,dt+d\mathsf M_t,
 \qquad \mathsf e_0=0.
\end{equation}
This suggests choosing \(Y\) to satisfy the differential equation below. We impose the terminal condition so that the resulting process is a stable, conditionally centered filter of future noise:
\begin{equation}\label{eq:filter-terminal}
 dY_t=\kappa G_tY_t\,dt-d\mathsf M_t,
 \qquad Y_T=0.
\end{equation}
Indeed, if $\eta:=\mathsf e+Y$, the correction and noise terms cancel
and
\[
 d\eta_t=-\Lambda_K\eta_t\,dt+\Lambda_KY_t\,dt.
\]
Thus $\mathsf e$ can be estimated entirely through $Y$.  At the same
time, the occurrence of the same matrix $G_t$ in
\eqref{eq:filter-terminal} and in the quadratic variation
\eqref{eq:filter-qv} allows the conditional It\^o isometry to yield a clean
 conditional second-moment bound for $Y_t$. The following lemma
makes these two conclusions precise.

\begin{lemma}[The future filter]\label{lem:filter}

For every $\kappa>0$, let the propagator $\Phi(t,s)$, $t\leq s$, solve
\begin{equation}\label{eq:filter-propagator}
 \partial_t\Phi(t,s)=\kappa G_t\,\Phi(t,s),
 \qquad \Phi(s,s)=I,
\end{equation}
and define
\begin{equation}\label{eq:filter-definition}
 Y_t:=\int_t^T\Phi(t,s)\,d\mathsf M_s.
\end{equation}
Then $Y$ has a continuous version satisfying
\eqref{eq:filter-terminal}, and
\begin{equation}\label{eq:filter-moments}
 \E[Y_t\mid\F_t]=0,
 \qquad
 \E\big[Y_tY_t^{\mathsf T}\mid\F_t\big]\leq\frac1{N\kappa}\,I.
\end{equation}
Moreover, the solution of \eqref{eq:filter-error-equation} satisfies
\begin{equation}\label{eq:filter-error}
 \sup_{0\leq t\leq T}\E\,|\mathsf e_t|^2
 \leq\frac{CK}{N\kappa}.
\end{equation}

\end{lemma}

\begin{proof}

Pathwise uniqueness for the linear terminal-value equation
\eqref{eq:filter-propagator} gives the composition identity
\begin{equation}\label{eq:filter-composition}
 \Phi(t,s)=\Phi(t,u)\Phi(u,s),
 \qquad 0\leq t\leq u\leq s\leq T.
\end{equation}
Indeed, as functions of $t\in[0,u]$, both sides solve
\eqref{eq:filter-propagator} and take the value $\Phi(u,s)$ at $t=u$.
Moreover, integrating \eqref{eq:filter-propagator} backward from $s+h$
to $s$ gives
\[
 \Phi(s,s+h)=I-\kappa G_sh+o(h)
\]
for almost every $s$.  Applying \eqref{eq:filter-composition} with the
triple $(t,s,s+h)$ therefore yields
\begin{equation}\label{eq:filter-upper-derivative}
 \partial_s\Phi(t,s)=-\kappa\Phi(t,s)G_s
 \qquad\text{for almost every }s\geq t.
\end{equation}
For fixed $t$, the map $s\mapsto\Phi(t,s)$ depends only on
$(G_u)_{t\leq u\leq s}$ and is continuous and adapted, hence
predictable.

Hence the stochastic integral in \eqref{eq:filter-definition} has
conditional mean zero.  By
\eqref{eq:filter-upper-derivative} and the symmetry of $G_s$,
\[
 \frac{d}{ds}\big(\Phi(t,s)\Phi(t,s)^{\mathsf T}\big)
 =-2\kappa\Phi(t,s)G_s\Phi(t,s)^{\mathsf T}.
\]

Integrating this identity from $t$ to $T$ and using $\Phi(t,t)=I$ gives
\begin{equation}\label{eq:filter-energy-identity}
 2\kappa\int_t^T\Phi(t,s)G_s\Phi(t,s)^{\mathsf T}\,ds
 =I-\Phi(t,T)\Phi(t,T)^{\mathsf T}.
\end{equation}
The conditional It\^o isometry for the martingale transform on $[t,T]$
(applied entrywise; see \cite[Thm.~5.6 and
Prop.~4.15(v)]{LeGall2016}), followed by \eqref{eq:filter-qv} and
\eqref{eq:filter-energy-identity}, yields
\[
\begin{aligned}
 \E\big[Y_tY_t^{\mathsf T}\mid\F_t\big]
 &=\E\bigg[\int_t^T\Phi(t,s)\,d\langle\mathsf M\rangle_s\,
       \Phi(t,s)^{\mathsf T}\,\bigg|\,\F_t\bigg]\\
 &=\frac2N\E\bigg[\int_t^T\Phi(t,s)G_s
       \Phi(t,s)^{\mathsf T}\,ds\,\bigg|\,\F_t\bigg]\\
 &=\frac1{N\kappa}\E\big[I-\Phi(t,T)
       \Phi(t,T)^{\mathsf T}\mid\F_t\big]
 \leq\frac1{N\kappa}I.
\end{aligned}
\]
The last inequality follows from
$\Phi(t,T)\Phi(t,T)^{\mathsf T}\geq0$. This proves \eqref{eq:filter-moments}.

Write $\Phi_t:=\Phi(0,t)$.  By
\eqref{eq:filter-upper-derivative},
$\partial_t\Phi_t=-\kappa\Phi_tG_t$, with $\Phi_0=I$.  The determinant
chain rule, using
$\operatorname{adj}(\Phi_t)\Phi_t=(\det\Phi_t)I$, gives
\[
 \frac{d}{dt}\det\Phi_t
 =-\kappa\operatorname{tr}(G_t)\det\Phi_t
 \qquad\text{for almost every }t.
\]
Consequently,
\[
 \det\Phi_t
 =\exp\left(-\kappa\int_0^t\operatorname{tr}(G_s)\,ds\right)>0,
\]
so $\Phi_t$ is invertible.  Taking $(0,t,s)$ in
\eqref{eq:filter-composition} gives
$\Phi(t,s)=\Phi_t^{-1}\Phi_s$.
Pulling the $\F_t$-measurable factor $\Phi_t^{-1}$ out of the integral
in \eqref{eq:filter-definition},
\[
 Y_t=\Phi_t^{-1}\mathsf L_T-\Phi_t^{-1}\mathsf L_t
 \quad\text{almost surely, where }
 \mathsf L_t:=\int_0^t\Phi_s\,d\mathsf M_s .
\]
The right-hand side is continuous in $t$, and we pass to this version
of $Y$.  Since $t\mapsto\Phi_t^{-1}$ has finite variation, with
$d(\Phi_t^{-1})=\kappa\,G_t\Phi_t^{-1}\,dt$, the product rule
\cite[Thm.~5.10]{LeGall2016} gives
$dY_t=\kappa G_tY_t\,dt-d\mathsf M_t$.  Thus
$\eta:=\mathsf e+Y$ satisfies
$d\eta_t=-\Lambda_K\eta_t\,dt+\Lambda_KY_t\,dt$
with $\eta_0=Y_0$, so
\[
 \mathsf e_\ell(t)=-Y_\ell(t)+e^{-\vartheta_\ell t}Y_\ell(0)
 +\vartheta_\ell\int_0^te^{-\vartheta_\ell(t-s)}Y_\ell(s)\,ds.
\]
Cauchy--Schwarz and \eqref{eq:filter-moments} yield
$\E|\mathsf e_\ell(t)|^2\leq C/(N\kappa)$.  Summing over $\ell$ proves
\eqref{eq:filter-error}.

\end{proof}

\begin{proposition}[Corrected shadowing and cost]\label{prop:corrected}
Let $1\leq K\leq N$ and write
\begin{equation}\label{eq:corrected-scales}
 \mathfrak a:=\frac{K^3}{N},
 \qquad
 \varepsilon_{N,K}:=\frac1{\sqrt{NK}}+\mathfrak a.
\end{equation}
For every $h>0$ there exists a random pair $(\nu,\beta)$ with
$\nu_0=m^N_{\bx}$, each realization of which is admissible in
\eqref{eq:U-definition}, such that
\begin{align}
 \sup_{0\leq t\leq T}\E\,d_1\big(\mu^N_t,\nu_t\big)
 &\leq C_M\big(\varepsilon_{N,K}+\sqrt h+h\sqrt{\mathfrak a}\big),
 \label{eq:corrected-tracking}\\
 \E\int_0^T\!\!\int_\T L\big(y,\beta_t(y)\big)\,\nu_t(dy)\,dt
 &\leq\E\int_0^T\frac1N\sum_{i=1}^NL\big(X^i_t,\alpha^i_t\big)\,dt
 +C_M\big(\varepsilon_{N,K}+\sqrt h+h\sqrt{\mathfrak a}\big).
 \label{eq:corrected-action}
\end{align}
\end{proposition}

\begin{proof}
\emph{The correction.}  Apply Lemma~\ref{lem:filter} with strength
$\kappa=K^2$, and let $\mathsf e$ be the corresponding solution of
\eqref{eq:filter-error-equation}.  Set
\begin{equation}\label{eq:gamma-definition}
 \gamma^i_t:=K^2\,\Psi_K(X^i_t)^{\mathsf T}Y_t,
 \qquad
 A_t:=\frac1N\sum_{i=1}^N|\gamma^i_t|^2
 =K^4\,Y_t^{\mathsf T}G_tY_t.
\end{equation}
By \eqref{eq:filter-moments} and \eqref{eq:fourier-frame},
\begin{equation}\label{eq:A-pointwise}
 \E A_t\leq C\mathfrak a.
\end{equation}
By the centering in \eqref{eq:filter-moments}, since
$D_aL(X^i_t,\alpha^i_t)$ and
$\Psi_K(X^i_t)$ are $\F_t$-measurable,
\begin{equation}\label{eq:corrected-cancellation}
 \E\,\frac1N\sum_{i=1}^ND_aL\big(X^i_t,\alpha^i_t\big)\,\gamma^i_t=0.
\end{equation}

\emph{The flow.}  We run the construction of Section~\ref{sec:shadow}
on a
grid of mesh $h$, with the transport of the $i$-th component driven by
$\alpha^i+\gamma^i$ in place of $\alpha^i$, calling the result
$(\nu,\beta)$ and keeping the notation $\nu_{i,t}$ and $c_t$ of
\eqref{eq:c-definition}.  By \eqref{eq:A-pointwise} and
Fubini's theorem, $\int_0^TA_t\,dt<\infty$ almost surely.  Thus the
component transports have absolutely continuous paths, which is the
only time regularity used in Section~\ref{sec:shadow} and
Lemma~\ref{lem:FP}.  Both apply pathwise even though $\gamma^i$ is not
adapted.  Moreover,
$\int_0^T\!\int_\T|\beta_t|^2\,d\nu_t\,dt
\leq2TM^2+2\int_0^TA_t\,dt$, so every realization of $(\nu,\beta)$ is
admissible in \eqref{eq:U-definition}.  Rerunning the proof of
Lemma~\ref{lem:provisional} with the component drifts
$\alpha^i+\gamma^i$ produces one additional term, the displacement of
the components by the correction, which on each grid interval
contributes at most
\[
 \E\int_{s_k}^t\frac1N\sum_{i=1}^N|\gamma^i_s|\,ds
 \leq\int_{s_k}^t\sqrt{\E A_s}\,ds
 \leq Ch\sqrt{\mathfrak a}.
\]
Consequently,
\begin{equation}\label{eq:corrected-provisional}
 \E c_t\leq\E\,d_1(\mu^N_t,\nu_t)+C_M\sqrt h+Ch\sqrt{\mathfrak a}.
\end{equation}

Recalling that $R_t$ is the mean-zero primitive of $z_t$, our goal is now to estimate the two terms on the right-hand side of
\eqref{eq:wasserstein-frequency-split}.

\emph{Shadowing: the low modes.}  Let $\mathsf r_t$ be the flux
mismatch \eqref{eq:z-r-definition} of the uncorrected flow, associated
with $\balpha$.  Recalling $\mathsf r^\gamma_t$ from
\eqref{eq:correction-mode-decomposition}, define
\begin{equation}\label{eq:gamma-fluxes}
 \bar{\mathsf q}_t:=\sum_{i=1}^N\gamma^i_t\,\nu_{i,t},
 \qquad
 b_t:=\sum_{i=1}^N|\gamma^i_t|\int_\T\dm\big(y,X^i_t\big)\,\nu_{i,t}(dy).
\end{equation}
The particles are driven by $\balpha$ alone, while the components are
transported by $\alpha^i+\gamma^i$, and $\bar{\mathsf q}_t$ is the
flux that the correction contributes to the shadow flow.
$\mathsf r^\gamma_t$ compares this flux with the same drifts attached
to the particles, so its summands, like those of $\mathsf r_t$, are
zero-mass differences, of total transport cost $b_t$.
Writing $d_i:=\int\dm(y,X^i_t)\,\nu_{i,t}(dy)\leq C/N$, Cauchy--Schwarz
and Young's inequality give
\begin{equation}\label{eq:b-bound}
 b_t\leq
 \Big(\frac1N\sum_i|\gamma^i_t|^2\Big)^{1/2}
 \Big(N\sum_id_i^2\Big)^{1/2}
 \leq C\sqrt{A_tc_t}\leq C(A_t+c_t).
\end{equation}

Since the filter was chosen with $\kappa=K^2$,
\eqref{eq:filter-error} and Cauchy--Schwarz give
\begin{equation}\label{eq:applied-filter-error}
 \sup_{0\leq t\leq T}\E|\mathsf e_t|
 \leq\sup_{0\leq t\leq T}\big(\E|\mathsf e_t|^2\big)^{1/2}
 \leq\frac C{\sqrt{NK}}.
\end{equation}
Subtracting
\eqref{eq:filter-error-equation} from \eqref{eq:low-mode-equation},
the terms $d\mathsf M_t$ and $-K^2G_tY_t\,dt$ cancel, and solving the
resulting ODE coordinatewise gives
\begin{equation}\label{eq:low-mode-coordinate}
 (Z_t-\mathsf e_t)_\ell
 =\int_0^te^{-\vartheta_\ell(t-s)}
 \int_\T e_\ell\,d\big(\mathsf r_s+\mathsf r^\gamma_s\big)\,ds .
\end{equation}
Since $R_t$ has mean zero and, by Lemma~\ref{lem:cutoff}, $S_K$ is
supported on the frequencies $|n|<K$, every nonzero mode of $S_KR_t$
occurs in the frame $\Psi_K$.  Moreover,
\eqref{eq:primitive-fourier-coefficient} and
\eqref{eq:low-mode-vector} show that the coefficient of $R_t$ along
$e_\ell$ is $-(Z_t)_\ell$.
Summing \eqref{eq:low-mode-coordinate} against the cutoff weights
$\chi(n_\ell/K)$ therefore gives
\begin{equation}\label{eq:low-mode-representation}
 \begin{aligned}
 S_KR_t
 &=-\sum_{\ell=1}^{2K}
 \chi(n_\ell/K)\,(Z_t)_\ell\,e_\ell\\
 &=-\sum_{\ell=1}^{2K}
 \chi(n_\ell/K)\,\mathsf e_\ell(t)\,e_\ell
 -\int_0^tS_KP_{t-s}\big(\mathsf r_s+\mathsf r^\gamma_s\big)\,ds.
 \end{aligned}
\end{equation}
The proof of
Lemma~\ref{lem:flux}, with $|\gamma^i_t|$ in place of the uniform bound
on $|\alpha^i_t|$, recalling that $b_t$ is given by
\eqref{eq:gamma-fluxes}, yields
\begin{equation}\label{eq:gamma-flux-bound}
 \big\|P_u\mathsf r^\gamma_t\big\|_{L^1}\leq Cu^{-1/2}b_t,
 \qquad 0<u\leq T.
\end{equation}
For the first term on the right-hand side of
\eqref{eq:low-mode-representation}, $|\chi|\leq1$, Parseval's identity,
and \eqref{eq:applied-filter-error} give
\begin{equation}\label{eq:low-mode-filter-term}
\begin{aligned}
 \E\,\bigg\|\sum_{\ell=1}^{2K}
 \chi(n_\ell/K)\,\mathsf e_\ell(t)e_\ell\bigg\|_{L^1}
 &\leq \E\,\bigg\|\sum_{\ell=1}^{2K}
 \chi(n_\ell/K)\,\mathsf e_\ell(t)e_\ell\bigg\|_{L^2}\\
 &=\E\bigg(\sum_{\ell=1}^{2K}|\chi(n_\ell/K)|^2
 |\mathsf e_\ell(t)|^2\bigg)^{1/2}\leq \E|\mathsf e_t|
 \leq\frac C{\sqrt{NK}}.
\end{aligned}
\end{equation}

For the integral term, Minkowski's inequality, the
$L^1$-boundedness of $S_K$ in Lemma~\ref{lem:cutoff}(i),
Lemma~\ref{lem:flux}, and \eqref{eq:gamma-flux-bound} give
\begin{equation}\label{eq:low-mode-flux-term}
 \E\,\bigg\|\int_0^tS_KP_{t-s}
 \big(\mathsf r_s+\mathsf r^\gamma_s\big)\,ds\bigg\|_{L^1}
 \leq C_M\int_0^t(t-s)^{-1/2}\,\E(c_s+b_s)\,ds.
\end{equation}

Combining \eqref{eq:low-mode-representation},
\eqref{eq:low-mode-filter-term}, and \eqref{eq:low-mode-flux-term} gives
\begin{equation}\label{eq:low-mode-bound}
 \E\,\|S_KR_t\|_{L^1}
 \leq\frac C{\sqrt{NK}}
 +C_M\int_0^t(t-s)^{-1/2}\,\E(c_s+b_s)\,ds.
\end{equation}

\emph{Shadowing: the high modes.}
The particles have flux
$N^{-1}\sum_i\alpha^i_t\delta_{X^i_t}$, whereas the corrected shadow
flow has flux $\sum_i(\alpha^i_t+\gamma^i_t)\nu_{i,t}$.  Thus, by
\eqref{eq:z-r-definition} and \eqref{eq:gamma-fluxes}, their difference is
\[
 \frac1N\sum_i\alpha^i_t\delta_{X^i_t}
 -\sum_i(\alpha^i_t+\gamma^i_t)\nu_{i,t}
 =\mathsf r_t-\bar{\mathsf q}_t.
\]
The Duhamel representation \eqref{eq:duhamel} therefore
holds for $z_t=\mu^N_t-\nu_t$, the shadowing error of the corrected
flow, with
$\mathsf r_s-\bar{\mathsf q}_s$ in place of $\mathsf r_s$.  Indeed, its
proof uses only It\^o's formula for the particles and the pathwise
Fokker--Planck equation for the shadow flow.  Applying the corrected
representation with $\varphi=\psi_\ell$ and using
\eqref{eq:primitive-fourier-coefficient} gives, for every
$\ell\geq1$,
\begin{equation}\label{eq:high-mode-coordinate}
 \int_\T R_te_\ell\,dy
 =-\int_0^t\int_\T P_{t-s}e_\ell\,
       d\big(\mathsf r_s-\bar{\mathsf q}_s\big)\,ds
 -\frac{\sqrt2}{N}\sum_{i=1}^N\int_0^t
       \big(P_{t-s}e_\ell\big)(X^i_s)\,dW^i_s .
\end{equation}
Multiplying \eqref{eq:high-mode-coordinate} by
$1-\chi(n_\ell/K)$ and summing over $\ell\geq1$ gives
\begin{equation}\label{eq:high-mode-representation}
 (I-S_K)R_t
 =-\int_0^t(I-S_K)P_{t-s}\big(\mathsf r_s-\bar{\mathsf q}_s\big)\,ds
 +\mathcal N^K_t,
\end{equation}
where
\[
 \mathcal N^K_t(y):=-\frac{\sqrt2}N\sum_{i=1}^N\int_0^t
 \big((I-S_K)p_{t-s}\big)\big(y-X^i_s\big)\,dW^i_s .
\]
The stochastic series defining $\mathcal N^K_t$ converges
in $L^2(\Omega\times\T)$ by Lemma~\ref{lem:cutoff}(ii).

For the uncorrected flux mismatch,
Lemma~\ref{lem:cutoff}(i) and Lemma~\ref{lem:flux} give
\begin{equation}\label{eq:high-mode-flux-term}
 \big\|(I-S_K)P_{t-s}\mathsf r_s\big\|_{L^1}
 \leq C_M(t-s)^{-1/2}c_s.
\end{equation}
For the correction flux $\bar{\mathsf q}_s$ defined in
\eqref{eq:gamma-fluxes}, the component masses and Cauchy--Schwarz yield
\[
 |\bar{\mathsf q}_s|(\T)
 \leq\frac1N\sum_{i=1}^N|\gamma^i_s|
 \leq\bigg(\frac1N\sum_{i=1}^N|\gamma^i_s|^2\bigg)^{1/2}
 =\sqrt{A_s}.
\]
Young's inequality and Lemma~\ref{lem:cutoff}(ii) therefore
give
\begin{equation}\label{eq:high-mode-correction-pointwise}
 \big\|(I-S_K)P_{t-s}\bar{\mathsf q}_s\big\|_{L^1}
 \leq \big\|(I-S_K)p_{t-s}\big\|_{L^1}
       |\bar{\mathsf q}_s|(\T)
 \leq Ce^{-cK^2(t-s)}\sqrt{A_s}.
\end{equation}
Taking expectations in
\eqref{eq:high-mode-correction-pointwise}, integrating in $s$, and
using Jensen's inequality, \eqref{eq:A-pointwise}, and
\eqref{eq:corrected-scales}, we obtain
\begin{equation}\label{eq:high-mode-correction-term}
 \int_0^t\E\big\|(I-S_K)P_{t-s}\bar{\mathsf q}_s\big\|_{L^1}\,ds
 \leq C\sqrt{\mathfrak a}\int_0^te^{-cK^2(t-s)}\,ds
 \leq\frac{C\sqrt{\mathfrak a}}{K^2}=\frac C{\sqrt{NK}}.
\end{equation}
Finally, the It\^o isometry and Lemma~\ref{lem:cutoff}(ii), as in
Lemma~\ref{lem:martingale}, give
\begin{equation}\label{eq:high-mode-noise-term}
 \E\|\mathcal N^K_t\|_{L^1}
 \leq\Big(\frac2N\int_0^t\|(I-S_K)p_u\|_{L^2}^2\,du\Big)^{1/2}
 \leq\frac C{\sqrt{NK}}.
\end{equation}
Combining \eqref{eq:high-mode-representation},
\eqref{eq:high-mode-flux-term}, \eqref{eq:high-mode-correction-term},
and \eqref{eq:high-mode-noise-term}, we conclude that
\begin{equation}\label{eq:high-mode-bound}
 \E\,\big\|(I-S_K)R_t\big\|_{L^1}
 \leq C_M\int_0^t\big(1+(t-s)^{-1/2}\big)\,\E c_s\,ds
 +\frac C{\sqrt{NK}} .
\end{equation}

\emph{Shadowing: closing.}  Set $f(t):=\E\,d_1(\mu^N_t,\nu_t)$.  Combining
\eqref{eq:wasserstein-frequency-split}, \eqref{eq:low-mode-bound}, and
\eqref{eq:high-mode-bound},
\[
 f(t)\leq \frac C{\sqrt{NK}}
 +C_M\int_0^t\big(1+(t-s)^{-1/2}\big)\,\E(c_s+b_s)\,ds.
\]
By \eqref{eq:b-bound}, \eqref{eq:A-pointwise}, and
\eqref{eq:corrected-provisional},
\[
 \E(c_s+b_s)
 \leq C_M\big(f(s)+\mathfrak a+\sqrt h
                  +h\sqrt{\mathfrak a}\big).
\]
It follows that
\[
 f(t)\leq C_M\bigg(\frac1{\sqrt{NK}}+\mathfrak a+\sqrt h
                    +h\sqrt{\mathfrak a}\bigg)
 +C_M\int_0^t(t-s)^{-1/2}f(s)\,ds.
\]
As in the uncorrected setting, the weakly singular Gr\"onwall inequality \cite[Lem.~7.1.1]{Henry1981},
together with
\eqref{eq:corrected-scales}, now proves \eqref{eq:corrected-tracking}.

\emph{The cost.}  Convexity of $L(y,\cdot)$ gives
\[
 \int_\T L\big(y,\beta_t(y)\big)\,\nu_t(dy)
 \leq\sum_{i=1}^N\int_\T
 L\big(y,\alpha^i_t+\gamma^i_t\big)\,\nu_{i,t}(dy).
\]

We first move the base point from $y$ to $X^i_t$ with
\eqref{eq:L-space}, and then apply \eqref{eq:L-quadratic} at
$(X^i_t,\alpha^i_t)$.  Recalling \eqref{eq:c-definition},
\eqref{eq:gamma-fluxes}, and \eqref{eq:gamma-definition}, and using
that each component has mass $1/N$, the right-hand side is at most
\[
 \frac1N\sum_{i=1}^NL\big(X^i_t,\alpha^i_t\big)
 +\frac1N\sum_{i=1}^ND_aL\big(X^i_t,\alpha^i_t\big)\,\gamma^i_t
 +C_M\,c_t+C_M\,b_t+C\,A_t.
\]
We take expectations and integrate in time: the second term vanishes by
\eqref{eq:corrected-cancellation}, the last is bounded by $C\mathfrak a$
through \eqref{eq:A-pointwise}, and $c_t$ and $b_t$ are controlled by
\eqref{eq:corrected-provisional}, \eqref{eq:corrected-tracking} and the
bound on $\E(c_s+b_s)$ above.  This proves
\eqref{eq:corrected-action}.
\end{proof}

\begin{proposition}\label{prop:one-dim-hard}
Uniformly in $N$, $t_0$, and $\bx$,
\begin{equation}\label{eq:one-dim-hard}
 U\big(t_0,m^N_{\bx}\big)\leq V^N(t_0,\bx)+CN^{-4/7}.
\end{equation}
\end{proposition}

\begin{proof}
Let $\balpha$ be bounded by $M$ and, for large $N$, take
$K=\lfloor N^{1/7}\rfloor$, so that
$\varepsilon_{N,K}\leq CN^{-4/7}$ by \eqref{eq:corrected-scales}.
Each realization
of the pair of Proposition~\ref{prop:corrected} is admissible in
\eqref{eq:U-definition}, so, exactly as in the proof of
Proposition~\ref{prop:comparison}, the bounds
\eqref{eq:corrected-action} and \eqref{eq:corrected-tracking} and the
$d_1$-Lipschitz continuity of $F$ and $G$ give, after $h\downarrow0$,
\[
 U\big(t_0,m^N_{\bx}\big)\leq J_N(t_0,\bx;\balpha)+C_MN^{-4/7}.
\]
The passage to the value function is that of
Section~\ref{sec:comparison}: the smoothing of Lemma~\ref{lem:smoothing}
and the uniform feedback bound \eqref{eq:feedback-bound} reduce
\eqref{eq:one-dim-hard} to controls bounded by a universal constant, and
the finitely many small values of $N$ are absorbed into $C$.
\end{proof}

\subsection{The Gibbs competitor and the formerly easy inequality}\label{subsec:gibbs}

In this subsection we prove Proposition~\ref{prop:one-dim-easy},
constructing a particle control whose cost exceeds the mean field
value by at most $CN^{-4/7}\log^{1/7}(1+N)$.  Independent particles
cannot achieve this: replicating a mean field control across
independent noises, as in the easy inequality of
Theorem~\ref{thm:main}, produces empirical fluctuations of order
$N^{-1/2}$.  We therefore correlate the particles: their positions are sampled
from the Gibbs law \eqref{eq:gibbs-law} below, at a fixed inverse
temperature $\varsigma$, and transported along a regular near-optimal
flow (Lemma~\ref{lem:regular-pairs}) by its quantile map.  The Gibbs
law keeps the empirical discrepancy at $(N\varsigma)^{-1/3}$
(Lemma~\ref{lem:gibbs-shape}), while the drift that maintains it
costs the quadratic energy $(\varsigma/N)^2\log(1+N)$
(Proposition~\ref{prop:gibbs} and \eqref{eq:chi-range}).  The two
scales balance at $\varsigma_*=N^{5/7}\log^{-3/7}(1+N)$
as defined in \eqref{eq:optimal-temperature}, where both are
$N^{-4/7}\log^{1/7}(1+N)$.

A function on $\T^N$ is called \emph{diagonally invariant} if it is invariant
under the common translations $\by\mapsto\by+\theta(1,\dots,1)$.  For
$\by\in\T^N$ and $\varsigma>0$, we define the discrepancy and the Gibbs
law by
\begin{equation}\label{eq:gibbs-law}
 \mathcal W_N(\by):=d_1\big(m^N_{\by},\lambda\big),
 \qquad
 \pi_{N,\varsigma}(d\by)
 :=\mathcal Z_{N,\varsigma}^{-1}\,
 e^{-\varsigma\mathcal W_N(\by)}\,d\by,
\end{equation}
where $\mathcal Z_{N,\varsigma}$ normalizes the mass.  The density is
symmetric and diagonally invariant.  In particular every one-coordinate
marginal of $\pi_{N,\varsigma}$ is uniform.

\begin{lemma}[Empirical small ball]\label{lem:gibbs-small-ball}
For every $\delta>0$ there exist $C_\delta$ and $N_\delta$ such that, for
every $N\geq N_\delta$ and every $r$ satisfying
$\delta N^{-4/7}\leq r\leq N^{-1/2}$,
\begin{equation}\label{eq:gibbs-small-ball}
 \lambda^{\otimes N}\big(\mathcal W_N\leq r\big)
 \geq\exp\Big(-\frac {C_\delta}{Nr^2}\Big).
\end{equation}
\end{lemma}

\begin{proof}
Let $F_N$ be the empirical distribution function of $N$ independent
uniform points in $[0,1]$.  Since transportation on the circle costs
no more than transportation on the interval,
\[
 \mathcal W_N\leq\int_0^1|F_N(u)-u|\,du
 \leq\|F_N-\operatorname{id}\|_{L^\infty}.
\]
By \cite[Thm.~3, Eq.~(1.4)]{KomlosMajorTusnady1975}, $F_N$ can be coupled with a
standard Brownian bridge $B$.  Since $Nr\geq\delta N^{3/7}$, taking
$x=Nr/4$ in the exponential bound of that theorem gives, for all
sufficiently large $N$ depending only on $\delta$,
\[
 \Prob\left(
 \left\|\sqrt N\big(F_N-\operatorname{id}\big)-B\right\|_{L^\infty}
 >\frac{\sqrt N\,r}{2}\right)
 \leq C e^{-cNr};
\]
here we used $C\log(1+N)\leq Nr/4$.  On the complementary event,
$\|B\|_{L^\infty}\leq\sqrt N\,r/2$ implies
$\|F_N-\operatorname{id}\|_{L^\infty}\leq r$, and therefore
\[
 \lambda^{\otimes N}(\mathcal W_N\leq r)
 \geq\Prob\left(\|B\|_{L^\infty}\leq\frac{\sqrt N\,r}{2}\right)
 -C e^{-cNr}.
\]
The lower-tail representation of the Kolmogorov distribution
\cite[Sec.~3]{MarsagliaTsangWang2003} gives, for $a>0$,
\[
 \Prob\big(\|B\|_{L^\infty}\leq a\big)
 =\frac{\sqrt{2\pi}}{a}\sum_{k=1}^\infty
 \exp\left(-\frac{(2k-1)^2\pi^2}{8a^2}\right).
\]
Here $a=\sqrt N\,r/2\leq1/2$, so the first term of the series gives
\[
 \Prob\left(\|B\|_{L^\infty}\leq\frac{\sqrt N\,r}{2}\right)
 \geq\exp\left(-\frac{C}{Nr^2}\right).
\]
Since $Nr\geq\delta N^{3/7}$ and
$(Nr^2)^{-1}\leq\delta^{-2}N^{1/7}$, the exceptional probability in
the coupling is absorbed by the first term once $N_\delta$ is sufficiently
large.  This proves \eqref{eq:gibbs-small-ball}.
\end{proof}

\begin{lemma}[Static Gibbs estimate]\label{lem:gibbs-shape}
Whenever
$N^{1/2}\leq\varsigma\leq N^{5/7}$,
\begin{equation}\label{eq:gibbs-moments}
 \int_{\T^N}\mathcal W_N\,d\pi_{N,\varsigma}
 \leq C\,(N\varsigma)^{-1/3}.
\end{equation}
\end{lemma}

\begin{proof}
Let $N_1$ be given by Lemma~\ref{lem:gibbs-small-ball} with $\delta=1$.
If $N<N_1$, then $\mathcal W_N\leq\operatorname{diam}(\T)$ and
$(N\varsigma)^{-1/3}\geq N^{-4/7}\geq N_1^{-4/7}$, so
\eqref{eq:gibbs-moments} follows after increasing $C$.  We may therefore
assume $N\geq N_1$.

We take $r=(N\varsigma)^{-1/3}$.  The prescribed range of $\varsigma$ gives
$N^{-4/7}\leq r\leq N^{-1/2}$ and $\varsigma r\geq1$.  Inserting
\eqref{eq:gibbs-small-ball} into the
partition function and using
$(Nr^2)^{-1}=\varsigma r$ gives
$\mathcal Z_{N,\varsigma}\geq e^{-C\varsigma r}$.  Therefore, for a
sufficiently large numerical $A$ and every $v\geq0$,
\[
 \pi_{N,\varsigma}\Big(\mathcal W_N>Ar+\frac v\varsigma\Big)
 \leq \mathcal Z_{N,\varsigma}^{-1}
       e^{-\varsigma(Ar+v/\varsigma)}
 \leq e^{-c\varsigma r-v}.
\]
Integrating this tail and using $\varsigma^{-1}\leq r$ proves
\eqref{eq:gibbs-moments}.
\end{proof}

\begin{lemma}[Regular near-optimal pairs]\label{lem:regular-pairs}
For all $\eta>0$, $0<\tau<T/2$, and $m_0\in\Pp(\T)$, there exists a
classical Fokker--Planck pair $(m,\alpha)$ on $[\tau,T]$, with
$m_\tau=P_\tau m_0$, such that
\begin{align}
 &\int_0^\tau\left(\int_\T L(y,0)\,P_tm_0(dy)+F(P_tm_0)\right)dt
 \notag\\
 &\quad+\int_\tau^T\left(\int_\T L(y,\alpha_t(y))\,m_t(dy)+F(m_t)\right)dt
 +G(m_T)
 \leq U(0,m_0)+C\tau+\eta,
 \label{eq:regular-cost}
\end{align}
and
\begin{equation}\label{eq:regular-pair}
 \|\alpha\|_{L^\infty([\tau,T]\times\T)}\leq C,
 \qquad
 m_t=\rho_t\,dx
 \quad\text{with}\quad
 \|\rho_t\|_{L^\infty(\T)}\leq C\big(1+t^{-1/2}\big),
\end{equation}
and in particular
\begin{equation}\label{eq:chi-range}
 1\leq\chi_t:=\int_\T\rho_t^3\,dx\leq C\big(1+t^{-1}\big).
\end{equation}
The constants do not depend on $\eta$, $\tau$, or $m_0$.
\end{lemma}

The proof is given in Appendix~\ref{app:deferred}.

We now consider a classical Fokker--Planck pair $(m,\alpha)$ on $[\tau,T]$
with $m_t=\rho_t\,dx$, $\rho_t>0$.  We denote by $\Theta_t$ the increasing
degree-one circle map with $(\Theta_t)_\#m_t=\lambda$, normalized by
\begin{equation}\label{eq:quantile-map}
 \partial_x\Theta_t=\rho_t,
 \qquad
 \partial_t\Theta_t+\alpha_t\rho_t-\partial_x\rho_t=0.
\end{equation}
The first identity, equivalent to
$(\Theta_t)_\#m_t=\lambda$, determines $\Theta_t$ up to a
time-dependent additive constant.  Writing
$J_t:=\alpha_t\rho_t-\partial_x\rho_t$, the second identity is
$\partial_t\Theta_t=-J_t$ and fixes the evolution of this constant.
It is compatible with the first because
$\partial_t\rho_t+\partial_xJ_t=0$.  Thus, once $\Theta_\tau$ is fixed,
\eqref{eq:quantile-map} determines $\Theta_t$ on $[\tau,T]$.

We write $Q_t:=\Theta_t^{-1}$ for the quantile map of $m_t$ and define
\begin{equation*}
 \mathsf a_t(y):=\rho_t\big(Q_t(y)\big)^2,
 \qquad
 \chi_t=\int_\T\rho_t^3\,dx.
\end{equation*}

\begin{proposition}[The transported Gibbs law]\label{prop:gibbs}
Let $(m,\alpha)$ be as above and let $N^{1/2}\leq\varsigma\leq N^{5/7}$.
There exist a smooth Gibbs law $\pi_{N,\varsigma}^\varepsilon$, uniformly
equivalent to $\pi_{N,\varsigma}$, and a particle control on $[\tau,T]$
whose state has law
$(Q_t^{\otimes N})_\#\pi_{N,\varsigma}^\varepsilon$ at every time $t$.
Each one-coordinate marginal is $m_t$, and the drift decomposes as
$\alpha_t(X^i_t)+\gamma^i_t$, where
\begin{align}
 \E\int_\tau^T\frac1N\sum_{i=1}^N|\gamma^i_t|^2\,dt
 &\leq\frac{C\varsigma^2}{N^2}\int_\tau^T\chi_t\,dt,
 \label{eq:gibbs-energy}\\
 \E\,d_1\big(m^N_{\bX_t},m_t\big)
 &\leq C\,(N\varsigma)^{-1/3},
 \qquad \tau\leq t\leq T.
 \label{eq:gibbs-tracking}
\end{align}
Moreover,
\begin{equation}\label{eq:gibbs-action}
 \E\int_\tau^T\frac1N\sum_{i=1}^N
 L\big(X^i_t,\alpha_t(X^i_t)+\gamma^i_t\big)\,dt
 \leq
 \int_\tau^T\int_\T L\big(x,\alpha_t(x)\big)\,m_t(dx)\,dt
 +\frac{C\varsigma^2}{N^2}\int_\tau^T\chi_t\,dt.
\end{equation}
\end{proposition}

\begin{proof}
Let $\mathcal W_N^\varepsilon$ be a smooth symmetric and diagonally
invariant convolution approximation of $\mathcal W_N$, chosen so that
$|D_i\mathcal W_N^\varepsilon|\leq1/N$ and
$\varsigma\|\mathcal W_N^\varepsilon-\mathcal W_N\|_{L^\infty}\leq1$.
Indeed, moving one coordinate by $h$ displaces a mass
$1/N$ by the distance $|h|$, so $\mathcal W_N$ is $(1/N)$-Lipschitz
in each coordinate.  Convolution with a symmetric product mollifier
preserves this bound, symmetry, and diagonal invariance, while a
sufficiently small mollification scale gives the required uniform
approximation.
The Gibbs law $\pi_{N,\varsigma}^\varepsilon$ obtained by replacing
$\mathcal W_N$ with $\mathcal W_N^\varepsilon$ in
\eqref{eq:gibbs-law} then satisfies
\begin{equation}\label{eq:gibbs-density-comparison}
 e^{-2}\leq
 \frac{d\pi_{N,\varsigma}^\varepsilon}{d\pi_{N,\varsigma}}
 \leq e^2.
\end{equation}

Let $\mathfrak q_t$ be the density of
$(Q_t^{\otimes N})_\#\pi_{N,\varsigma}^\varepsilon$, and set
\[
 \Theta_t^N(\bx):=(\Theta_t(x^1),\dots,\Theta_t(x^N)).
\]
Then
\begin{equation}\label{eq:gibbs-lift-density}
 \mathfrak q_t(\bx)=\frac1{\mathcal Z_{N,\varsigma}^\varepsilon}
 e^{-\varsigma\mathcal W_N^\varepsilon(\Theta_t^N(\bx))}
 \prod_{i=1}^N\rho_t(x^i).
\end{equation}
Since $\mathfrak q_t$ is the push-forward of a fixed law by $Q_t^{\otimes N}$,
it solves the continuity equation with coordinate velocities
\[
 b^i_t(\bx)=\partial_tQ_t\big(\Theta_t(x^i)\big)
 =\alpha_t(x^i)-\frac{\partial_x\rho_t(x^i)}{\rho_t(x^i)},
\]
where the last identity follows from \eqref{eq:quantile-map}.  On the
other hand, \eqref{eq:gibbs-lift-density} gives
\[
 \partial_{x^i}\log\mathfrak q_t
 =\frac{\partial_x\rho_t(x^i)}{\rho_t(x^i)}
 -\varsigma\rho_t(x^i)
 D_i\mathcal W_N^\varepsilon(\Theta_t^N(\bx)).
\]
Consequently, if
\[
 \gamma^i_t(\bx):=-\varsigma\rho_t(x^i)
 D_i\mathcal W_N^\varepsilon(\Theta_t^N(\bx)),
\]
then $b^i_t=\alpha_t(x^i)+\gamma^i_t-\partial_{x^i}\log\mathfrak q_t$.
Substitution into the continuity equation yields
\begin{equation}\label{eq:gibbs-FP}
 \partial_t\mathfrak q_t=\sum_{i=1}^N\partial^2_{x^ix^i}\mathfrak q_t
 -\sum_{i=1}^N\partial_{x^i}
 \big((\alpha_t(x^i)+\gamma^i_t)\mathfrak q_t\big).
\end{equation}
The drifts $\alpha_t(x^i)+\gamma^i_t$ are bounded and Borel on
$[\tau,T]\times\T^N$, so the corresponding particle system, started
with density $\mathfrak q_\tau$, is well posed \cite[Thm.~1]{Veretennikov1981},
and its time marginals are the unique distributional solution of the
uniformly parabolic equation \eqref{eq:gibbs-FP} with initial datum
$\mathfrak q_\tau$ \cite{Figalli2008}.  Since $\mathfrak q$ is such a solution, the state
has density $\mathfrak q_t$ for every $t\in[\tau,T]$.

We write $\bY_t:=\Theta_t^N(\bX_t)$.  By construction,
$\operatorname{Law}(\bY_t)=\pi_{N,\varsigma}^\varepsilon$.

The derivative $D_i\mathcal W_N^\varepsilon$ is diagonally invariant.
Using the same invariance for $\pi_{N,\varsigma}^\varepsilon$ and
averaging over the common translation, we obtain
\[
 \int_{\T^N}\sum_{i=1}^N\mathsf a_t(y^i)
 |D_i\mathcal W_N^\varepsilon|^2
 \,d\pi_{N,\varsigma}^\varepsilon(\by)
 =\chi_t\int_{\T^N}\sum_{i=1}^N
 |D_i\mathcal W_N^\varepsilon|^2
 \,d\pi_{N,\varsigma}^\varepsilon
 \leq\frac{\chi_t}{N}.
\]
Here we used $\int_\T\mathsf a_t(y+\theta)\,d\theta=\chi_t$ and
$\sum_i|D_i\mathcal W_N^\varepsilon|^2\leq1/N$.  Together with
$\rho_t(X^i_t)^2=\mathsf a_t(Y^i_t)$, this proves
\eqref{eq:gibbs-energy}.

For the action estimate, we set
\[
 \mathsf c_t(y):=D_aL\big(Q_t(y),\alpha_t(Q_t(y))\big)\,
 \rho_t\big(Q_t(y)\big).
\]
Then
$D_aL(X^i_t,\alpha_t(X^i_t))\,\gamma^i_t
=-\varsigma\mathsf c_t(Y^i_t)
D_i\mathcal W_N^\varepsilon(\bY_t)$.
Integration by parts against
$e^{-\varsigma\mathcal W_N^\varepsilon}$ and the uniformity of the
one-coordinate marginals give
\[
 \E\sum_{i=1}^N\mathsf c_t(Y^i_t)\,
 D_i\mathcal W_N^\varepsilon(\bY_t)
 =\frac1{\varsigma}\,\E\sum_{i=1}^N\mathsf c_t'(Y^i_t)
 =\frac N{\varsigma}\int_\T\mathsf c_t'(y)\,dy=0,
\]
and hence the first-order term in \eqref{eq:L-quadratic} vanishes after
summing over the particles and taking expectations.  Since each
$Y^i_t$ is uniform, each $X^i_t=Q_t(Y^i_t)$ has law $m_t$.
Thus \eqref{eq:L-quadratic} and \eqref{eq:gibbs-energy} give
\eqref{eq:gibbs-action}.

It remains to prove the tracking estimate.  For a nondecreasing
degree-one map $Q$ and $\mu\in\Pp(\T)$, let $\Gamma$ be an optimal
coupling of $\mu$ and $\lambda$.  For each
$(y,z)\in\operatorname{supp}\Gamma$, we choose lifts with
$h=z-y\in[-\frac12,\frac12]$.  Monotonicity and the
degree-one property imply
\[
 \int_0^1\dm\big(Q(y+\theta),Q(z+\theta)\big)\,d\theta
 \leq\int_0^1|Q(u+h)-Q(u)|\,du=|h|.
\]
Translating $\Gamma$ by $\theta$, pushing it forward by $Q$ in both
coordinates, and integrating gives
\[
 \int_0^1d_1\big(Q_\#(\tau_\theta)_\#\mu,Q_\#\lambda\big)\,d\theta
 \leq d_1(\mu,\lambda).
\]
We apply this bound with $Q=Q_t$ and $\mu=m^N_{\bY_t}$,
noting that $(Q_t)_\#m^N_{\bY_t}=m^N_{\bX_t}$,
$(Q_t)_\#\lambda=m_t$, and
$(\tau_\theta)_\#m^N_{\bY_t}=m^N_{\bY_t+\theta(1,\dots,1)}$.  By the
diagonal invariance of $\pi_{N,\varsigma}^\varepsilon$, each
translated configuration $\bY_t+\theta(1,\dots,1)$ has the law of
$\bY_t$, so the integrand on the left-hand side has expectation
$\E\,d_1(m^N_{\bX_t},m_t)$ for every $\theta$, while the right-hand
side is $\mathcal W_N(\bY_t)$.  Taking expectations and using
\eqref{eq:gibbs-density-comparison} and \eqref{eq:gibbs-moments}, we
obtain
\[
 \E d_1(m^N_{\bX_t},m_t)
 \leq\int_{\T^N}\mathcal W_N\,d\pi_{N,\varsigma}^\varepsilon
 \leq C(N\varsigma)^{-1/3}.
\]
This is \eqref{eq:gibbs-tracking}.
\end{proof}

\begin{proposition}\label{prop:one-dim-easy}
Uniformly in $N$, $t_0$, and $\bx$,
\begin{equation}\label{eq:one-dim-easy}
 V^N(t_0,\bx)\leq U\big(t_0,m^N_{\bx}\big)
 +CN^{-4/7}\log^{1/7}(1+N).
\end{equation}
\end{proposition}

\begin{proof}
We translate $t_0$ to $0$, writing again $T$ for the shortened
horizon, and put $\tau=N^{-4/7}\wedge T/4$.  The case $T=0$ is
immediate.

\emph{The entrance.}  We use the zero control on $[0,\tau]$ and write
$\bar X^i_t=x^i+\sqrt2W^i_t$, $\bar\mu^N_t=m^N_{\bar\bX_t}$, and
$\bar m_t=P_tm^N_{\bx}$.  The zero-drift specialization of
\eqref{eq:mollified-duality}, estimated as in
Lemma~\ref{lem:martingale}, gives
\begin{equation}\label{eq:short-heat}
 \E d_1\big(\bar\mu^N_t,\bar m_t\big)
 \leq CN^{-1/2}t^{1/4},\qquad 0<t\leq T.
\end{equation}
Integrating \eqref{eq:short-heat} on $[0,\tau]$ and using
$\tau\leq N^{-4/7}$ yields
\begin{equation}\label{eq:entrance-costs}
 \int_0^\tau\E\,d_1\big(\bar\mu^N_t,\bar m_t\big)\,dt
 \leq CN^{-17/14},
 \qquad
 \E\,d_1\big(\bar\mu^N_\tau,\bar m_\tau\big)\leq CN^{-9/14},
\end{equation}
and both bounds are $o(N^{-4/7})$.  We apply
Lemma~\ref{lem:regular-pairs} with $m_0=m^N_{\bx}$ and
$\eta=N^{-1}$.  By \eqref{eq:regular-cost}, the heat segment
$(\bar m_t)_{t\leq\tau}$ followed by the resulting classical pair
$(m,\alpha)$ has mean field cost at most
$U(0,m^N_{\bx})+CN^{-4/7}$, and its density satisfies
\eqref{eq:chi-range}.

\emph{The target process.}  We apply Proposition~\ref{prop:gibbs} to
$(m,\alpha)$ at the inverse temperature
\begin{equation}\label{eq:optimal-temperature}
 \varsigma_*:=N^{5/7}\log^{-3/7}(1+N),
\end{equation}
which balances the energy against the discrepancy below.  It lies in
$[N^{1/2},N^{5/7}]$ once $\log(1+N)\leq\sqrt N$, hence for all $N$
beyond a numerical threshold, and smaller $N$ are absorbed into the
constant of \eqref{eq:one-dim-easy}.  Since
$\tau=N^{-4/7}\wedge T/4$, \eqref{eq:chi-range} gives
\begin{equation}\label{eq:gibbs-chi-integral}
 \int_\tau^T\chi_t\,dt
 \leq C\left(T-\tau+\log\frac{T}{\tau}\right)
 \leq C\log(1+N).
\end{equation}
Indeed, if $\tau=T/4$ the logarithm is $\log4$, while otherwise
$\tau=N^{-4/7}$.  Combining \eqref{eq:optimal-temperature},
\eqref{eq:gibbs-energy}, and \eqref{eq:gibbs-chi-integral} gives
\begin{equation}\label{eq:gibbs-log-energy}
 \E\int_\tau^T\frac1N\sum_{i=1}^N|\gamma^i_t|^2\,dt
 \leq C\,\frac{\varsigma_*^2}{N^2}\,\log(1+N)
 =CN^{-4/7}\log^{1/7}(1+N).
\end{equation}
The action estimate \eqref{eq:gibbs-action} bounds the excess expected
action by the right-hand side of \eqref{eq:gibbs-log-energy}.  The
tracking estimate
\eqref{eq:gibbs-tracking} and the $d_1$-Lipschitz continuity of $F$ and
$G$ cost only
$C(N\varsigma_*)^{-1/3}=CN^{-4/7}\log^{1/7}(1+N)$, the same order.
In particular, at the entrance,
$\E\,d_1\big(m^N_{\bX_\tau},\bar m_\tau\big)
\leq CN^{-4/7}\log^{1/7}(1+N)$.

\emph{The transfer.}  An optimal matching of the two labeled
configurations realizes the distance
$d_1(\bar\mu^N_\tau,m^N_{\bX_\tau})$.  Since only finitely many
permutations are involved, a minimizing permutation can be chosen
measurably from the two configurations.  Constructing the entrance and
the target process on a product space, we choose this permutation at
time
$\tau$ and relabel the target process and its future Brownian motions
accordingly.  Thus, writing $\widehat\bX$ for the relabeled target,
\eqref{eq:entrance-costs} and the entrance bound give
\begin{equation}\label{eq:matching}
 \E D_N\leq CN^{-4/7}\log^{1/7}(1+N),
 \qquad
 D_N:=\frac1N\sum_{i=1}^N
 \dm\big(\bar X^i_\tau,\widehat X^i_\tau\big).
\end{equation}
With
$\widehat u^i_t=\alpha_t(\widehat X^i_t)+\widehat\gamma^i_t$
denoting the realized drift of the relabeled target, we set
\[
 \widetilde X^i_t
 :=\bar X^i_\tau+\widehat X^i_t-\widehat X^i_\tau,
 \qquad
 \widetilde\alpha^i_t:=\widehat u^i_t,
 \qquad \tau\leq t\leq T.
\]
The permutation is measurable at time $\tau$, so the relabeled future
noises are Brownian motions for the enlarged filtration and
$\widetilde\balpha$ is an admissible particle control.  Moreover,
translation invariance of the circle distance gives, pathwise,
\begin{equation}\label{eq:transfer-displacement}
 \frac1N\sum_{i=1}^N
 \dm\big(\widetilde X^i_t,\widehat X^i_t\big)=D_N,
 \qquad \tau\leq t\leq T.
\end{equation}
By \eqref{eq:L-space}, the boundedness of $\alpha$, Cauchy--Schwarz,
the elementary bound $\dm^2\leq C\dm$ on $\T$,
\eqref{eq:gibbs-log-energy}, and \eqref{eq:matching},
\[
 \begin{split}
 &\E\int_\tau^T\frac1N\sum_{i=1}^N
 \left|L\big(\widetilde X^i_t,\widehat u^i_t\big)
       -L\big(\widehat X^i_t,\widehat u^i_t\big)\right|dt\\
 &\qquad\leq C\,\E D_N
 +C(\E D_N)^{1/2}
 \left(\E\int_\tau^T\frac1N\sum_{i=1}^N
 |\widehat\gamma^i_t|^2\,dt\right)^{1/2}
 \leq CN^{-4/7}\log^{1/7}(1+N).
 \end{split}
\]
By \eqref{eq:transfer-displacement}, the $d_1$-distance between the two
empirical measures is at most $D_N$ at every time, so the transfer
changes the mean field costs by at most
$CN^{-4/7}\log^{1/7}(1+N)$.  Combining the
entrance costs, the target cost, and the transfer proves
\eqref{eq:one-dim-easy}.
\end{proof}

\subsection{Optimality of the exponent}\label{subsec:one-dim-lower}

For the data of the example \eqref{eq:one-dim-lower} the particle
value admits a Cole--Hopf representation.  The mechanism is that of
the optimality example of
\cite[Prop.~2.9 and Sec.~7]{DaudinDelarueJackson2024}, where the
transform reduces a terminal Wasserstein cost to a partition-function
estimate.  In our example the Wasserstein discrepancy enters as a
running cost, so the transform produces instead an $N$-body
Schr\"odinger semigroup, and the optimality statement becomes a
ground-state estimate for its generator, whose potential is the
discrepancy.  To prove the estimate we localize the discrepancy to
$R=N^{1/7}$ cells of the torus and bound the contribution of each
cell from below by a spectral inequality that is uniform in the
number of particles the cell contains.

Fix an integer $R\geq2$, let $N=R^7$, and partition the torus into
the cells $I_j:=[j\ell,(j+1)\ell)$, $j=0,\dots,R-1$, of width
$\ell:=R^{-1}$.  Let $\zeta(y):=\min\{y,1-y\}$ be the tent function
on $[0,1]$, with mean $\bar\zeta:=\int_0^1\zeta=\tfrac14$, and set
\begin{equation*}
 \zeta_j(x):=\ell\,\zeta\big((x-j\ell)/\ell\big)\ \text{ on }I_j,
 \qquad
 \zeta_j:=0\ \text{ off }I_j.
\end{equation*}
Every signed combination $\sum_j\pm\,\zeta_j$ is $1$-Lipschitz, so
Kantorovich--Rubinstein duality \eqref{eq:d1-duality} gives, for every
$\mu\in\Pp(\T)$,
\begin{equation}\label{eq:tent-sum}
 S_R(\mu):=\sum_{j=0}^{R-1}
 \Big|\int_\T\zeta_j\,d(\mu-\lambda)\Big|
 \leq d_1(\mu,\lambda).
\end{equation}
The sum $S_R$ is the localized discrepancy: each term measures the
deviation of $\mu$ from the uniform measure on a single cell.

\begin{lemma}[Cell spectral estimate]\label{lem:cell}
For every $a_->0$ there exists $c>0$ such that, for every integer
$k\geq1$, every $a\geq a_-$, every $b\in\R$, and every
$g\in H^1([0,1]^k)$,
\begin{equation}\label{eq:cell}
 \int_{[0,1]^k}|\nabla g|^2
 +a\int_{[0,1]^k}\big|\mathsf Z_k+b\big|\,|g|^2
 \geq c\int_{[0,1]^k}|g|^2,
 \qquad
 \mathsf Z_k:=\frac1{\sqrt k}\sum_{i=1}^k
 \big(\zeta(y_i)-\bar\zeta\big).
\end{equation}
\end{lemma}

\begin{proof}
Normalize $\|g\|_{L^2}=1$ and write $g=\bar g+h$ with $\int h=0$.  If
$E:=\int|\nabla g|^2$ exceeds a universal constant there is nothing to
prove; otherwise the Poincar\'e inequality on the cube gives
$\|h\|_{L^2}\leq C\sqrt E$ and $|\bar g|^2=1-\|h\|_{L^2}^2$.  For
independent uniform $Y_1,\dots,Y_k$ the variable $\mathsf Z_k$ is
centered with moments bounded uniformly in $k$, so expanding the
fourth moment gives
\[
 \E\,|\mathsf Z_k+b|^4
 \leq C\big(\E\,|\mathsf Z_k+b|^2\big)^2
 \qquad\text{uniformly in }k\text{ and }b\in\R,
\]
and the Paley--Zygmund inequality gives
$\E\,|\mathsf Z_k+b|\geq c\,\big(\E|\mathsf Z_k+b|^2\big)^{1/2}
\geq c\,\big(\operatorname{Var}\zeta(Y_1)\big)^{1/2}>0$.  Writing
$m_2:=(\E|\mathsf Z_k+b|^2)^{1/2}$, expanding
$|g|^2=|\bar g+h|^2$, discarding the nonnegative $h^2$ term, and
applying Cauchy--Schwarz to the cross term,
\[
 \int\big|\mathsf Z_k+b\big|\,|g|^2
 \geq c\,|\bar g|^2\,m_2-2\,|\bar g|\,m_2\,\|h\|_{L^2}
 \geq c'\,m_2>0
\]
provided $E$ is smaller than a suitable universal constant.  In either
case \eqref{eq:cell} holds with a constant depending only on $a_-$ and
$\zeta$.
\end{proof}

\begin{lemma}[Many-particle spectral estimate]\label{lem:many-particle}
Let $N=R^7$ and let
\begin{equation*}
 \mathcal H_N:=-\sum_{i=1}^N\partial^2_{x^ix^i}
 +\frac N2\,d_1\big(m^N_{\bx},\lambda\big)
 \qquad\text{on }L^2(\T^N).
\end{equation*}
There exists $c>0$ such that the bottom of the spectrum of
$\mathcal H_N$ is at least $cR^3$.
\end{lemma}

\begin{proof}
By \eqref{eq:tent-sum} it suffices to bound from below the operator with
potential $\frac N2S_R(m^N_{\bx})$.  Its ground energy, the
infimum of its spectrum, is the infimum of the quadratic form
\[
 \mathcal E(g):=\int_{\T^N}|\nabla g|^2
 +\frac N2\int_{\T^N}S_R\big(m^N_{\bx}\big)\,|g|^2
\]
over $g\in H^1(\T^N)$ with $\|g\|_{L^2(\T^N)}=1$.  Partition $\T^N$
into the $R^N$ boxes
$I_{j_1}\times\cdots\times I_{j_N}$,
$(j_1,\dots,j_N)\in\{0,\dots,R-1\}^N$.
Allowing $g$ to be $H^1$ on each box separately, with no
matching across the box boundaries, enlarges the class of competitors,
so this infimum can only decrease.  For such a $g$,
\[
 \mathcal E(g)=\sum_B\mathcal E_B\big(g|_B\big)
 \geq\sum_Be_B\int_B|g|^2
 \geq\min_Be_B,
\]
where the sum runs over the boxes, $\mathcal E_B$ is the form
restricted to the box $B$, and $e_B$ is its ground energy: the infimum
of $\mathcal E_B$ over $H^1(B)$ functions with unit $L^2(B)$-norm.  It
therefore suffices to prove that $e_B\geq cR^3$ for every box.

Fix a box, let $k_j$ be the number of coordinates in $I_j$, and put
$k_0:=N/R=R^6$.  Rescaling the
coordinates of each cell to $[0,1]$ multiplies the kinetic energy by
$\ell^{-2}=R^2$, while the potential separates over the cells:
\begin{equation*}
 \frac N2\,S_R\big(m^N_{\bx}\big)
 =\sum_{j=0}^{R-1}\frac1{2R}
 \Big|\sum_{i=1}^{k_j}\zeta(y^i_j)-k_0\bar\zeta\Big|.
\end{equation*}
Thus its quadratic form is the sum of the $R$ cell forms.
Applying the lower bound for each cell form with the remaining
coordinates frozen, and then integrating in those coordinates, shows
that the ground energy of the box is at least the sum of the cell
ground energies.
Since $\sum_jk_j=N$, at least $R/2$ of the indices satisfy
$k_j\leq2k_0$.  Fix such a $j$ and write $k=k_j$.  If $k\leq k_0/4$,
then $0\leq\zeta\leq\frac12$ and $\bar\zeta=\frac14$ give
$\big|\sum_i\zeta(y^i)-k_0\bar\zeta\big|\geq k_0/8$, so the $j$-th cell
potential is at least $R^5/16$.  If $k_0/4\leq k\leq2k_0$, the $j$-th
cell operator is
\[
 R^2\Big(-\Delta_y+a_k\,\big|\mathsf Z_k+b_k\big|\Big),
 \qquad
 a_k=\frac{\sqrt k}{2R^3}\in\Big[\frac14,\frac1{\sqrt2}\Big],
 \qquad
 b_k=\frac{(k-k_0)\,\bar\zeta}{\sqrt k},
\]
and Lemma~\ref{lem:cell} bounds its ground energy from below by $cR^2$.
In both cases the cell contributes at least $cR^2$ to the ground
energy.  Since at least $R/2$ cells do so and the remaining cell
operators are nonnegative, the box has ground energy at least $cR^3$,
as required.
\end{proof}

\begin{proof}[Proof of Theorem~\ref{thm:one-dim}]
The bounds \eqref{eq:one-dim-rate} and
\eqref{eq:one-dim-hard-statement} follow from
Propositions~\ref{prop:one-dim-hard} and~\ref{prop:one-dim-easy}.  It
remains to prove the optimality statement \eqref{eq:one-dim-lower}.

For $H(x,p)=\frac12p^2$ the Lagrangian is $L(a)=\frac12a^2$, and the
Cole--Hopf and Feynman--Kac formulas give
\begin{equation}\label{eq:cole-hopf}
 V^N(0,\bx)=-\frac2N\log\mathcal Z_N(\bx),
 \qquad
 \mathcal Z_N(\bx):=\E_{\bx}\exp\Big(-\frac N2\int_0^T
 d_1\big(\mu^N_t,\lambda\big)\,dt\Big)
 =\big(e^{-T\mathcal H_N}\1\big)(\bx),
\end{equation}
the expectation being over the uncontrolled particles
$dX^i_t=\sqrt2\,dW^i_t$.  Since $\T^N$ has unit volume,
$\|\1\|_{L^2(\T^N)}=1$, and the spectral theorem together with
Lemma~\ref{lem:many-particle} gives
\begin{equation}\label{eq:partition-average}
 \int_{\T^N}\mathcal Z_N(\bx)\,d\bx
 =\int_{\T^N}\big(e^{-T\mathcal H_N}\1\big)(\bx)\,d\bx
 =\big\langle\1,e^{-T\mathcal H_N}\1\big\rangle_{L^2(\T^N)}
 \leq e^{-cTR^3}.
\end{equation}
Fix $\delta>0$ and let
$\mathcal A_{N,\delta}
:=\{\bx\in\T^N:\ d_1(m^N_{\bx},\lambda)\leq\delta R^{-4}\}$.
Lemma~\ref{lem:gibbs-small-ball}, applied with $N=R^7$ and
$r=\delta R^{-4}=\delta N^{-4/7}$, shows that for all sufficiently
large $R$ depending on $\delta$,
\begin{equation}\label{eq:entrance-volume}
 |\mathcal A_{N,\delta}|\geq e^{-C_\delta R}.
\end{equation}

By \eqref{eq:partition-average} and \eqref{eq:entrance-volume}, the
average of $\mathcal Z_N$ over $\mathcal A_{N,\delta}$ is at most
$\exp\big(-cTR^3+C_\delta R\big)$, so there exists
$\bx_R\in\mathcal A_{N,\delta}$ satisfying this bound, and
\eqref{eq:cole-hopf} gives
\begin{equation}\label{eq:particle-lower}
 V^N(0,\bx_R)\geq2cT\,R^{-4}-C_\delta\,R^{-6}.
\end{equation}

For the mean field problem started from $m^N_{\bx_R}$, the uncontrolled
heat flow is admissible with zero action.  Since
$\bx_R\in\mathcal A_{N,\delta}$, the $d_1$-contractivity of the heat
semigroup gives
\begin{equation}\label{eq:mean-upper}
 0\leq U\big(0,m^N_{\bx_R}\big)
 \leq\int_0^Td_1\big(P_tm^N_{\bx_R},\lambda\big)\,dt
 \leq T\,\delta\,R^{-4}.
\end{equation}
Choosing $\delta$ smaller than the leading constant in
\eqref{eq:particle-lower} and then $R$ sufficiently large, the
difference of \eqref{eq:particle-lower} and \eqref{eq:mean-upper} is at
least $c\,T\,R^{-4}=c\,T\,N^{-4/7}$, which proves
\eqref{eq:one-dim-lower}.
\end{proof}

\section{Extension to additive common noise}\label{sec:common-noise}

We close with the extension to additive common noise, a classical
feature of mean field models
\cite{CardaliaguetDelarueLasryLions2019,CarmonaDelarue2018II}.
Let $W^0,W^1,\dots,W^N$ be independent $d$-dimensional Brownian motions and
fix $\sigma_0\in\R$.  The particle dynamics are now
\begin{equation}\label{eq:common-particle-SDE}
 dX^i_t=\alpha^i_t\,dt+\sqrt2\,dW^i_t+\sigma_0\,dW^0_t,
 \qquad X^i_{t_0}=x^i,
\end{equation}
and $J_N^{\sigma_0}$ and $V^{N,\sigma_0}$ are defined as in
\eqref{eq:JN-definition}--\eqref{eq:VN-definition}, with
\eqref{eq:particle-SDE} replaced by \eqref{eq:common-particle-SDE}.  The
controls may be progressively measurable with respect to any common
filtration for which $W^0,W^1,\dots,W^N$ are independent Brownian motions
on $[t_0,T]$,
exactly as in Section~\ref{subsec:problems}.  When sections
of a control in the idiosyncratic variables are needed
(Lemma~\ref{lem:common-sections} below), we work on product Wiener space
with the raw canonical filtrations.  The usual augmentations are restored
only after progressively measurable Borel representatives have been chosen.

The limiting problem is posed in its strong formulation
\cite{CarmonaDelarue2018II,DjetePossamaiTan2022}: the value
$U^{\sigma_0}$ is defined on a canonical space carrying independent random
variables $\xi,W,W^0$, with $\xi$ of law $m_0$ and $W,W^0$ Brownian
motions on $[t_0,T]$.  Given a square-integrable control $\alpha$ adapted to their joint
filtration, let
\begin{equation*}
 dX_t=\alpha_t\,dt+\sqrt2\,dW_t+\sigma_0\,dW^0_t,
 \qquad X_{t_0}=\xi,
 \qquad m_t=\mathcal L(X_t\mid\mathcal F^0_t),
\end{equation*}
where $(\mathcal F^0_t)$ is the filtration of $W^0$.  Then
\begin{equation}\label{eq:common-U-definition}
 U^{\sigma_0}(t_0,m_0)
 :=\inf_\alpha\E\bigg[\int_{t_0}^T
 \Big(L(X_t,\alpha_t)+F(m_t)\Big)dt+G(m_T)\bigg].
\end{equation}
We use the corresponding stochastic Fokker--Planck formulation
(cf.\ the conditional mimicking of
\cite{LackerShkolnikovZhang2020}): a bounded
$\mathcal F^0_t$-progressively measurable field $\beta_t(x)$ is admissible
if the associated conditional law satisfies, for every smooth $\varphi$,
\begin{equation}\label{eq:common-FP}
 d\int_{\T^d}\varphi\,dm_t
 =\int_{\T^d}\Big(\Big(1+\frac{\sigma_0^2}{2}\Big)\Delta\varphi
 +\beta_t\cdot\nabla\varphi\Big)\,dm_t\,dt
 +\sigma_0\Big(\int_{\T^d}\nabla\varphi\,dm_t\Big)\cdot dW^0_t .
\end{equation}
The proof of Proposition~\ref{prop:common-comparison} converts any such
field into a strong control in \eqref{eq:common-U-definition} of the
same cost.

\begin{theorem}[Additive common noise]\label{thm:common-noise}
Let Assumption~\ref{ass:standing} hold.  There exists a constant $C$,
independent of $N$ and $\sigma_0$, such that
\begin{equation}\label{eq:common-two-sided}
 \big|U^{\sigma_0}(t,m^N_{\bx})-V^{N,\sigma_0}(t,\bx)\big|
 \leq C r_{N,d}
\end{equation}
for every $N\geq1$, $\sigma_0\in\R$, and
$(t,\bx)\in[0,T]\times(\T^d)^N$.  For $d\geq2$ the rate $r_{N,d}$
cannot be improved, already for $\sigma_0=0$.
\end{theorem}

\begin{proposition}[Common-noise comparison]\label{prop:common-comparison}
Suppose that $\balpha$ is a bounded particle control in
\eqref{eq:common-particle-SDE} on product Wiener space, realized as a
nonanticipating Borel functional of $(W^0,W^1,\dots,W^N)$, and that
$|\alpha^i|\leq M$.  Then
\begin{equation}\label{eq:common-comparison}
 U^{\sigma_0}(t_0,m^N_{\bx})
 \leq J_N^{\sigma_0}(t_0,\bx;\balpha)+C_Mr_{N,d},
\end{equation}
where $C_M$ is independent of $N$ and $\sigma_0$.
\end{proposition}

For the proofs, fix a time grid as in Section~\ref{sec:shadow} and choose
raw-progressive Borel representatives of the controls in
Proposition~\ref{prop:common-comparison}.  Using the standard random-translation
reduction for additive common noise
\cite{LackerWebster2015,CardaliaguetDelarueLasryLions2019}, set
\[
 \mathsf w_t:=\sigma_0(W^0_t-W^0_{t_0}),
 \qquad Y^i_t:=X^i_t-\mathsf w_t,
 \qquad \widetilde\mu^N_t:=\frac1N\sum_{i=1}^N\delta_{Y^i_t},
\]
so that $dY^i_t=\alpha^i_t\,dt+\sqrt2\,dW^i_t$, and apply the shadow-flow
construction of Section~\ref{sec:shadow}, on the given time grid, to the
shifted particles $Y^i$ and the controls $\alpha^i$.  Denote the resulting
flow, components, drift, and provisional cost by $\widetilde\nu$,
$\widetilde\nu_i$, $\widetilde\beta$, and $\widetilde c$.  That the
controls may also depend on $W^0$ does not affect any step of the
construction.  Write $\omega^{\rm id}=(\omega^1,\dots,\omega^N)$ for the
idiosyncratic Brownian coordinates.

\begin{lemma}[Progressive sections]\label{lem:common-sections}
The maps
\[
 (t,\omega^0,\omega^{\rm id})\longmapsto
 \big(\widetilde\nu_t,\widetilde\nu_{1,t},\dots,
 \widetilde\nu_{N,t}\big)
\]
and the field
\[
 (t,y,\omega^0,\omega^{\rm id})\longmapsto\widetilde\beta_t(y)
\]
are jointly measurable.  Moreover, outside a set of idiosyncratic Wiener
measure zero, the measure-valued section obtained by fixing
$\omega^{\rm id}$ is adapted to the raw filtration of $W^0$,
$\widetilde\beta^{\omega^{\rm id}}$ is measurable with respect to
$\operatorname{Prog}(\mathbb F^0)\otimes\mathcal B(\T^d)$, and, for almost
every $\omega^0$,
\begin{equation}\label{eq:section-FP}
 \partial_t\widetilde\nu_t
 =\Delta\widetilde\nu_t-\diver(\widetilde\beta_t\widetilde\nu_t),
 \qquad \widetilde\nu_{t_0}=m^N_{\bx},
\end{equation}
in the sense of distributions.
\end{lemma}

\begin{proof}
For every $s\leq T$, the restriction of a raw-progressive control to
$[t_0,s]\times\Omega^0\times\Omega^{\rm id}$ is measurable with respect to
$\mathcal B([t_0,s])\otimes\mathcal F^0_s\otimes
\mathcal F^{\rm id}_s$.  Its section at any fixed
$\omega^{\rm id}$ is therefore $\mathbb F^0$-progressive.  The shifted
states have the same property.  Proceeding recursively over the time grid,
the Borel selector in Lemma~\ref{lem:selection} makes each optimal
decomposition jointly measurable at the left endpoint, while
\eqref{eq:component-definition} and \eqref{eq:beta-definition} preserve
joint measurability on the following interval.  The recursion in fact
preserves the stronger restriction property of the first paragraph:
for every $s$, the restrictions to $[t_0,s]$ are measurable with
respect to $\mathcal B([t_0,s])\otimes\mathcal F^0_s\otimes
\mathcal F^{\rm id}_s$, and $\mathcal B(\T^d)$ for the drift.  Fixing
$\omega^{\rm id}$, the measure-valued sections are therefore adapted
and the drift section is
$\operatorname{Prog}(\mathbb F^0)\otimes\mathcal B(\T^d)$-measurable.
Finally, Lemma~\ref{lem:FP} applied
on product Wiener space and Fubini's theorem give a single null set in the
idiosyncratic coordinate outside which \eqref{eq:section-FP} holds for
almost every $\omega^0$.
\end{proof}

\begin{lemma}[Causal solution map]\label{lem:causal-solution}
Let $\Omega^0=C_0([0,T];\R^d)$ be canonical Wiener space with its raw
filtration, let $m_0\in\Pp(\T^d)$, and let
$b:[t_0,T]\times\Omega^0\times\T^d\to\R^d$ be bounded and measurable
with respect to
$\operatorname{Prog}(\mathbb F^0)\otimes\mathcal B(\T^d)$.  There exists
a Borel map
\[
 \mathsf S_b:\T^d\times C_0([0,T];\R^d)\times\Omega^0
 \longrightarrow C([t_0,T];\T^d)
\]
such that, whenever $\xi,W,W^0$ are independent, $\xi$ has law $m_0$, and
$W,W^0$ are Brownian motions,
$Z:=\mathsf S_b(\xi,W,W^0)$ is the unique strong solution of
\begin{equation}\label{eq:random-environment-SDE}
 dZ_t=b_t(W^0,Z_t)dt+\sqrt2\,dW_t,
 \qquad Z_{t_0}=\xi.
\end{equation}
The solution is causal up to null sets: for every $t$, the path
$Z|_{[t_0,t]}$ is measurable with respect to the completion, under the
law of the inputs, of the $\sigma$-field generated by
$(\xi,W|_{[0,t]},W^0|_{[0,t]})$.
\end{lemma}

\begin{proof}
All input and output path spaces are standard Borel.  Girsanov's theorem,
starting from $\xi+\sqrt2(W-W_{t_0})$, gives a weak solution of
\eqref{eq:random-environment-SDE}.  Boundedness of $b$ gives Novikov's
condition, also conditionally on $(\xi,W^0)$.  Thus the marginal law of
$(\xi,W^0)$ is unchanged, while the corrected driver and $W^0$ form a
Brownian motion with zero cross-variation in a filtration containing $\xi$.
The three inputs have the prescribed product law.  This weak solution is
temporally compatible with $(\xi,W,W^0)$.

For uniqueness, place two jointly temporally compatible solutions over the
same input.  Compatibility keeps $W$ Brownian in their joint filtration.
Moreover, the increments of $W$ after any time $s$ are independent of
the joint past at time $s$ jointly with the future increments of $W^0$
(by compatibility), and are independent of those increments themselves
(the inputs being independent).  Hence they are independent of the
joint past enlarged by the whole path of $W^0$, so that $W$ remains a
Brownian motion in the joint filtration under the regular conditional
probability given $W^0$.
After conditioning on $W^0$, for almost every $\omega^0$ both solve the SDE
with the same Brownian motion and the deterministic bounded Borel drift
$(t,x)\mapsto b_t(\omega^0,x)$.  After lifting the drift periodically to
$\R^d$, pathwise uniqueness from \cite[Thm.~1]{Veretennikov1981} shows that
the two solutions agree.  By
\cite[Lem.~2.10 and Thm.~1.5]{Kurtz2014}, pointwise uniqueness in the
compatible class and weak existence yield a strong solution
$Z=\mathsf S_b(\xi,W,W^0)$ for one Borel map $\mathsf S_b$, jointly in all
three inputs.  By \cite[Prop.~2.13]{Kurtz2014}, this strong compatible
solution is adapted to the completed natural filtration of the inputs,
which is the asserted causality.
\end{proof}

\begin{proof}[Proof of Proposition~\ref{prop:common-comparison}]
The proof of Proposition~\ref{prop:tracking} applies to the shifted system
without change: the martingale $M^\varepsilon$ in
\eqref{eq:martingale-definition} contains
only the independent noises $W^1,\dots,W^N$, while $W^0$ merely enters the
adapted integrands.  Hence
\begin{equation}\label{eq:common-tracking}
 \sup_{t_0\leq t\leq T}
 \E d_1(\widetilde\nu_t,\widetilde\mu^N_t)
 \leq C_M(r_{N,d}+\sqrt h),
\end{equation}
and Lemma~\ref{lem:provisional} holds with tildes throughout.

By Lemma~\ref{lem:common-sections}, for almost every
fixed $\omega^{\rm id}$ the section
$(\widetilde\nu^{\omega^{\rm id}},
\widetilde\beta^{\omega^{\rm id}})$ is $\mathbb F^0$-adapted and satisfies
\eqref{eq:section-FP}.  For such a section define
\begin{equation}\label{eq:translated-shadow}
 \nu^{\omega^{\rm id}}_t
 :=(\tau_{\mathsf w_t})_\#\widetilde\nu^{\omega^{\rm id}}_t,
 \qquad
 \beta^{\omega^{\rm id}}_t(x)
 :=\widetilde\beta^{\omega^{\rm id}}_t(x-\mathsf w_t).
\end{equation}
We suppress the superscript $\omega^{\rm id}$ in the estimates below.
It\^o's formula for the translation by $\mathsf w$ and \eqref{eq:section-FP} show
that the pair in \eqref{eq:translated-shadow} is $\mathbb F^0$-adapted and
satisfies \eqref{eq:common-FP}.

We next verify that it is admissible for the strong formulation
\eqref{eq:common-U-definition}.  On a product extension carrying
$\xi,W,W^0$, with $\xi$ of law $m^N_{\bx}$ and $\xi,W,W^0$ independent,
consider
\[
 dZ_t=\widetilde\beta^{\omega^{\rm id}}_t(Z_t)dt+\sqrt2dW_t,
 \qquad Z_{t_0}=\xi.
\]
Lemma~\ref{lem:causal-solution}, applied with
$b=\widetilde\beta^{\omega^{\rm id}}$, gives a strong solution through
one Borel map of $(\xi,W,W^0)$, adapted to the completed joint
filtration of the inputs.  Conditional on
$\mathcal F^0_T$, its law solves \eqref{eq:section-FP}.  Uniqueness for the
Fokker--Planck equation with bounded measurable drift and nondegenerate
constant diffusion \cite{Figalli2008} identifies its
time-$t$ marginal with $\widetilde\nu^{\omega^{\rm id}}_t$.  Since the
latter is $\mathcal F^0_t$-measurable, the tower property gives the same
identity conditional on $\mathcal F^0_t$.  Therefore, setting
$X_t=Z_t+\mathsf w_t$,
\begin{equation*}
 \mathcal L(X_t\mid\mathcal F^0_t)=\nu^{\omega^{\rm id}}_t,
 \qquad
 \E\big[L(X_t,\beta_t(X_t))\mid\mathcal F^0_t\big]
 =\int_{\T^d}L(x,\beta_t(x))\nu_t(dx).
\end{equation*}
Thus $(\nu^{\omega^{\rm id}},\beta^{\omega^{\rm id}})$ is a competitor in
\eqref{eq:common-U-definition}.

Translations are isometries of $(\T^d,\dm)$, and
$\mu^N_t=(\tau_{\mathsf w_t})_\#\widetilde\mu^N_t$.  Consequently
\begin{equation*}
 d_1(\nu_t,\mu^N_t)
 =d_1(\widetilde\nu_t,\widetilde\mu^N_t).
\end{equation*}
Moreover, the proof of Lemma~\ref{lem:action}, applied after the same
translation to both arguments of $L$, gives
\begin{equation}\label{eq:common-action-comparison}
 \int_{\T^d}L(x,\beta_t(x))\nu_t(dx)
 \leq\frac1N\sum_{i=1}^NL(X^i_t,\alpha^i_t)
 +C_M\widetilde c_t.
\end{equation}
Indeed, convexity treats the drift variable exactly as before, while
the Lipschitz bound of Lemma~\ref{lem:lagrangian}(iii) and
$\dm(z+\mathsf w_t,Y^i_t+\mathsf w_t)=\dm(z,Y^i_t)$ treat the space variable.

For almost every fixed $\omega^{\rm id}$, use
$(\nu^{\omega^{\rm id}},\beta^{\omega^{\rm id}})$ in
\eqref{eq:common-U-definition} and take expectation only over $W^0$.
Then integrate the resulting inequality over $\omega^{\rm id}$.  Combining
\eqref{eq:common-tracking}--\eqref{eq:common-action-comparison} with the
$d_1$-Lipschitz bounds for $F$ and $G$, exactly as in the proof of
Proposition~\ref{prop:comparison}, yields
\[
 U^{\sigma_0}(t_0,m^N_{\bx})
 \leq J_N^{\sigma_0}(t_0,\bx;\balpha)
 +C_M(r_{N,d}+\sqrt h).
\]
Letting $h\downarrow0$ proves \eqref{eq:common-comparison}.
\end{proof}

\begin{proof}[Proof of Theorem~\ref{thm:common-noise}]
We first prove the hard inequality.  Smooth $F$ and $G$ as in
Lemma~\ref{lem:smoothing}.  For smooth costs, the common-noise particle value
is the classical solution of
\begin{multline}\label{eq:common-HJBN}
 -\partial_tv-\sum_{i=1}^N\Delta_{x^i}v
 -\frac{\sigma_0^2}{2}\sum_{i,j=1}^N
 \operatorname{tr}D^2_{x^ix^j}v
+\frac1N\sum_{i=1}^NH(x^i,ND_{x^i}v)
 =F(m^N_{\bx}),
\\
 v(T,\cdot)=G(m^N_{\cdot}).
\end{multline}
Equation \eqref{eq:common-HJBN} is again semilinear and uniformly parabolic
with smooth data (the idiosyncratic Laplacians provide the ellipticity,
uniformly in $\sigma_0$), and classical solvability and the verification
argument are as for Lemma~\ref{lem:finite}.  The gradient and feedback
estimates are
\cite[Lem.~3.1 and Rem.~3.2]{CardaliaguetDaudinJacksonSouganidis2023},
extended to the common-noise setting, with the same proof, at the start of
\cite[Sec.~4]{CardaliaguetDaudinJacksonSouganidis2023}, with the
common-noise parameter there given by $a_0=\sigma_0^2/2$.  The proof uses
the costs only through their $d_1$-Lipschitz constants, so the resulting
optimal feedback is bounded by $M_*$, uniformly in $N$, the smoothing
parameter, and $\sigma_0$.  Strong well-posedness of the closed-loop system
follows from \cite[Thm.~1]{Veretennikov1981}.
Applying
Proposition~\ref{prop:common-comparison} to this feedback and then using the
stability argument from the proof of Theorem~\ref{thm:main} gives
\begin{equation}\label{eq:common-hard}
 U^{\sigma_0}(t,m^N_{\bx})
 \leq V^{N,\sigma_0}(t,\bx)+Cr_{N,d}.
\end{equation}

We turn to the easy inequality.  We take an arbitrary admissible control in
\eqref{eq:common-U-definition} with $m_0=m^N_{\bx}$.  On canonical space it
has, after modification on a $dt\otimes d\Prob$-null set, a raw-progressive
factorization
\begin{equation}\label{eq:control-factorization}
 \alpha_t=a\big(t,\xi,W|_{[t_0,t]},W^0|_{[t_0,t]}\big)
\end{equation}
for a Borel nonanticipating map $a$.  This is the same elementary
progressive-section argument as in Lemma~\ref{lem:common-sections}.  For each
$i=1,\dots,N$, define
\[
 \alpha^i_t:=a\big(t,x^i,W^i|_{[t_0,t]},W^0|_{[t_0,t]}\big)
\]
and let $X^i$ solve \eqref{eq:common-particle-SDE} with this control.  These
controls are admissible: since the law $m^N_{\bx}$ of $\xi$ charges
each $x^i$ with mass at least $1/N$, their square-integrability follows by
disintegrating that of $\alpha$ with respect to $\xi$.

Write $m^i_t=\mathcal L(X^i_t\mid\mathcal F^0_t)$.  Conditional on
$\mathcal F^0_t$, the pairs $(X^i_t,\alpha^i_t)$ are independent, though not
necessarily identically distributed.  The mixture representation of the
law corresponding to \eqref{eq:control-factorization} gives
\begin{equation}\label{eq:conditional-mixture}
 m_t=\frac1N\sum_{i=1}^Nm^i_t,
 \qquad
 \E L(X_t,\alpha_t)
 =\frac1N\sum_{i=1}^N\E L(X^i_t,\alpha^i_t).
\end{equation}

By Lemma~\ref{lem:empirical}, applied conditionally on $\mathcal F^0_t$
with the conditional laws $m^i_t$ and their mixture $m_t$ from
\eqref{eq:conditional-mixture},
\begin{equation}\label{eq:common-conditional-empirical}
 \E\left[d_1\left(\frac1N\sum_{i=1}^N\delta_{X^i_t},m_t\right)
 \middle|\mathcal F^0_t\right]\leq Cr_{N,d}.
\end{equation}
By \eqref{eq:conditional-mixture}, the averaged action cost agrees exactly
with that of the mean field control.  The $d_1$-Lipschitz continuity of $F$
and $G$, together with \eqref{eq:common-conditional-empirical}, yields for
the replicated particle control
\[
 J_N^{\sigma_0}(t,\bx;\balpha)
 \leq \E\bigg[\int_t^T\big(L(X_s,\alpha_s)+F(m_s)\big)ds+G(m_T)\bigg]
 +C r_{N,d}.
\]
Taking the infimum over mean field controls gives
\begin{equation}\label{eq:common-easy}
 V^{N,\sigma_0}(t,\bx)
 \leq U^{\sigma_0}(t,m^N_{\bx})+Cr_{N,d}.
\end{equation}
Together, \eqref{eq:common-hard} and \eqref{eq:common-easy} prove
\eqref{eq:common-two-sided}.
\end{proof}

\begin{remark}\label{rem:common-scope}
The sharp one-dimensional rate of Theorem~\ref{thm:one-dim} involves additional
difficulties in this setting, and we do not pursue it here.
\end{remark}

\subsection*{Acknowledgments}
The author thanks Joe Jackson and Alpár Mészáros for helpful comments on the first version of the paper.

\appendix
\addcontentsline{toc}{section}{Appendices}
\addtocontents{toc}{\protect\setcounter{tocdepth}{0}}

\section{Auxiliary lemmas}\label{app:construction}

\subsection{Deferred proofs}\label{app:deferred}

\begin{proof}[Proof of Lemma~\ref{lem:selection}]
Consider the compact metric space
$\mathcal D_N(\nu):=\{(\pi_i)_{i=1}^N:\pi_i\geq0,\ \pi_i(\T^d)=1/N,\
\sum_i\pi_i=\nu\}$, a closed subset of the compact space of $N$-tuples of
sub-probability measures endowed with the weak topology, and the continuous
cost
$\mathcal C(\pi,\by):=\sum_i\int\dm(y,y^i)\pi_i(dy)$.

Decompositions correspond exactly to
couplings.  Given $\pi\in\mathcal D_N(\nu)$, the measure
$\gamma:=\sum_i\pi_i\otimes\delta_{y^i}$ is a coupling of $\nu$ and
$m^N_{\by}$ with $\int\dm\,d\gamma=\mathcal C(\pi,\by)$.  Conversely, given a
coupling $\gamma$ of $\nu$ and $m^N_{\by}$, write
$m^N_{\by}=\sum_{x\in S}c_x\delta_x$ with $S$ the set of distinct points among
$y^1,\dots,y^N$ and $c_x=\#\{i:y^i=x\}/N$, and set
$\pi_i:=\frac{1}{N c_{y^i}}\,\gamma(\cdot\times\{y^i\})$.  Then
$\pi\in\mathcal D_N(\nu)$ and $\mathcal C(\pi,\by)=\int\dm\,d\gamma$.  Since
optimal couplings exist and realize $d_1$ (Kantorovich duality on a compact
space, \cite[Chap.~4--5]{Villani2009}), the minimum of
$\mathcal C(\cdot,\by)$ on $\mathcal D_N(\nu)$ equals $d_1(\nu,m^N_{\by})$
and is attained.

The constraint correspondence $(\nu,\by)\mapsto\mathcal D_N(\nu)$ has
nonempty compact values and closed graph (the defining constraints pass
to weak limits), so it is upper hemicontinuous into a compact space,
hence weakly measurable.  Since $\mathcal C$ is jointly continuous, the
measurable maximum theorem \cite[Thm.~18.19]{AliprantisBorder2006}
provides a Borel measurable selector of the argmin, as claimed.
\end{proof}

\begin{proof}[Proof of Lemma~\ref{lem:FP}]
Work on the full-measure event on which $t\mapsto\alpha_t(\omega)$ is Borel
with $|\alpha^i_t|\leq M$ for a.e.\ $t$.  Fix $\varphi\in C^\infty(\T^d)$,
$k$, and $i$, and consider on $[s_k,s_{k+1})$ the function
$g(t):=\int_{\T^d}\varphi\,d\nu_{i,t}$.  Using
\eqref{eq:component-definition}, self-adjointness of $P_\tau$, and the
commutation of $P_\tau$ with translations,
\[
 g(t)=\int_{\T^d}(P_{t-s_k}\varphi)\big(z+A^k_i(t)\big)\,\pi_{i,k}(dz) .
\]
The map $(t,z)\mapsto(P_{t-s_k}\varphi)(z+A^k_i(t))$ is Lipschitz in $t$
(uniformly in $z$) and continuously differentiable in $t$ for a.e.\ $t$, with
\[
 \partial_t\Big[(P_{t-s_k}\varphi)\big(z+A^k_i(t)\big)\Big]
 =(P_{t-s_k}\Delta\varphi)\big(z+A^k_i(t)\big)
 +\alpha^i_t\cdot(P_{t-s_k}\nabla\varphi)\big(z+A^k_i(t)\big),
\]
since $\partial_tP_\tau\varphi=\Delta P_\tau\varphi=P_\tau\Delta\varphi$ and
$\frac{d}{dt}A^k_i(t)=\alpha^i_t$ a.e.  Integrating against $\pi_{i,k}$ and
undoing the same identifications,
\[
 g'(t)=\int_{\T^d}\big(\Delta\varphi
 +\alpha^i_t\cdot\nabla\varphi\big)\,d\nu_{i,t}
 \qquad\text{for a.e.\ }t\in(s_k,s_{k+1}),
\]
with $g$ absolutely continuous on $[s_k,s_{k+1})$.  Summing over $i$ and
recalling $q_t=\beta_t\nu_t$ from \eqref{eq:beta-definition}, the map
$t\mapsto\int_{\T^d}\varphi\,d\nu_t$ is absolutely continuous on each grid
interval with derivative
$\int_{\T^d}\big(\Delta\varphi+\beta_t\cdot\nabla\varphi\big)\,d\nu_t$
a.e.  Since $t\mapsto\nu_t$ is continuous across the grid times (the
recoupling changes only the decomposition, not the measure), the same
holds on all of $[t_0,T]$.  Equivalently, \eqref{eq:shadow-FP} holds
against every product test function $\eta(t)\varphi(y)$ with
$\eta\in C^\infty_c((t_0,T))$, hence, by density of sums of such
products, in the sense of distributions.  The initial condition is
\eqref{eq:shadow-initial}.  Finally, $\beta$ is bounded by $M$ and jointly measurable, and
$t\mapsto\nu_t$ is continuous, so $(\nu^\omega,\beta^\omega)$ is admissible in
\eqref{eq:U-definition}.
\end{proof}

\begin{lemma}[Regularity of the smoothed noise]\label{lem:noise-version}
For every $\varepsilon>0$ and $t\in[t_0,T]$, the random field
$M^\varepsilon_t$ of \eqref{eq:martingale-definition} admits a version
that is continuous in $y$ and jointly measurable in $(y,\omega)$.
\end{lemma}

\begin{proof}
For each $y$, the integrand
$p_{t-s+\varepsilon}(y-X^i_s)$ is bounded and progressive.  For
$y,y'\in\T^d$, the quadratic variation of the martingale (in the upper
integration limit) defining $M^\varepsilon_t(y)-M^\varepsilon_t(y')$ is
bounded, deterministically, by
\[
 \frac{2T}{N}\sup_{0\leq\tau\leq T}
 \big\|p_{\tau+\varepsilon}(y-\cdot)-p_{\tau+\varepsilon}(y'-\cdot)
 \big\|_{L^\infty}^2
 \leq C_{\varepsilon}\,\dm(y,y')^2 ,
\]
by the mean value theorem, $\nabla p_{\tau+\varepsilon}$ being bounded
uniformly over $\tau\in[0,T]$.  Hence, by the Burkholder--Davis--Gundy
inequality, for every $m\geq1$,
\[
 \E\big|M^\varepsilon_t(y)-M^\varepsilon_t(y')\big|^{2m}
 \leq C_{\varepsilon,m}\,\dm(y,y')^{2m} .
\]
Choosing $2m>d$, the Kolmogorov continuity criterion provides a
modification of $y\mapsto M^\varepsilon_t(y)$ that is continuous in $y$
and jointly measurable.
\end{proof}

\begin{proof}[Proof of Lemma~\ref{lem:regular-pairs}]
Let $F^n,G^n$ be the approximations of
Lemma~\ref{lem:smoothing}, and set
\[
 \delta_n:=T\|F^n-F\|_{L^\infty}+\|G^n-G\|_{L^\infty}.
\]
For fixed $n$, the smooth Fourier cylinders $F^n,G^n$ are
$d_1$-semiconcave and have spatially $C^2$ linear derivatives.  Hence
\cite[Lem.~3.1]{DaudinDelarueJackson2024} gives an optimal pair
$(m^n,\alpha^n)$ for $U^n(0,m_0)$ such that $\alpha^n$ is continuously
differentiable in space and
\begin{equation}\label{eq:regular-feedback}
 \|\alpha^n\|_{L^\infty}
 +\sup_{0\leq t<T}\sqrt{T-t}\,
   \|D_x\alpha^n_t\|_{L^\infty}\leq C.
\end{equation}
The constant is independent of $n$, since the cited estimate uses only
the $d_1$-Lipschitz constants of $F^n,G^n$, which are $L_F,L_G$ by
Lemma~\ref{lem:smoothing}(b).

Let $Z$ be the optimal diffusion with feedback $\alpha^n$, and, using
the same Brownian motion and initial condition, define
\[
 d\widetilde Z_t
 =\1_{[\tau,T]}(t)\alpha^n_t(\widetilde Z_t)\,dt+\sqrt2\,dW_t,
 \qquad \operatorname{Law}(\widetilde Z_0)=m_0.
\]
Then $\operatorname{Law}(\widetilde Z_\tau)=P_\tau m_0$, and
\eqref{eq:regular-feedback} followed by Gr\"onwall's inequality yields
\begin{equation}\label{eq:regular-coupling}
 \sup_{0\leq t\leq T}
 \E\,\dm(Z_t,\widetilde Z_t)\leq C\tau.
\end{equation}

We write $\widetilde m_t=\operatorname{Law}(\widetilde Z_t)$.  On
$[0,\tau]$ the action changes by at most $C\tau$.  On $[\tau,T]$,
\eqref{eq:L-space}, the local boundedness of $D_aL$, and
\eqref{eq:regular-feedback} give
\[
 \left|L\big(Z_t,\alpha^n_t(Z_t)\big)
 -L\big(\widetilde Z_t,\alpha^n_t(\widetilde Z_t)\big)\right|
 \leq C\big(1+(T-t)^{-1/2}\big)\dm(Z_t,\widetilde Z_t).
\]
Thus \eqref{eq:regular-coupling}, the integrability of
$(T-t)^{-1/2}$, and the $d_1$-Lipschitz bounds for $F^n,G^n$ show that
the heat segment followed by $(\widetilde m,\alpha^n)$ has cost at most
$U^n(0,m_0)+C\tau$.  Lemma~\ref{lem:stability} then bounds its cost for
the original data by
\[
 U(0,m_0)+C\tau+2\delta_n.
\]
We choose $n$ so large that $2\delta_n\leq\eta$, and set
$(m,\alpha):=(\widetilde m,\alpha^n)$ on $[\tau,T]$.  This proves
\eqref{eq:regular-cost} and the bound on $\alpha$ in
\eqref{eq:regular-pair}.

Finally, the spliced diffusion has a bounded drift on $[0,T]$, with a
bound independent of $n,\tau,m_0$.  The heat-kernel upper bound gives
$\|\rho_t\|_{L^\infty}\leq C(1+t^{-1/2})$.  Positivity and the
classical regularity on $[\tau,T]$ follow from standard parabolic
regularity.  This completes
\eqref{eq:regular-pair}; \eqref{eq:chi-range} follows from
$\int\rho^3\leq\|\rho\|_{L^\infty}^2$ and Jensen's inequality.
\end{proof}

\subsection{Properties of the Lagrangian}

The following standard properties of the Lagrangian \eqref{eq:Legendre}
are used in Sections~\ref{sec:comparison} and~\ref{sec:one-dim}.

\begin{lemma}[Properties of the Lagrangian]\label{lem:lagrangian}
Under \ref{ass:H-convex}--\ref{ass:H-x}, the Lagrangian \eqref{eq:Legendre}
has the following properties.
\begin{enumerate}[label=\textup{(\roman*)},leftmargin=2.6em]
\item For each $(x,a)$, the supremum in \eqref{eq:Legendre} is attained at a
unique $p(x,a)$, characterized by $a=-D_pH(x,p(x,a))$, and
$|p(x,a)|\leq C(1+|a|)$.
\item $L\in C^1(\T^d\times\R^d)$ and $L(x,\cdot)$ is uniformly convex,
with $D_aL(x,a)=-p(x,a)$.  Moreover $H(x,p)=\sup_a\{-L(x,a)-a\cdot p\}$, with the
supremum attained at $a=-D_pH(x,p)$.
\item $D_xL(x,a)=-D_xH\big(x,p(x,a)\big)$, and
$|D_xL(x,a)|\leq C(1+|a|)$ for all $(x,a)$.
\item $|D_pH(x,p)|\leq C(1+|p|)$ for all $(x,p)$.
\end{enumerate}
\end{lemma}

\begin{proof}
(i) The map $p\mapsto-H(x,p)-a\cdot p$ is strictly concave by
\ref{ass:H-convex} and tends to $-\infty$ as $|p|\to\infty$ (since
$H(x,p)\geq H(x,0)+D_pH(x,0)\cdot p+\frac{1}{2C_0}|p|^2$), so the supremum is
attained at a unique point $p(x,a)$, characterized by the first-order
condition $a=-D_pH(x,p(x,a))$.  Set
$\kappa_0:=\sup_{x\in\T^d}|D_pH(x,0)|$, which is finite since $D_pH$ is
continuous and $\T^d$ is compact.  Strong monotonicity of $D_pH(x,\cdot)$
(again \ref{ass:H-convex}) gives
$|a|=|D_pH(x,p(x,a))|\geq C_0^{-1}|p(x,a)|-\kappa_0$, which yields the bound.

(ii) By the implicit function theorem ($D^2_{pp}H$ is invertible), $p(\cdot)$
is $C^1$, hence so is
$L(x,a)=-H(x,p(x,a))-a\cdot p(x,a)$.  The envelope theorem gives
$D_aL(x,a)=-p(x,a)$, and differentiating the relation $a=-D_pH(x,p(x,a))$
yields $D^2_{aa}L=(D^2_{pp}H)^{-1}(x,p(x,a))$, so that
$C_0^{-1}I_d\leq D^2_{aa}L\leq C_0I_d$.  The formula for $H$ holds
because $H(x,\cdot)$ is convex and finite: writing $L(x,a)=H^*(x,-a)$
with $H^*$ the convex conjugate in $p$, it reads
$H^{**}(x,p)=H(x,p)$, the equality of a finite convex function with
its double conjugate.  The supremum is attained at $a=-D_pH(x,p)$ by
the smoothness of $H$.

(iii) The formula is the envelope theorem in $x$, and the bound
follows from it by (i) and \ref{ass:H-x}.

(iv) is immediate from $D^2_{pp}H\leq C_0I_d$ and the definition of
$\kappa_0$.
\end{proof}

\subsection{Heat-kernel, cutoff, empirical-measure, and quantization
estimates}

\begin{lemma}[Heat-kernel bounds]\label{lem:heat-kernel}
For $0<\tau\leq T+1$,
\[
 \|\nabla p_\tau\|_{L^1(\T^d)}\leq C\tau^{-1/2},
 \qquad
 \|p_\tau\|_{L^2(\T^d)}^2\leq C\big(1+\tau^{-d/2}\big),
 \qquad
 \int_{\T^d}\dm(u,0)\,p_\tau(u)\,du\leq C\sqrt\tau .
\]
\end{lemma}

\begin{proof}
Let $g_\tau$ be the Gaussian kernel on $\R^d$,
$g_\tau(x)=(4\pi\tau)^{-d/2}e^{-|x|^2/4\tau}$, so that
$p_\tau(x)=\sum_{k\in\Z^d}g_\tau(x+k)$.  By scaling,
$\|\nabla p_\tau\|_{L^1(\T^d)}\leq\|\nabla g_\tau\|_{L^1(\R^d)}
=C\tau^{-1/2}$.

By Parseval's identity and comparison with a Gaussian integral,
\[
 \|p_\tau\|_{L^2}^2=\sum_{k\in\Z^d}e^{-8\pi^2|k|^2\tau}
 \leq1+C\int_{\R^d}e^{-4\pi^2|\xi|^2\tau}\,d\xi
 =1+C\tau^{-d/2},
\]
where the comparison uses $\tau\leq T+1$.

Finally, since $\dm(u,0)\leq|u+k|$ for every $k\in\Z^d$,
\[
 \int_{\T^d}\dm(u,0)\,p_\tau(u)\,du
 \leq\int_{\R^d}|x|\,g_\tau(x)\,dx=C\sqrt\tau ,
\]
by scaling.
\end{proof}

The next lemma provides the smooth frequency cutoff used in
Section~\ref{sec:one-dim}.  It is stated on the circle, where it is
needed.

\begin{lemma}[Smooth frequency cutoff]\label{lem:cutoff}
Fix a smooth even $\chi:\R\to[0,1]$ with $\chi=1$ on $[-1/2,1/2]$ and
$\chi=0$ outside $(-1,1)$, and for $K\geq1$ define the Fourier
multiplier $S_K$ on finite measures on $\T$ by
$\widehat{S_K\sigma}(n)=\chi(n/K)\,\widehat\sigma(n)$.  Then, for
$0<u\leq T$:
\begin{enumerate}[label=\textup{(\roman*)},leftmargin=2.6em]
\item $\|S_K\sigma\|_{L^1(\T)}\leq C\,|\sigma|(\T)$ for every
finite signed measure $\sigma$, with $C$ independent of $K$;
\item $\|(I-S_K)p_u\|_{L^1(\T)}\leq Ce^{-cK^2u}$ and
$\int_0^T\|(I-S_K)p_u\|_{L^2(\T)}^2\,du\leq CK^{-1}$.
\end{enumerate}
\end{lemma}

\begin{proof}
(i) $S_K\sigma=k_K\star\sigma$ with
\[
 k_K(x)=\sum_{n\in\Z}\chi(n/K)\,e^{2\pi inx}
 =\sum_{m\in\Z}K\,\widehat\chi\big(K(x+m)\big)
\]
by the Poisson summation formula, where
$\widehat\chi(\xi)=\int_\R\chi(\eta)e^{-2\pi i\eta\xi}d\eta$ is a
Schwartz function.  Hence
$\|k_K\|_{L^1(\T)}\leq\|K\widehat\chi(K\cdot)\|_{L^1(\R)}
=\|\widehat\chi\|_{L^1(\R)}$, and Young's inequality gives (i).

(ii) If $K^2u\leq1$, the claim follows from (i) and $\|p_u\|_{L^1}=1$
upon adjusting $c$, so assume $a:=4\pi^2K^2u\geq1$.  The function
$m(\eta):=(1-\chi(\eta))e^{-a\eta^2}$ is Schwartz, and Poisson summation
as in (i) gives
\[
 (I-S_K)p_u(x)=\sum_{n\in\Z}m(n/K)\,e^{2\pi inx},
 \qquad
 \big\|(I-S_K)p_u\big\|_{L^1(\T)}\leq\|\widehat m\|_{L^1(\R)} .
\]
On the support of $1-\chi$ one has $\eta^2\geq1/4$, so
$m=e^{-a/8}n_a$ with $n_a(\eta):=(1-\chi(\eta))e^{-a(\eta^2-1/8)}$ and
$\eta^2-1/8\geq\eta^2/2$ there.  Since $a\geq1$,
\[
 \|n_a\|_{L^2(\R)}^2\leq\int_\R e^{-a\eta^2}d\eta\leq C,
 \qquad
 \|n_a'\|_{L^2(\R)}^2
 \leq C\int_\R\big(1+a^2\eta^2\big)e^{-a\eta^2}d\eta\leq Ca,
\]
so, by $\|\widehat g\|_{L^1(\R)}\leq C\|g\|_{H^1(\R)}$
(Cauchy--Schwarz against $(1+\xi^2)^{-1}$),
\[
 \|\widehat m\|_{L^1(\R)}
 \leq Ce^{-a/8}\|n_a\|_{H^1(\R)}
 \leq Ce^{-a/8}\,a^{1/2}\leq Ce^{-a/16},
\]
which is the first bound in (ii).  The second follows by Parseval's
identity: since $1-\chi(n/K)$ vanishes for $|n|\leq K/2$,
\[
 \int_0^T\big\|(I-S_K)p_u\big\|_{L^2}^2\,du
 \leq\sum_{|n|\geq K/2}\int_0^\infty e^{-8\pi^2n^2u}\,du
 =\sum_{|n|\geq K/2}\frac{1}{8\pi^2n^2}
 \leq CK^{-1}. \qedhere
\]
\end{proof}

The next lemma is the empirical-measure estimate used for the easy
inequality with common noise (Section~\ref{sec:common-noise}).  The
non-identically distributed matching estimate needed here is precisely
\cite[Thm.~3 and Eq.~(12)]{BobkovLedoux2021}.  A ghost sample and Jensen's
inequality then give the estimate relative to the averaged law.

\begin{lemma}[Conditional empirical estimate]\label{lem:empirical}
Let $\mathcal G$ be a $\sigma$-algebra and let $X^1,\dots,X^N$ be
$\T^d$-valued random variables that are conditionally independent given
$\mathcal G$, with conditional laws $m^1,\dots,m^N$, and set
$\bar m:=\frac1N\sum_{i=1}^Nm^i$.  Then, almost surely,
\[
 \E\bigg[d_1\Big(\frac1N\sum_{i=1}^N\delta_{X^i},\,\bar m\Big)\,\Big|\,
 \mathcal G\bigg]\leq C\,r_{N,d} .
\]
\end{lemma}

\begin{proof}
On an enlargement of the probability space, let $Y^1,\dots,Y^N$ be
random variables that are conditionally independent given
$\mathcal G\vee\sigma(X^1,\dots,X^N)$, with conditional laws
$\mathcal L(Y^i\mid\mathcal G\vee\sigma(X^1,\dots,X^N))=m^i$.
Under the conditional law of $(X^1,Y^1,\dots,X^N,Y^N)$ given
$\mathcal G$, the couples $(X^i,Y^i)$ are then
independent, and $X^i$, $Y^i$ share the law $m^i$.  Identifying $\T^d$
with $[0,1)^d$, the torus distance is dominated by the Euclidean
distance, so every coupling on the cube induces a coupling on the torus
of no larger cost.  Applying \cite[Thm.~3 and
Eq.~(12)]{BobkovLedoux2021}, whose constants depend only on the
dimension, to that conditional law gives
\[
 \E\bigg[d_1\Big(\frac1N\sum_{i=1}^N\delta_{X^i},\,
 \frac1N\sum_{i=1}^N\delta_{Y^i}\Big)\,\Big|\,\mathcal G\bigg]
 \leq C\,r_{N,d} .
\]
Since
$\mathcal L(Y^i\mid\mathcal G\vee\sigma(X^1,\dots,X^N))=m^i$,
\[
 \E\bigg[\frac1N\sum_{i=1}^N\delta_{Y^i}\,\Big|\,
 \mathcal G\vee\sigma(X^1,\dots,X^N)\bigg]=\bar m,
\]
and $\nu\mapsto d_1(\mu,\nu)$ is a supremum of affine functionals of
$\nu$, so Jensen's inequality gives
\[
 d_1\Big(\frac1N\sum_{i=1}^N\delta_{X^i},\,\bar m\Big)
 \leq\E\bigg[d_1\Big(\frac1N\sum_{i=1}^N\delta_{X^i},\,
 \frac1N\sum_{i=1}^N\delta_{Y^i}\Big)\,\Big|\,
 \mathcal G\vee\sigma(X^1,\dots,X^N)\bigg].
\]
Taking $\E[\,\cdot\mid\mathcal G]$ of both sides completes the proof.
\end{proof}

The last lemma is the elementary quantization bound used in the proof
of the optimality clause of Theorem~\ref{thm:main}
(Section~\ref{sec:comparison}).

\begin{lemma}[Quantization lower bound]\label{lem:quantization}
Let $\nu\in\Pp(\T^d)$ satisfy $\nu\geq\kappa\lambda$ for some
$\kappa>0$, where $\lambda$ is the uniform measure.  There exists
$c=c(d,\kappa)>0$ such that every $\mu\in\Pp(\T^d)$ supported on at
most $N$ points satisfies $d_1(\mu,\nu)\geq c\,N^{-1/d}$.
\end{lemma}

\begin{proof}
Let $S:=\operatorname{supp}\mu$ and
$\operatorname{dist}(x,S):=\min_{y\in S}\dm(x,y)$, so that $\#S\leq N$.
Since $\lambda(B(x,r))\leq C_dr^d$ for every $x$ and $r$, the set
$\{\operatorname{dist}(\cdot,S)\leq r\}$ has measure at most
$NC_dr^d$, hence at most $\frac12$ for
$r\leq R_N:=(2C_dN)^{-1/d}$; enlarging $C_d$ if necessary, also
$R_N\leq\operatorname{diam}(\T^d)$.  The layer-cake formula then gives
\[
 \int_{\T^d}\operatorname{dist}(x,S)\,\lambda(dx)
 =\int_0^{\operatorname{diam}(\T^d)}
 \lambda\big(\operatorname{dist}(\cdot,S)>r\big)\,dr
 \geq\frac{R_N}2 .
\]
Under any coupling of $\nu$ and $\mu$, the second coordinate lies in
$S$, so the transport cost is at least
$\int_{\T^d}\operatorname{dist}(x,S)\,\nu(dx)
\geq\kappa\int_{\T^d}\operatorname{dist}(x,S)\,\lambda(dx)
\geq\frac\kappa2R_N$.
Taking the infimum over couplings proves the claim with
$c=\frac\kappa2(2C_d)^{-1/d}$.
\end{proof}

\section{Necessity of the logarithm in two dimensions}\label{app:two-dim}

Throughout, $d=2$, $\lambda$ denotes
the uniform measure on $\T^2$, and we fix the data
\begin{equation}\label{eq:two-dim-data}
 H(x,p)=\tfrac12|p|^2,
 \qquad
 F(m)=d_1(m,\lambda),
 \qquad
 G=0,
\end{equation}
which satisfy Assumption~\ref{ass:standing} and have Lagrangian
$L(x,a)=\tfrac12|a|^2$.  For $N\geq1$, $\bx\in(\T^2)^N$, and an
admissible control $\balpha\in\mathcal A^N_0$ on $[0,T]$, define the
action and the discrepancy
\begin{equation}\label{eq:action-discrepancy}
 \mathsf A:=\E\int_0^T\frac1N\sum_{i=1}^N|\alpha^i_t|^2\,dt,
 \qquad
 \mathsf D:=\int_0^T\E\,d_1\big(\mu^N_t,\lambda\big)\,dt,
\end{equation}
so that $J_N(0,\bx;\balpha)=\tfrac12\mathsf A+\mathsf D$.

The proof relies on a trade between the noise and the
control: on every window of length $\tau$, the
idiosyncratic noise regenerates fluctuations of the empirical measure
at the optimal matching scale $\sqrt{\log(N\tau)/N}$ (the fresh-noise
estimate, Lemma~\ref{lem:fresh-noise}), while, by a
pathwise coupling, suppressing a fluctuation within the window costs
an action superlinear in the suppressed distance.
Proposition~\ref{prop:two-dim-lower} converts this trade into a lower
bound on the total cost that holds for every initial
configuration. The optimality clause is deduced from
this bound in the proof of Theorem~\ref{thm:main}
(Section~\ref{sec:comparison}).

We begin with two auxiliary estimates.  For $\varepsilon>0$, define
the smoothed Riesz kernel $K_\varepsilon:\T^2\to\R^2$ by its Fourier
coefficients
\begin{equation}\label{eq:riesz-kernel}
 \widehat{K_\varepsilon}(k)
 :=\frac{2\pi ik}{4\pi^2|k|^2}\,e^{-4\pi^2|k|^2\varepsilon},
 \quad k\neq0,
 \qquad
 \widehat{K_\varepsilon}(0):=0,
\end{equation}
so that
$K_\varepsilon\star\sigma=\nabla(-\Delta)^{-1}P_\varepsilon\sigma$ for
every finite signed measure $\sigma$ on $\T^2$ with $\sigma(\T^2)=0$.

\begin{lemma}[Kernel bounds]\label{lem:riesz-kernel}
For $0<\varepsilon\leq1$ and $0<\tau\leq T$,
\[
 \|K_\varepsilon\|_{L^2(\T^2)}^2\leq C\log\big(1+\varepsilon^{-1}\big),
 \qquad
 \|K_\varepsilon\|_{L^4(\T^2)}^4\leq C\varepsilon^{-1},
 \qquad
 \|\nabla p_\tau\|_{L^\infty(\T^2)}\leq C\tau^{-3/2}.
\]
\end{lemma}

\begin{proof}
By Parseval's identity,
\[
 \|K_\varepsilon\|_{L^2}^2
 =\sum_{k\neq0}\frac{e^{-8\pi^2|k|^2\varepsilon}}{4\pi^2|k|^2}
 \leq C\sum_{0<|k|\leq\varepsilon^{-1/2}}\frac1{|k|^2}
 +C\varepsilon\sum_{|k|>\varepsilon^{-1/2}}e^{-8\pi^2|k|^2\varepsilon}
 \leq C\log\big(1+\varepsilon^{-1}\big),
\]
the first sum by the standard lattice count and the second by
comparison with a Gaussian integral.

For the $L^4$ bound, integrating the coefficients
\eqref{eq:riesz-kernel} in the smoothing time gives the heat
representation
$K_\varepsilon=\int_\varepsilon^\infty\nabla p_u\,du$.  For
$0<u\leq1$, the periodization of the Gaussian kernel gives
$|\nabla p_u(x)|\leq Cu^{-3/2}e^{-\dm(x,0)^2/(8u)}$, while for $u>1$
every nonzero mode has decayed and
$|\nabla p_u(x)|\leq Ce^{-cu}$.  Writing $r:=\dm(x,0)$,
\[
 |K_\varepsilon(x)|
 \leq C\int_\varepsilon^1u^{-3/2}e^{-r^2/(8u)}\,du+C
 \leq C\min\big(\varepsilon^{-1/2},\,r^{-1}\big)+C
 \leq\frac{C}{r+\sqrt\varepsilon},
\]
using $\int_0^\infty u^{-3/2}e^{-r^2/(8u)}\,du=C/r$.  Consequently
\[
 \int_{\T^2}|K_\varepsilon|^4
 \leq C\int_0^1\frac{r\,dr}{(r+\sqrt\varepsilon)^4}
 \leq C\varepsilon^{-1}.
\]

Finally, as in the proof of the gradient bound in
Lemma~\ref{lem:heat-kernel},
$|\nabla p_\tau(x)|\leq\sum_{k\in\Z^2}|\nabla g_\tau(x+k)|$ with
$g_\tau$ the Gaussian kernel, and
$|\nabla g_\tau(z)|=\frac{|z|}{2\tau}g_\tau(z)\leq C\tau^{-3/2}
e^{-|z|^2/(8\tau)}$.  For $\tau\leq T$ the sum is at most
$C\tau^{-3/2}$.
\end{proof}

\begin{lemma}[Lipschitz truncation]\label{lem:lipschitz-truncation}
There exists $C$ such that for every $h\in C^\infty(\T^2)$ and every
$R>0$ there exists an $R$-Lipschitz function $h_R:\T^2\to\R$ with
\begin{equation}\label{eq:truncation-bound}
 \int_{\T^2}\big|\nabla(h-h_R)\big|^2
 \leq\frac{C}{R^2}\int_{\T^2}|\nabla h|^4 .
\end{equation}
\end{lemma}

\begin{proof}
The Lusin approximation of Sobolev functions on the torus
\cite[Lem.~5.1]{AmbrosioStraTrevisan2019} (the Lipschitz truncation
of \cite{AcerbiFusco1984}), applied with $p=4$, provides an
$R$-Lipschitz function $h_R:\T^2\to\R$ whose exceptional set
$E:=\{h\neq h_R\}$ satisfies
$|E|\leq CR^{-4}\int_{\T^2}|\nabla h|^4$.  Since $\nabla(h-h_R)=0$
almost everywhere on $E^c$,
\[
 \int_{\T^2}|\nabla(h-h_R)|^2
 \leq2\int_E|\nabla h|^2+2R^2\,|E|
 \leq2\Big(\int_{\T^2}|\nabla h|^4\Big)^{1/2}|E|^{1/2}+2R^2\,|E| ,
\]
and inserting the bound on $|E|$ in both terms proves
\eqref{eq:truncation-bound}.
\end{proof}

The heart of the argument is the following fresh-noise estimate: two
independent copies of a cloud of $N$ particles, each diffused for a
time $\tau$, are separated in $d_1$ at the optimal matching rate.  It
is an optimal-matching lower bound \cite{AjtaiKomlosTusnady1984} for
independent, non-identically distributed samples, proved by adapting
the Fourier-analytic fourth-moment argument of
\cite[Sec.~7]{BobkovLedoux2021}.  The density hypothesis excludes
clustered centers.

\begin{lemma}[Fresh-noise estimate]\label{lem:fresh-noise}
Let $d=2$.  There exist a universal integer $N_0$ and, for every
$\varrho\geq1$, a constant $c_\varrho>0$ depending only on $\varrho$
and $T$, with the following property.  Let $N\geq N_0$,
$\by\in(\T^2)^N$, and $N^{-1/2}\leq\tau\leq T$ satisfy
$\|P_\tau m^N_{\by}\|_{L^\infty(\T^2)}\leq\varrho$, and let
$Z^1,\dots,Z^N,\widetilde Z^1,\dots,\widetilde Z^N$ be independent
random variables with $Z^i$ and $\widetilde Z^i$ of density
$p_\tau(\cdot-y^i)$.  Then
\begin{equation}\label{eq:fresh-noise}
 \E\,d_1\Big(\frac1N\sum_{i=1}^N\delta_{Z^i},\,
 \frac1N\sum_{i=1}^N\delta_{\widetilde Z^i}\Big)
 \geq c_\varrho\sqrt{\frac{\log(1+N\tau)}{N}}\,.
\end{equation}
\end{lemma}

\begin{proof}
Write $\mu$ and $\widetilde\mu$ for the two empirical measures,
$z:=\mu-\widetilde\mu$, and fix $\varepsilon:=N^{-1}\leq\tau$.  Let $h$
be the mean-zero solution of $-\Delta h=P_\varepsilon z$ on $\T^2$, so
that $\nabla h=K_\varepsilon\star z$, and set
\[
 A:=\int_{\T^2}|\nabla h|^2,
 \qquad
 B:=\int_{\T^2}|\nabla h|^4 .
\]

By Parseval's identity and
\eqref{eq:riesz-kernel},
$A=\sum_{k\neq0}e^{-8\pi^2|k|^2\varepsilon}\,|\widehat z(k)|^2
/(4\pi^2|k|^2)$.  For each $k\neq0$, the $N$ summands of
$\widehat z(k)=\frac1N\sum_i(e^{-2\pi ik\cdot Z^i}
-e^{-2\pi ik\cdot\widetilde Z^i})$ are independent with mean zero, and
each has second moment
$2(1-|\widehat{p_\tau}(k)|^2)=2(1-e^{-8\pi^2|k|^2\tau})$.  Hence,
exactly and independently of the centers,
\begin{equation*}
 \E A=\frac2N\sum_{k\neq0}
 \frac{e^{-8\pi^2|k|^2\varepsilon}
 \big(1-e^{-8\pi^2|k|^2\tau}\big)}{4\pi^2|k|^2}\,.
\end{equation*}
On the modes with $(8\pi^2\tau)^{-1}\leq|k|^2\leq(8\pi^2\varepsilon)^{-1}$
(a nonempty range once $N_0$ is large, since
$\tau/\varepsilon=N\tau\geq\sqrt N$), the two factors in the numerator
are at least $e^{-1}$ and $1-e^{-1}$, and the lattice count gives
\begin{equation}\label{eq:energy-lower}
 \E A\geq\frac cN\log\big(\tau/\varepsilon\big)
 \geq c_0\,\frac{\log(1+N\tau)}{N}=:a_0 .
\end{equation}

Write
$\nabla h=\frac1N\sum_i\xi_i$ with
$\xi_i:=K_\varepsilon(\cdot-Z^i)-K_\varepsilon(\cdot-\widetilde Z^i)$.
The fields $\xi_1,\dots,\xi_N$ are independent, and each has mean zero
because $Z^i$ and $\widetilde Z^i$ share their law.  Expanding the
fourth power and discarding the terms containing an isolated factor,
\[
 \E\Big|\sum_{i=1}^N\xi_i(x)\Big|^4
 \leq\sum_{i=1}^N\E|\xi_i(x)|^4
 +3\Big(\sum_{i=1}^N\E|\xi_i(x)|^2\Big)^2 .
\]
By Fubini's theorem and the translation invariance of the Lebesgue
measure, $\int_{\T^2}\E|\xi_i|^4\leq16\|K_\varepsilon\|_{L^4}^4$, so
the first term integrates to at most
$CN\varepsilon^{-1}=CN^2$ by the $L^4$ bound of
Lemma~\ref{lem:riesz-kernel}.  For the second,
\[
 \sum_{i=1}^N\E|\xi_i(x)|^2
 \leq4\sum_{i=1}^N\big(|K_\varepsilon|^2\star p_\tau(\cdot-y^i)\big)(x)
 =4N\big(|K_\varepsilon|^2\star P_\tau m^N_{\by}\big)(x)
 \leq4N\varrho\,\|K_\varepsilon\|_{L^2}^2,
\]
which is at most $C\varrho N\log(1+N)$ by the $L^2$ bound of
Lemma~\ref{lem:riesz-kernel}.  Since $\tau\geq N^{-1/2}$ gives
$\log(1+N)\leq C\log(1+N\tau)$, we conclude that
\begin{equation}\label{eq:fourth-moment}
 \E B\leq\frac{CN^2+C\varrho^2N^2\log^2(1+N)}{N^4}
 \leq C_1\varrho^2\,a_0^2 .
\end{equation}

Let $R>0$, to be chosen, and let $h_R$ be the
$R$-Lipschitz truncation of Lemma~\ref{lem:lipschitz-truncation}.  The
function $P_\varepsilon h_R$ is again $R$-Lipschitz, and, since
$P_\varepsilon z$ has the density $-\Delta h$,
\[
 R\,d_1(\mu,\widetilde\mu)
 \geq\int_{\T^2}P_\varepsilon h_R\,dz
 =\int_{\T^2}h_R\,(-\Delta h)
 =\int_{\T^2}\nabla h_R\cdot\nabla h
 \geq A-\Big(\frac{CB}{R^2}\Big)^{1/2}A^{1/2},
\]
by \eqref{eq:truncation-bound} and the Cauchy--Schwarz inequality.
Taking expectations and using
$\E[B^{1/2}A^{1/2}]\leq(\E B)^{1/2}(\E A)^{1/2}$ together with
\eqref{eq:energy-lower} and \eqref{eq:fourth-moment} (note
$a_0\leq\E A$),
\[
 R\,\E\,d_1(\mu,\widetilde\mu)
 \geq\E A-\frac{C_2\varrho}{R}\,(\E A)^{3/2}.
\]
The choice $R:=2C_2\varrho\sqrt{\E A}$ gives
$\E\,d_1(\mu,\widetilde\mu)\geq\sqrt{\E A}/(4C_2\varrho)
\geq c_\varrho\sqrt{a_0}$, which is \eqref{eq:fresh-noise}.
\end{proof}

\begin{proposition}[Dynamical lower bound]\label{prop:two-dim-lower}
Let $d=2$ and let the data be \eqref{eq:two-dim-data}.  There exist
$c>0$ and $N_0$, depending only on $T$, such that
\begin{equation*}
 J_N(0,\bx;\balpha)\geq c\,r_{N,2}
 \qquad\text{for all }N\geq N_0,\
 \bx\in(\T^2)^N,\ \balpha\in\mathcal A^N_0 .
\end{equation*}
\end{proposition}

\begin{proof}
Abbreviate $r_N:=r_{N,2}$ and recall from
\eqref{eq:action-discrepancy} that
$J_N(0,\bx;\balpha)=\tfrac12\mathsf A+\mathsf D$.  If
$\mathsf A\geq r_N$ there is nothing to prove, so assume
$\mathsf A\leq r_N$.  Fix
\[
 \tau:=\kappa\sqrt{r_N},
\]
with $\kappa=\kappa(T)\in(0,1]$ chosen at the end.  Enlarging $N_0$, we
may assume $N^{-1/2}\leq\tau\leq T/2$ and
$\log(1+N\tau)\geq\tfrac13\log(1+N)$ (as
$N\tau=\kappa N^{3/4}\log^{1/4}(1+N)$).

For $s\in[0,T-\tau]$, define the freed
particles and their empirical measure by
\[
 \bar X^i_{s+\tau}:=X^i_s+\sqrt2\big(W^i_{s+\tau}-W^i_s\big),
 \qquad
 \bar\mu^N_{s+\tau}:=\frac1N\sum_{i=1}^N\delta_{\bar X^i_{s+\tau}} .
\]
Since $\dm(X^i_{s+\tau},\bar X^i_{s+\tau})\leq\int_s^{s+\tau}
|\alpha^i_t|\,dt$, the Cauchy--Schwarz inequality, first in time and
then in the particle and sample averages, gives
\begin{equation}\label{eq:freezing}
 \E\,d_1\big(\mu^N_{s+\tau},\bar\mu^N_{s+\tau}\big)
 \leq\sqrt{\tau\,\mathsf A_s},
 \qquad
 \mathsf A_s:=\E\int_s^{s+\tau}\frac1N\sum_{i=1}^N|\alpha^i_t|^2\,dt .
\end{equation}

The Lipschitz constant of $p_\tau(x-\cdot)$
is at most $C\tau^{-3/2}$, by the last bound of
Lemma~\ref{lem:riesz-kernel}, so
testing it against $\mu^N_s-\lambda$ in the duality
\eqref{eq:d1-duality} gives
\[
 \big\|P_\tau\mu^N_s-1\big\|_{L^\infty(\T^2)}
 \leq C\tau^{-3/2}\,d_1\big(\mu^N_s,\lambda\big).
\]
Hence, on the event
$\Gamma_s:=\{d_1(\mu^N_s,\lambda)\leq\tau^{3/2}\}$,
\begin{equation}\label{eq:density-good}
 \big\|P_\tau m^N_{\bX_s}\big\|_{L^\infty(\T^2)}\leq1+C=:\varrho .
\end{equation}

On a product extension carrying Brownian motions
$\widetilde W^1,\dots,\widetilde W^N$ independent of $\F_T$, set
$\widetilde X^i_{s+\tau}:=X^i_s
+\sqrt2(\widetilde W^i_{s+\tau}-\widetilde W^i_s)$ and let
$\widetilde\mu^N_{s+\tau}$ be the associated empirical measure.
The increments generating the $2N$ variables
$\bar X^i_{s+\tau},\widetilde X^i_{s+\tau}$ are independent of $\F_s$,
so the conditional law of these variables given $\F_s$ is the product
of the densities $p_\tau(\cdot-X^i_s)$, and
$\E[d_1(\bar\mu^N_{s+\tau},\widetilde\mu^N_{s+\tau})\mid\F_s]$ is the
expectation in \eqref{eq:fresh-noise} evaluated at the centers
$\by=\bX_s$.  On $\Gamma_s$ these centers satisfy the density
hypothesis of Lemma~\ref{lem:fresh-noise}, by \eqref{eq:density-good},
so, almost surely on $\Gamma_s$,
\[
 \E\big[d_1\big(\bar\mu^N_{s+\tau},\widetilde\mu^N_{s+\tau}\big)
 \,\big|\,\F_s\big]
 \geq c_\varrho\sqrt{\frac{\log(1+N\tau)}{N}}=:\rho_N .
\]
Since $\bar\mu^N_{s+\tau}$ and $\widetilde\mu^N_{s+\tau}$ have the
same conditional law, the triangle inequality upgrades this to
\[
 \E\big[d_1\big(\bar\mu^N_{s+\tau},\lambda\big)\,\big|\,\F_s\big]
 \geq\rho_N/2
 \qquad\text{on }\Gamma_s .
\]
Combining with \eqref{eq:freezing}, then applying
Markov's inequality to $\Gamma_s^c$ and the bound
$\rho_N\leq Cr_N$,
\begin{align*}
 \E\,d_1\big(\mu^N_{s+\tau},\lambda\big)
 &\geq\E\big[\1_{\Gamma_s}\,
 d_1\big(\bar\mu^N_{s+\tau},\lambda\big)\big]
 -\sqrt{\tau\,\mathsf A_s}
 \\
 &\geq\frac{\rho_N}2
 -C\,r_N\,\tau^{-3/2}\,\E\,d_1\big(\mu^N_s,\lambda\big)
 -\sqrt{\tau\,\mathsf A_s}\,.
\end{align*}

We integrate over $s\in[0,T-\tau]$: the two
discrepancy integrals are at most $\mathsf D$, while Fubini's theorem
gives $\int_0^{T-\tau}\mathsf A_s\,ds\leq\tau\mathsf A$ and, with the
Cauchy--Schwarz inequality,
$\int_0^{T-\tau}\sqrt{\tau\mathsf A_s}\,ds\leq\tau\sqrt{T\mathsf A}$.
Hence
\begin{equation}\label{eq:window-inequality}
 (T-\tau)\,\frac{\rho_N}2
 \leq\big(1+Cr_N\tau^{-3/2}\big)\,\mathsf D
 +\tau\sqrt{T\mathsf A}\,.
\end{equation}

By the choice of $\tau$ and
\eqref{eq:density-good}, $\rho_N\geq c_1r_N$ with $c_1=c_1(T)>0$,
while $r_N\tau^{-3/2}=\kappa^{-3/2}r_N^{1/4}$ vanishes as
$N\to\infty$, so, enlarging $N_0$, the prefactor in
\eqref{eq:window-inequality} is at most $2$.  Moreover
$\mathsf A\leq r_N$ gives
$\tau\sqrt{T\mathsf A}\leq\kappa\sqrt T\,r_N$, and $T-\tau\geq T/2$.
Thus \eqref{eq:window-inequality} yields
\[
 \frac{c_1T}{4}\,r_N\leq2\,\mathsf D+\kappa\sqrt T\,r_N,
\]
and the choice $\kappa:=1\wedge\big(c_1\sqrt T/8\big)$ leaves
$\mathsf D\geq(c_1T/16)\,r_N$.  In both cases
$J_N(0,\bx;\balpha)\geq\big(\tfrac12\wedge\tfrac{c_1T}{16}\big)r_N$.
\end{proof}

\addtocontents{toc}{\protect\setcounter{tocdepth}{1}}
\bibliographystyle{amsalpha}
\bibliography{references}

\end{document}